\documentclass[a4paper,10pt]{extarticle}
\usepackage[margin=2cm]{geometry}
\usepackage[pagebackref=false,linktoc=page,colorlinks=true,linkcolor=blue,citecolor=blue]{hyperref}
\usepackage{algorithmicx}
\usepackage{algorithm}
\usepackage{algpseudocode}
\usepackage{setspace}
\usepackage{epsfig}
\usepackage{psfrag}
\usepackage[retainorgcmds]{IEEEtrantools}
\usepackage{xcolor}
\usepackage{authblk}
\usepackage{latexsym}
\usepackage{amsmath,amssymb,amsfonts,amsthm}
\usepackage{mathrsfs}
\usepackage{bm}
\usepackage{bbm}
\usepackage{breqn}
\usepackage[T1]{fontenc}
\usepackage{pstool, psfrag} 
\usepackage{mathrsfs}

\usepackage{graphicx}
\usepackage{natbib}
\usepackage{pstool,psfrag}
\usepackage{enumitem}
\usepackage[all]{xy}
\usepackage[section]{placeins}

\newtheorem{theorem}{Theorem}
\newtheorem{lemma}{Lemma}
\newtheorem{example}{Example}
\newtheorem{corollary}{Corollary}
\newtheorem{proposition}{Proposition}
\newtheorem{remark}{Remark}
\makeatletter \newcommand{\PR}{\ensuremath{\mathbb{P}}}
\newcommand{\sm}{\ensuremath{\setminus}}
\newcommand\ddfrac[2]{\frac{\displaystyle #1}{\displaystyle #2}}
\newcommand{\RR}{\ensuremath{\mathrm{\mathbb{R}}}}

\newcommand{\EE}{\ensuremath{\mathbb{E}}}
\newcommand{\AAA}{\ensuremath{A}_0}
\renewcommand{\AA}{\ensuremath{A}_1}
\newcommand{\AB}{\ensuremath{A}_2}

\newcommand{\BA}{\ensuremath{B}_1}
\newcommand{\BB}{\ensuremath{B}_2}

\newcommand{\SI}{\ensuremath{\mathbf{S}}}
\newcommand{\NN}{\ensuremath{\mathbb{N}}}
\newcommand{\ZZ}{\ensuremath{\mathbb{Z}}}

\newcommand{\PP}{\ensuremath{\mathbb{P}}}

\newcommand{\eps}{\ensuremath{\varepsilon}}
\newcommand{\borel}{\ensuremath{\mathscr{B}}}
\newcommand{\level}{\ensuremath{v}}
\newcommand{\res}{\ensuremath{z}}
\newcommand{\Res}{\ensuremath{Z}}

\newcommand{\wk}{\ensuremath{\,\overset {\mathrm {w} }{\longrightarrow
    }\,}}
\newcommand{\cind}{\ensuremath{\,\overset {d }{\longrightarrow
    }\,}}
\newcommand{\cinp}{\ensuremath{\,\overset {p }{\longrightarrow
    }\,}}

\newcommand{\scale}{b}
\newcommand{\loc}{a}

\newcommand{\tpo}{\ensuremath{{t}}}

\newcommand{\CommaBin}{\mathbin{\raisebox{0.5ex}{,}}}

\newcommand{\kkminone}{1\,:\,k-1}
\newcommand{\kplusonek}{1\,:\,k}

\newcommand{\afunD}{a}
\newcommand{\bfunD}{b}
\newcommand{\kk}{0\,:\,k-1}
\newcommand{\tk}{t-k\,:\,t-1}

\newcommand{\skb}{-s+1\,:\,-s+k}
\newcommand{\tplusonek}{t-k+1\,:\,t}
\newcommand{\kplusonekplusone}{0\,:\,k}
\newcommand{\jk}{j-k\,:\,j-1}
\newcommand{\tplusonetplusone}{0\,:\,t}

\title{\textbf{Hidden tail chains and recurrence equations
    for dependence parameters associated with extremes of higher-order
    Markov chains}} \author[*]{\textbf{Ioannis
    Papastathopoulos}}
\author[*]{\textbf{Adrian Casey}}
\author[$\dagger$]{\textbf{Jonathan A. Tawn}}
\affil[*]{\small School of Mathematics and Maxwell Institute,
  University of Edinburgh, Edinburgh, EH9 3FD, United Kingdom}
\affil[$\dagger$]{\small Department of Mathematics and Statistics,
  Lancaster University, Lancaster, LA1 4YF, United Kingdom}

\affil[$$]{\small i.papastathopoulos@ed.ac.uk
  $\quad$ acasey2@exseed.ed.ac.uk $
  \quad$ j.tawn@lancaster.ac.uk } 
\date{}
\begin{document}
\maketitle
\begin{abstract}
  We derive some key extremal features for
  $k$th order Markov chains that can be used to understand how the
  process moves between an extreme state and the body of the process.\
  The chains are studied given that there is an exceedance of a
  threshold, as the threshold tends to the upper endpoint of the
  distribution.\ Unlike previous studies with
  $k>1$, we consider processes where standard limit theory describes
  each extreme event as a single observation without any information
  about the transition to and from the body of the distribution.\ Our
  work uses different asymptotic theory which results in
  non-degenerate limit laws for such processes.\ We study the extremal
  properties of the initial distribution and the transition
  probability kernel of the Markov chain under weak assumptions for
  broad classes of extremal dependence structures that cover both
  asymptotically dependent and asymptotically independent Markov
  chains.\ For chains with
  $k>1$, the transition of the chain away from the exceedance involves
  novel functions of the
  $k$ previous states, in comparison to just the single value, when
  $k=1$. This leads to an increase in the complexity of determining
  the form of this class of functions, their properties and the method
  of their derivation in applications. We find that it is possible to
  derive an affine normalization, dependent on the threshold excess,
  such that non-degenerate limiting behaviour of the process is
  assured for all lags.\ These normalization functions have an
  attractive structure that has parallels to the Yule-Walker
  equations.\ Furthermore, the limiting process is always linear in
  the innovations.\ We illustrate the results with the study of
  $k$th order stationary Markov chains with exponential margins based
  on widely studied families of copula dependence structures.
\end{abstract}
\frenchspacing
\section{Introduction}
The extreme value theory of sequences of independent and identically
distributed (i.i.d.) random variables has often been generalised to
include the situation where the variables are no longer independent,
as in the monograph of \cite{leadling83} where for stationary
processes the focus is on long-range dependence conditions and local
clustering of extremes as measured by the extremal index.\ Among the
most useful stochastic processes are positive recurrent Markov chains,
with a continuous state space, which provide the backbone of a broad
range of statistical models and meet the required long-range
dependence conditions \citep{obri87,root88}.\ Such models have
attracted considerable interest in the analysis of extremes of
stochastic processes, by considering the behaviour of the process when
it is extreme, that is, when it exceeds a high threshold.\ \cite{root88}
showed that, under certain circumstances, the times of extreme events
of stationary Markov chains that exceed a high threshold converge to a
homogeneous Poisson process and that the limiting characteristics of
the values within an extreme event, including the extremal index, can
be derived as the threshold converges to the upper endpoint of the
marginal distribution.

Although powerful, this approach only reveals the behaviour of the
chain whilst it remains at the same level of marginal extremity as the
threshold, and therefore it is only informative about clustering for a
subset of processes, termed asymptotically dependent processes which
are defined below.\ For any Markov process which does not exhibit any
asymptotic dependence, known as asymptotically independent processes,
e.g., for any Gaussian Markov process, this limit theory describes
each extreme event as a single observation.\ Motivated by this
observation, in this paper we seek to understand better the behaviour
of a Markov chain within an extreme event under less restrictive
conditions through using a more refined limit theory.\ Our analysis
allows us to characterise the event as it moves between an extreme
state and the body of the distribution.\ In the case of first-order
Markov chains, \cite{papaetal17} treat both asymptotically dependent
and asymptotically independent chains in a unified theory.\ The focus
of this paper is similar, but this time on higher-order Markov chains,
which leads to a substantial increase in the complexity relative to
the situations considered by \cite{papaetal17}. To the best of our
knowledge, important characteristics of the extremal behaviour of
higher-order Markov chains have not been dealt with in-depth, yet
these are crucial for understanding the evolution of extreme events of
random processes and for providing well-founded parametric models that
can be used for inference, prediction and assessment of
risk, e.g., \cite{winttawn15, winttawn17}.

To help illustrate the complexity in higher-order Markov chains,
consider standard measures of extremal dependence \citep{coleheff99}.\
When analysing the extremal behaviour of a real-valued Markov process
$\{X_t \,:\, t=0,1,2,\dots \}$ with marginal distribution $F_t$ and
stationary copula \citep{joe15}, one has to distinguish between two
classes of extremal dependence.\ Let $D=\{1, \hdots, k\}$, where
$k\in \NN$ is the order of a Markov chain, and define
$M = 2^D \sm \{\emptyset\}$.\ The two classes can be
characterized through the quantities
\begin{equation*}
  \chi_{A}=\lim_{u \to 1}\PR(\{F_s(X_s)>u\}_{s\in A}\,|\,
  F_0(X_0)>u), \qquad \text{$A \in M$},
\end{equation*}
assuming that the limit exists for all $A\in M$.\ Bounds can be
obtained for the coefficients, e.g., for any $A_1, A_2 \in M$,
$0\leq\chi_{A_1 \cup A_2} \leq \min(\chi_{A_1}, \chi_{A_2})$.\ When
$\chi_D>0$ ($\chi_j = 0$ for all $j \in D$) we say that the process is
\emph{fully asymptotically dependent} (\emph{fully asymptotically
  independent}); for these processes, respectively, the bounds imply
that for all $A\in M$, $\chi_A>0$ if $|A|<k$ and $\chi_A=0$ if
$|A|>2$.
Previous work on $k$th order Markov chains considers only the case
where the process is fully asymptotically dependent, deriving results
via the tail chain \citep{resnzebe13a,janssege14}.\ In this paper, we
cover both cases, but also, we consider intermediate cases where
$\chi_A>0$ and $\chi_B=0$, for at least one $A,B \in M\sm D$
with $B \supsetneq A$.\

To derive greater detail about the behaviour within extreme events for
Markov chains more information than simply 
 $\chi_A$ for all $A\in M$ is required. In the case of stationary, fully asymptotically dependent Markov chains, where $k\ge 1$, with regular variation assumptions on the marginal
distribution, the appropriate strategy is to study \emph{tail chains} 
 \citep{janssege14}. A tail chain arises as a limiting process
 after witnessing an extreme state, under
rescaling of the future Markov chain by the extreme observation,
resulting in the tail chain being a multiplicative random walk. Even in these restrictive cases few results exist, e.g., \cite{perf97, yun98, janssege14}.
However, tail chains fail to reveal the detailed structure of extreme events for Markov chains for which $\chi_D=0$.
For such process, we show that the \emph{hidden tail
  chain} contains the necessary information, where the distinction between the hidden tail chain
and the tail chain is explained below.\
\cite{papaetal17} show that it is simplest to focus 
on Markov chains with marginal distributions with
exponential-like tails, in the max-domain of attraction of the Gumbel
distribution, using affine normalizations, as this reveals structure
not apparent through the use of regularly varying marginals with affine normalisations.\ 
This is the approach we will also take.



We let
$\sup\{x \,:\, F_0(x) < 1\}=\infty$.\ For $k>1$, we assume that there
exist $k-1$ norming functions $a_t\,:\,\RR \to \RR$ and
$b_t\,:\,\RR \to \RR_+$, for $t=1,\ldots, k-1$, such that
\[
  \left\{\frac{X_t - a_t(X_0)}{b_t(X_0)} \,:\,t=1,\ldots,k-1\right\}
  ~\Big|~\{X_0 > u\} \wk \{\Res_t \,:\,t=1,\ldots, k-1\}
\]
as $u\to \infty$, where $\wk$ denotes weak convergence of measures and
$(\Res_1,\ldots, \Res_{k-1})$ is a random vector that is
non-degenerate in each component.\ Then our aim is to find conditions
that guarantee the existence of an infinite sequence of additional
functions $a_t\,:\,\RR \to \RR$ and $b_t\,:\,\RR \to \RR_+$ for
$t=k,k+1,\ldots$, such that

\[
  \left\{\frac{X_t -
      a_t(X_0)}{b_t(X_0)}\,:\,t=1,2,\ldots\right\}~\Big|~ \{X_0 > u\} \wk
  \{\Res_t \,:\,t=1, 2,\ldots\}
\]
where each $\Res_t$ is non-degenerate.\ The limit process
$\{\Res_t\,:\,t=1,2,\ldots\}$ is termed the hidden tail chain.\ Note
that this limit process generalises the tail chain studied by
\cite{janssege14}, as in that treatment, the norming functions are
restricted to be $a_t(x)=x$ and $b_t(x)=1$ for all $t$, and any
$\Res_t$ can be degenerate at $\{-\infty\}$.\ In cases where we find
that $a_t(x)/x\to \alpha_t < 1$ as $x\to \infty$ for all
$t=1,2,\ldots ,k-1$, that is, the process has asymptotic pairwise
independence for all lags up to $k-1$, then the tail chain degenerates
as $\{-\infty, -\infty, \ldots\}$ but the hidden tail chain is
$\{\Res_1, \Res_2, \ldots\}$, which is finite and stochastic for all
components. Furthermore, if we find that $a_t(x) \sim x$ and
$b_t(x)\sim 1$ as $x\to \infty$ are required for all
$t=1,2,\ldots ,k-1$ then the hidden tail chain is identical to the
tail chain. So the hidden tail chain reveals important structure of
the extreme events lost by the tail chain when the tail chain becomes
degenerate but it equals the tail chain otherwise. Hence, hidden tail
chains have wider use than tail chains. We focus almost exclusively on
forward in time hidden tail chains, as above, but we also briefly
discuss back-and-forth hidden tail chains, expanding on the equivalent
feature for tail chains that \cite{janssege14} study.

Our primary target is to find how the first $k-1$ norming functions
$a_t(\cdot)$ and $b_t(\cdot)$, control those where $t\geq k$ and to
identify the transition dynamics of the hidden tail chain.\ To find
the behaviour of the $t\ge k$ norming functions requires a step-change
in approach relative to the case when $k=1$ studied by
\cite{papaetal17}. In particular, the transitions involve novel
functions, $a$ and $b$, of the $k$ previous values in comparison to
just the single value when $k=1$. Here we develop results for
determining the form of this class of functions and present a method
of their derivation in applications. It is important to also characterise how the dynamics of the hidden
tail chain encode information about how the process changes along its
index and state space.\  Under weak conditions, we make
the surprising finding that 
we can always express $\Res_{t}$ in the form
\begin{equation*}
  \Res_{t} =
  \psi^\loc_\tpo(\bm \Res_{\tk})+\psi^\scale_\tpo(\bm \Res_{\tk
  }) \,\eps_\tpo \qquad \text{for all $t>k$},
\end{equation*}
where $\bm \Res_{\tk}=(\Res_{t-k},\ldots, \Res_{t-1})$,
$\psi^\loc_\tpo\,:\, \RR^k \to \RR$,
$\psi^\scale_\tpo\,:\,\RR^k \to \RR_+$ are continuous update functions
and $\{ \eps_\tpo \,:\, t=1,2,\dots \}$ is a sequence of
non-degenerate i.i.d.\ innovations.\ This simple structure for the
hidden tail chain is controlled through the update functions
$\psi^\loc_\tpo$ and $\psi^\scale_\tpo$ which we show take particular
classes of forms.\ Using the values of $a_t,b_t$ and the properties of
$Z_t$, as $t\rightarrow \infty$, we are able to investigate how the
Markov chain returns to a non-extreme state following the occurrence
of an extreme state.\ Addtionally, some parallels between the extremal
properties of the norming functions and the Yule--Walker equations,
used in standard time-series analysis \citep{yule27, walk31}, are
identified.

The limit theory developed in this paper is the first that considers
asymptotic independence when studying extreme values of any structured
process other than a first-order Markov processes. The extension to
$k$th order Markov processes opens the possibility to developing
similar theory for much broader classes of graphical models. Studying
multivariate extreme values on graphical structures has been a rich
vein of research recently, with several influential papers such as
\cite{engehitz19, segers2019one, asensege2022} and \cite{
  engelke2020sparse}. However, all these papers focus on the case of
all underlying distributions of cliques on the graph being
asymptotically dependent. We believe that the results in this paper
will help to unlock these approaches to enable the case when some, or
all, cliques have asymptotic independence.
\paragraph{Organization of the paper.}
In Section~\ref{sec:th_res}, we state our main theoretical results for
higher-order tail chains with affine update functions under rather
broad assumptions on the extremal behaviour of both fully
asymptotically dependent and fully asymptotically independent Markov
chains. As in previous accounts (\cite{perf94,resnzebe13a,janssege14},
\cite{kulisoul} and \cite{papaetal17}), our results only need the
homogeneity (and not the stationarity) of the Markov chain and
therefore, we state our results in terms of homogeneous Markov chains
with initial distribution $F_0$.\ In Section~\ref{sec:SDE} we study
hidden tail chains of asymptotically independent and asymptotically
dependent stationary Markov chains with standardized marginal
distributions. Subsequently, in Section~\ref{sec:autocovariance} we
characterise closed-form solutions for the norming functions for a
class of asymptotically independent Markov chains, with the structure
of these functions paralleling that of the autocovariance in
Yule--Walker equations.\ In Section~\ref{sec:examples_section}, we
provide examples of Markov chains constructed from widely studied
joint distributions including a Markov chain which is neither fully
asymptotically independent or dependent.\ All proofs are postponed to
Appendix~\ref{sec:proofs}.

\paragraph{Some notation.} We use the following notation.\ Vectors are
typeset in bold and vector algebra is used throughout the paper.\ For
a sequence of measurable functions $\{g_t\}_{t \in \NN}$ and real
valued numbers $\{x_t\}_{t \in \NN}$, the notation $\bm g_{\tk}(x)$
and $\bm x_{\tk}$, , for $t,k,t-k\in \NN$, is used to denote
$(g_{t-k}(x),\hdots,g_{t-1}(x))$ and $(x_{t-k},\hdots,x_{t-1})$,
respectively.\ By convention, univariable functions on vectors are
applied componentwise, e.g., if $f \,:\,\RR\to \RR$,
$\bm x \in \RR^k$, then $f (\bm x) = (f (x_1), \ldots, f (x_k))$.\ The
symbols $\bm 0_p$ and $\bm 1_p$, where $p \in \NN$ are used to denote
the vectors $(0,\ldots,0) \in \RR^p$ and $(1,\ldots,1)\in\RR^p$.\ For
a topological space $E$ we denote its Borel-$\sigma$-algebra by
$\borel(E)$ and the set of bounded continuous functions on $E$ by
$C_b(E)$.\ If $f_n,f$ are real-valued functions on $E$, we say that
$f_n$ converges uniformly on compact sets to $f$ if for any compact
$C \subset E$ the convergence
$\lim_{n \to \infty} \sup_{x \in C} \lvert f_n(x)-f(x) \rvert = 0$
holds true.\ Moreover, $f_n$ is said to converge uniformly on compact
sets to $\infty$ if $\inf_{x \in C} f_n(x) \to \infty$ for compact
sets $C \subset E$.\ Weak convergence of measures on $E$ is
abbreviated by $\wk$.\ \smash{For random elements $X, X_1,X_2,\dots$
  defined on the same probability space, we say $\{X_n\}$} converges
in distribution to $X$, \smash{and we write $X_n\cind X$, if the
  distributions $P_n$ of the $X_n$ converge weakly to the distribution
  $P$} of $X$, that is, if $P_n\wk
  P$.\ 
The closure of set $A$ is denoted by $\overline{A}$.\ 
We use the notation $\lVert \bm x \lVert $ for the $L_1$ norm of a
$k$-dimensional vector $\bm x$.\ For a Cartesian coordinate system
$\RR^k$ with coordinates $x_1,\ldots,x_{k}$, $\nabla$ is defined by
the partial derivative operators as
$\nabla=\sum_{i=1}^k (\partial/\partial x_i) \bm e_i$ for an
orthonormal basis $\{\bm e_1, \ldots, \bm e_k\}$.\ For a
differentiable function $f\,:\,\RR^k\rightarrow \RR$,
$\nabla f(\bm x)=((\nabla f)_1(\bm x),\dots, (\nabla f)_k(\bm x))$
denotes the gradient vector of $f$ at $\bm x$. The notation
$\bm x\cdot \bm y$ is used for the scalar product of two vectors
$\bm x, \bm y\in \RR^k$, that is,
$\bm x\cdot \bm y = \sum_{i=1}^kx_i \,y_i$.\ A function
$f\,:\,\RR^k \to \RR$ is termed $\rho-$homogeneous, $\rho\in\RR$ if
$f(t \bm x)= t^\rho f(\bm x)$ for all $t>0$ and $\bm x\in\RR^k$.\ The
spectral radius $\rho(\bm J)$ of a square matrix
$\bm J\in \RR^{k\times k}$ equals the maximum of the modulus of its
eigenvalues. The standard $(k-1)$-dimensional unit-simplex
$\left\{\bm \omega \in \RR_+^k\,:\, \lVert\bm \omega \lVert
  =1\right\}$, $k\geq 1$, is denoted by $\SI_{k-1}$.

\section{Theory}
\label{sec:th_res}
\subsection{Marginal standardization}
To facilitate the generality of our theoretical developments, our
assumptions about the margins of the process throughout
Section~\ref{sec:th_res} 
only
concern the tail behaviour of the random variable at which we
condition the Markov process to exceed a level.\ This assumption is in
the style of theoretical approaches in conditional extreme value
theory \citep{heffres07} and is made precise by Assumption $\AAA$.
\begin{description}[wide=0\parindent] \em
\item[Assumption $\AAA$.] $F_0$ has upper end point $\infty$ and there
  exists a non-degenerate probability distribution $H_0$ on
  $[0,\infty)$ and a measurable norming function $\sigma(v)>0$, such
  that
  \[
    \frac{F_0(v+\sigma(v) x)}{\overline{F}_0(v)}\wk H_0(x)\qquad
    \text{as $v\to\infty$}.
  \]
\end{description}
From \cite{pick75}, the limit distribution $H_0$ can be identified by
a generalized Pareto distribution with a non-negative shape
parameter.\ In later sections, we will take
$H_0(x)=1-\exp(-x), x\geq 0$, that is, the standard exponential
distribution, such that $F_0$ lies in the maximum domain of attraction
of a Gumbel distribution.

\subsection{Chains with location and scale norming}
\label{sec:A}
The next assumption ensures that after an extreme event at time $t=0$,
a non-degenerate initial distribution, given $\{X_0=v\}$, is obtained
in the limit as $v\to\infty$ for the first $k-1$ renormalized states
of the Markov process.
\begin{description}[wide=0\parindent]
\item[Assumption $\AA$.] \emph{(behaviour of initial states in the
    presence of an extreme event)} \em If $k>1$, there exist
  \begin{enumerate}
  \item[$(i)$] for $t=1,\hdots,k-1$, measurable functions
    $a_t\,:\,\RR \to \RR$ and $b_t\,:\,\RR \to \RR_+$, satisfying
    $a_t(v)+b_t(v)\,x\to \infty$ as $v\to \infty$, for all $x\in \RR$;
  \item[$(ii)$] a distribution $G$ supported on $\RR^{k-1}$ that has
    non-degenerate margins such that
  \end{enumerate}
  \begin{IEEEeqnarray*}{rCl}
    && \Pr\left(\frac{\bm X_{\kkminone}-\bm a_{\kkminone}(v)}{\bm
        b_{\kkminone}(v)} \leq \bm \res_{\kkminone}~\Big|~X_0 = v\right)
    \wk G\,[-\bm \infty, \bm \res_{\kkminone}] \quad \text{as
      $v\to \infty$.}
  \end{IEEEeqnarray*}
\end{description}
\begin{remark}\label{rem:dist_support} When saying that a distribution
  is supported on a subset $A$ of $\RR^k$, we do not allow the
  distribution to place mass at the boundary $\partial A$ of $A$.
\end{remark}
Assumption $\AA$ implies that $a_t(v) \to \infty$ and
$b_t(v)=o(a_t(v))$ as $v\to \infty$ and in Section \ref{sec:scale_only}
we cover the case where $a_t(v) = 0$ for all $v > 0$ and
$b_t(v)\to \infty$ as $v\to
\infty$.\

After initializing the states $X_0,\ldots, X_{k-1}$, a complete
characterization of the advancing sequence of states for $t\geq k$ can
be given from the one-step transition probability kernel of the
homogeneous Markov process
\[
  \pi(\bm x_{\kk}, x_k):= \PR\left(X_{k}\leq x_k \mid \bm X_{\kk} = \bm
    x_{\kk}\right).
\]
To motivate our next assumption about the behaviour of the transition
probability kernel of the process, consider how a complete
characterization may be given for higher-order Markov processes with
$k>1$ using induction on $\NN$.\ Fix a $t \geq k> 1$ and assume there
exist sequences of norming functions $a_i$ and $b_i$,
$i=1,\dots, t-1$, such that,
\begin{equation*}
  \frac{\bm X_{1\,:\,t-1} - \bm a_{1\,:\,t-1}(
    X_0)}{\bm b_{1\,:\,t-1}(X_0)} ~\Big |~\{X_0>u\} \cind \bm \Res_{1\,:\,t-1}\qquad \text{as $u\to \infty$},
\end{equation*}
where each $\Res_i$ is a random variable with a non-degenerate
distribution on $\RR$. Therefore, what is required is to assert that,
under the induction hypothesis, we can find $a_t$ and $b_t$ such that
$\{\bm X_{1\,:\,t} - \bm a_{1\,:\,t}( X_0)\}/\bm b_{1\,:\,t}(X_0) \mid
\{X_0>u\} \cind \bm \Res_{1\,:\,t}$, as $u\to \infty$, where $\Res_t$
is a random variable with a non-degenerate distribution supported on
$\RR$ and $\bm a_{1\,:\,t}=(\bm a_{1\,:\,t-1}, a_t)$ and
$\bm b_{1\,:\,t}=(\bm b_{1\,:\,t-1}, b_t)$.\ To motivate our
assumptions that guarantee this latter convergence, it suffices to
consider marginal convergence, \textit{viz}., the case where the
distribution of $\{X_t-a_t(X_0)\}/b_t(X_0)\mid \{X_0 > u\}$ converges
weakly under the induction hypothesis.\ For any $t\geq k$, standard
calculations give that
\begin{IEEEeqnarray}{l}
  \Pr\left(\frac{X_t-a_t(X_0)}{b_t(X_0)} \leq x_t~\Big|~X_0
    >u\right)=\nonumber\\
  \frac{1}{\overline{F}_0(u)}
  \int\displaylimits_{u}^{\infty}  \int\displaylimits_{\RR^{t-1}} 
  \PR\left(\bm
    X_{0\,:\,t-1}\in d\bm x_{0\,:\,t-1} \right)
  \Bigg[\,\int\displaylimits_{ -\infty}^{x_t}
  \PR\left(\frac{X_t-a_t(X_0)}{b_t(X_0)} \in d z_t ~\Big |~ \bm
    X_{0\,:\,t-1} =\bm x_{0\,:\,t-1}\right)\Bigg],
  \label{eq:marg_conv}
\end{IEEEeqnarray}
where $d \bm x_{0\,:\,t-1}$ is shorthand for
$ d x_0\times\cdots\times d x_{t-1}$.\ Firstly, replace $a_t(X_0)$ by
$a_t(x_0)$ in the innermost integral, by virtue of the conditioning on
the exact value of $X_0$ being equal to $x_0$.\ Then, use the Markov
property so that the conditioning on all previous states is reduced to
conditioning on the previous $k$ states.\ Changing variables to
$z_0 = \{x_0 - u\}/\sigma(u)$ and $z_i = \{x_i - a_i(x_0)\}/b_i(x_0)$
for $i=1,\dots,t-1$, we thus have that for any $t\geq k$, expression
\eqref{eq:marg_conv} is equal to
\begin{IEEEeqnarray}{rCl}
  &&\int\displaylimits_{0}^{\infty} \frac{F_0\{v_u(d
    z_0)\}}{\overline{F}_0(u)} \times \left[
    \,\,\,\int\displaylimits_{\RR^{t-1}} \PR\left(\frac{\bm
        X_{1\,:\,t-1}-\bm a_{1\,:\,t-1}(v_u(z_0))}{\bm
        b_{1\,:\,t-1}(v_u(z_0))}\in d\bm z_{1\,:\,t-1}~\Big|~ X_0 =
      v_u(z_0)\right)\times\right.
  \nonumber \\\nonumber\\
  && \quad\times \left.\left\{\,\,\int\displaylimits_{-\infty}^{x_t}
      \PR\left(\frac{X_t-a_t(v_u(z_0))}{b_t(v_u(z_0))} \in d z_t ~\Big
        |~ \frac{\bm X_{\tk}-\bm a_{\tk}(v_u(z_0))} {\bm
          b_{\tk}(v_u(z_0))}=\bm
        z_{\tk}\right)\right\}\right],\label{eq:kernel_connexion}
\end{IEEEeqnarray}
where $a_0(x)=x$, $b_0(x)=1$ for all $x\in \RR$ and
$v_u(z_0)=u + \sigma(u)\, z_0$.\ Hence, convergence of the innermost
integral in the curly parentheses in
equation~\eqref{eq:kernel_connexion}, is necessary for marginal
convergence.\ Let
$\bm A_{t}(v,\bm z)=\bm a_{\tk}(v) + \bm b_{\tk}(v) \,\bm z$.\ Then,
observe that for arbitrary mappings $a\,:\,\RR^k\to \RR$ and
$b\,:\,\RR^k\to\RR_+$ , the innermost integral can be written as
\begin{IEEEeqnarray*}{rCl}
  &&\int\displaylimits_{-\infty}^{x_t} \PR\left(\frac{X_t-a(\bm
      X_{\tk})}{b(\bm X_{\tk})} \in \frac{d z_t}{\psi_{t,u}^b(\bm
      X_{\tk})}-\psi_{t,u}^a(\bm
      X_{\tk}) ~\Big |~ \bm X_{\tk}=\bm A_{t}(v_u(z_0),\bm
    z_{\tk})\right),
\end{IEEEeqnarray*}
where $\psi_{t,u}^a$ and $\psi_{t,u}^b$ are given in expression
\eqref{eq:psi_functions} and depend on the mappings $a$, $b$ and the
normings $a_t, b_t$, $t \geq k$.\ Therefore, the connection between
the convergence of the transition probability kernel and the sought
marginal convergence can be established when the oscillation of the
functions $a$ and $b$ in a neighbourhood of infinity is controlled,
that is, when $a$ and $b$ are chosen such that the functions
$\psi_{t,u}^a$ and $\psi_{t,u}^b$ converge locally uniformly to
real-valued limits ($\psi_t^\loc$ and $\psi_t^\scale$, respectively),
as $u\to \infty$.\ These observations motivate our next assumption.
\begin{description}[wide=0\parindent] 
\item[Assumption $\AB$.] \emph{(behaviour of the next state of the
    process as the previous $k$ states become extreme)} \em Let
  $k\geq 1$, $a_0(x)=x$ and $b_0(x)=1$. If $k > 1$, suppose that for
  $\bm X_{\kkminone}$, the Assumption $\AA$ holds with norming
  functions $\bm a_{\kkminone}$ and $\bm b_{\kkminone}$.\ There exist,
  \begin{enumerate}[wide=0\parindent] 
  \item[$(i)$] for $t=k,k+1,\dots$, measurable functions
    $a_t\,:\,\RR \to \RR$ and $b_t\,:\,\RR \to \RR_+$,
    continuous update functions
    $\psi_{t}^a\,:\,\RR^k \to \RR$,
    $\psi_t^b\,:\,\RR^k \to \RR_+$ and measurable functions
    $\afunD\,:\,\RR^k \to \RR$,
    $\bfunD\,:\,\RR^k \to \RR_+$, such that, for
    all $\bm \res\in\RR^k$
    \begin{equation}   
      \psi_{t,v}^a(\bm \res_v):=\frac{a(\bm A_{t}(v,\bm \res_v))-a_{t}(v)}{b_{t}(v)}
      \to \psi_{t}^a( \bm \res) \quad \text{and} \quad
      \psi_{t,v}^b(\bm \res_v):=\frac{b(\bm A_{t}(v,\bm \res_v))}{b_{t}(v)}\to \psi_{t}^b(\bm \res), 
      \label{eq:psi_functions}
    \end{equation}
    whenever $\bm \res_v \to \bm \res$ as $v\to\infty$, where
    $\bm A_{t}(v,\bm \res):= \bm a_{\tk}(v)+\bm b_{\tk} (v)\,\bm \res$;
  \item[$(ii)$] a non-degenerate distribution $K$ supported on $\RR$,
    such that for all $\bm \res\in\RR^k$ and for any $f\in C_b(\RR)$
  \begin{equation*}
    \int_\RR f(x) \pi[\bm A_{t}(v,\bm \res_v), a(\bm A_{t}(v,\bm \res_v)) +
    b(\bm A_{t}(v,\bm \res_v))\,dx]\to\int_{\RR}f(x)K(dx),\qquad t=k,k+1,\dots,
    \label{eq:A2}
  \end{equation*}
  whenever $\bm \res_v \to \bm \res$ as $v\to \infty$.
\end{enumerate}
\end{description}
By the same token, we establish the weak convergence of the
renormalized Markov chain to a hidden tail-chain.\ This is asserted by
Theorem \ref{thm:tailchain} below.
\begin{theorem}
  \label{thm:tailchain}
  Let $\{X_t\,:\,t=0,1,\hdots\}$ be a homogeneous $k$th order
  Markov chain satisfying Assumptions $\AAA, \AA$ and $\AB$.\ Then as
  $v\to \infty$
  \begin{equation}
    \left(\frac{X_0 - u}{\sigma(v)}\CommaBin\frac{X_1 -
        a_1(X_0)}{b_1(X_0)}\CommaBin\cdots\CommaBin\frac{X_t -
        a_t(X_0)}{b_t(X_0)}\right)~\Big|~\{X_0>v\} \cind
    (E_0,\Res_1,\hdots,\Res_t), \qquad t \geq k,
    \label{eq:htc}
  \end{equation}
  where
\begin{enumerate}
\item[$(i)$] $E_0\sim H_0$ and $(\Res_1,\Res_2\hdots,\Res_t) $ are independent,
\item[$(ii)$] $\Res_0=0$ a.s., $(\Res_1,\hdots,\Res_{k-1})\sim G$ and
  \begin{equation}
    \Res_{s} = \psi_s^a( \bm \Res_{s-k\,:\,s-1}) + \psi_s^b(
    \bm \Res_{s-k\,:s-1})\,\varepsilon_s,\qquad s=k,k+1,\hdots
    \label{eq:tail_chain}
  \end{equation}
  for a sequence of i.i.d. random variables $\varepsilon_s\sim K$.
\end{enumerate}
\end{theorem}
Our theory provides a constructive approach to identify the sequence
of additional norming functions $a_t$ and $b_t$, for $t=k,k+1,\dots$.\
This is due to Proposition \ref{lemma:doa_equivalence} below, a proof
is given in Appendix~\ref{sec:proof_doa_equivalence}.
\begin{proposition}
  \label{lemma:doa_equivalence}
  Let $a\,:\,\RR^k\to \RR$ and $b\,:\,\RR^k\to \RR_+$ be measurable
  maps.\ Let $t\geq k$ and $\bm \res\in\RR^k$.\ The following
statements are equivalent
  \begin{itemize}
  \item[(i)] There exist measurable functions $a_t\,:\,\RR\to \RR$,
    $b_t\,:\,\RR\to \RR_+$ and continuous functions
    $\psi_t^a\,:\,\RR^k\to\RR$ and $\psi_t^b\,:\,\RR^k\to\RR_+$, such
    that convergence \eqref{eq:psi_functions} holds.
  \item[(ii)] There exist continuous functions
    $\lambda_t^a\,:\,\RR^k\to\RR$ and $\lambda_t^b\,:\,\RR^k\to\RR_+$,
    such that for all $\bm \res \in \RR^k$
    \begin{IEEEeqnarray*}{rCl} && \frac{a(\bm A_{t}(v,\bm
\res_v))-a(\bm A_{t}(v,\bm 0))}{b(\bm A_{t}(v,\bm 0))}
\to\lambda_{t}^a( \bm \res) \quad\text{and} \quad \frac{b(\bm
A_{t}(v,\bm \res_v))}{b(\bm A_{t}(v,\bm 0))}\to \lambda_t^b(\bm \res),
\end{IEEEeqnarray*}
\text{whenever $\bm \res_v\to \bm \res$ as $v\to \infty$}.
\end{itemize}
\end{proposition}
\begin{remark}
  \label{rem:order1} When $k=1$, due to Proposition
  \ref{lemma:doa_equivalence} we can choose, without loss of
  generality, $a_1(v)=a(v)$ and $b_1(v)=b(v)$ so that
  $\psi_1^a(0) = \{a(v) - a_1(v)\}/b_1(v) = 0$ and
  $\psi_1^b(0)= b(v)/b_1(v) = 1$.\ Consequently,
  expression~\eqref{eq:tail_chain} implies that
  $\Res_1 = \varepsilon_1$ and thus, the special case of $k=1$ in
  Theorem~\ref{thm:tailchain} corresponds to the results of
  \cite{papaetal17}.
\end{remark}

\subsection{Near extremally independent chains}
\label{sec:EI_chains}
In this section, we consider Markov chains where no norming of the
location and no norming of the scale are needed.\ This case resembles
the formulation of Theorem \ref{thm:tailchain}, but has $a_t(v)=0$ and
$b_t(v)=1$, for all $t\geq 1$.\ Thus, the next assumption ensures that
after an extreme event at time $t=0$, a non-degenerate distribution,
given $\{X_0=v\}$, is obtained in the limit as $v\to\infty$ for the
first $k$ states of the Markov process, without any renormalization.
\begin{description}[wide=0\parindent]
\item[Assumption $\AA^*$.] \emph{(behaviour of next $k$ states in the
    presence of an extreme event)} \em There exists a distribution $G$
  supported on $\RR^{k}$ that has non-degenerate margins such that
  \begin{IEEEeqnarray*}{rCl}
    && \Pr\left(\bm X_{1\,:\,k} \leq \bm \res_{1\,:\,k}~\Big|~X_0 =
      v\right) \wk G\,[-\bm \infty, \bm \res_{1\,:\,k}] \quad \text{as
      $v\to \infty$.}
  \end{IEEEeqnarray*}
\end{description}
\cite{hefftawn04} showed that Assumption $\AA^*$ holds for the
Morgenstern copula for $k\geq 1$.\ A related assumption also appears
\cite{maulik2002asymptotic} for the case $k=1$ with $(X_0, X_1)$ being
nonnegative random variables.\ Here, we note that if
$\bm X_{\kplusonekplusone}$ has the independence copula, then
$G(\bm \res_{1\,:\,k}) = \prod_{j=1}^kF_j(z_j)$ whereas cases with
$G_j(z) \geq F_j(z)$ ($G_j(z) \leq F_j(z)$) for all $z\in\RR$, with
$G_j$ being the $j$th marginal distribution of $G$, correspond to
positive (negative) near extremal independence at lag $j$ in the
hidden tail chain.

Assumptions $\AAA$ and $\AA^*$ are sufficient to establish the weak
convergence of the conditioned Markov chain to a hidden tail-chain in
Theorem \ref{thm:tailchain_EI} below.\ The proof of this theorem
follows along the lines of the proof of Theorem~\ref{thm:tailchain}
and is omitted for brevity.
\begin{theorem}
  \label{thm:tailchain_EI}
  Let $\{X_t\,:\,t=0,1,\hdots\}$ be a homogeneous $k$th order Markov
  chain satisfying Assumptions $\AAA$ and $\AA^*$.\ Then as
  $v\to \infty$
  \begin{equation}
    \left(\frac{X_0 - u}{\sigma(v)}, X_1,\dots, X_t\right)~\Big|~\{X_0>v\} \cind
    (E_0,\Res_1,\hdots,\Res_t)
    \label{eq:htc_EI}
  \end{equation}
  where
  \begin{enumerate}
  \item[$(i)$] $E_0\sim H_0$ and $(\Res_1,\Res_2\hdots,\Res_t) $ are
    independent,
  \item[$(ii)$] $(\Res_1,\hdots,\Res_{k})\sim G$ and
    \[
      \Res_{s}=\pi^{-1}(\bm \Res_{s-k\,:\,s-1}, U_s), \qquad s=k+1,k+2,\dots,
    \]
    where $\{U_s\}$ is a sequence of independent uniform$(0,1)$ random
    variables and $\pi^{-1}\,:\,\RR^k \times (0,1)\to \RR$ with
    $\pi^{-1}(\bm z, u) := \inf\{x \in \RR\,:\,\pi(\bm z, x)> u\}$.
  \end{enumerate}
\end{theorem}
\noindent We note that if $\bm X_{\kplusonekplusone}$ has the
independence copula, then $\pi^{-1}(\bm z, u)$ is independent of
$\bm z$.

\subsection{Nonnegative chains with only scale norming}
\label{sec:scale_only} 
Consider nonnegative Markov
chains where no norming of the location is needed.\ As in
\cite{papaetal17}, we require extra care relative to Section
\ref{sec:A} since the convergences in Assumption $\AB$ $(i)$ will
be satisfied for all $x \in (0, \infty)$, not all $x \in [0, \infty)$.\ Hence, we have to control the mass of the
limiting renormalized initial distribution and the limiting
renormalized transition probability kernel of the Markov process.
\begin{description}[wide=0\parindent] \em
\item[Assumption $\BA$.] \emph{(behaviour of the next state as the
    previous states becomes extreme)} There exist measurable functions
  $b_t\,:\,\RR_+\to\RR_+$ for $t=1,\hdots,k$, such that
  $b_t(v)\to \infty$ as $v\to\infty$ and a non-degenerate distribution
  function $G$ on $[0,\infty)^{k}$, with no mass at any of the
  half-planes
  $C_j=\{(\res_1,\ldots,\res_{k-1})\in [0,\infty)^{k-1}:
  \text{$\res_j = 0$}\}$, that is, $G(\{C_j\})=0$ for $j=1,\ldots, k-1$,
  such that as $v\to \infty$,
  \begin{IEEEeqnarray*}{rCl} && \Pr\left(\frac{\bm X_{\kkminone}}{\bm
        b_{\kkminone}(v)} \leq \bm \res_{\kkminone}~\Big|~X_0 =
      v\right) \wk G\,[\bm 0, \bm \res_{\kkminone}] \quad \text{as
      $v\to \infty$.}
  \end{IEEEeqnarray*}
\end{description}

\begin{description}[wide=0\parindent]
\item[Assumption $\BB$.] \emph{(behaviour of the next state of the
process as the initial states become extreme)} \em Let $k\geq 1$ and
$b_0(x)=x$. If $k > 1$, suppose that for $\bm X_{\kkminone}$, the
Assumption $\BA$ holds with norming functions $\bm b_{\kkminone}$.\
There exist
\begin{enumerate}
\item[$(i)$] for $t=k,k+1,\dots$, measurable functions
  $b_t\,:\,\RR_+ \to \RR_+$, continuous update functions
  $\psi_t^b\,:\,\RR_+^k \to \RR_+$ and a measurable function
  $\bfunD\,:\,\RR_+^k \to \RR_+$, such that for all
  $\delta_1,\ldots,\delta_k > 0$ and $\bm \res \in [\delta_1,\infty)
  \times\ldots \times [\delta_k, \infty)$,
  \begin{equation} \lim_{v\to\infty} \frac{b( \bm B_{t}(v,\bm
      \res_v))}{b_{t}(v)}\to \psi_{t}^b(\bm \res)>0,
    \label{eq:psi_functions_scale}
  \end{equation} whenever $\bm \res_v \to \bm \res$ as $v\to
  \infty$ and 
  $\sup\{\lVert \bm \res \lVert_{\infty}\,:\, \bm
  \res\in A_c\} \to 0$ as $c\downarrow 0$,
  where $\bm B_{t}(v,\bm \res):= \bm b_{\tk} (v)\,\bm \res$
  and  $A_c = \{\bm \res\in
  (0, \infty)^k\,:\, \psi_t^b(\bm \res)\leq c\}$ with the convention
  that $\sup(\emptyset) = 0$;
\item[$(ii)$] a non-degenerate distribution $K$ supported on
  $[0,\infty)$ with no mass at $\{0\}$, that is, $K\,\{0\}=0$, such
  that, for any $f\in C_b(\RR_+)$,
  \end{enumerate}
  \begin{equation*} \int_{\RR_+}f(x)\, \pi[\bm B_{t}(v,\bm \res_v),
b(\bm B_{t}(v,\bm \res_v))\,dx]\to \int_{\RR_+} f(x)\,K(dx), \qquad
\text{$t=k,k+1,\dots$}
  \end{equation*} whenever $\bm \res_v \to \bm \res$ as $v\to \infty$.
\end{description}

\begin{theorem}
  \label{thm:tailchain:nonneg} Let $\{X_t\,:\,t=0,1,\hdots\}$ be a
homogeneous Markov chain satisfying Assumptions $\AAA, \BA$ and $
\BB$.\ Then as $u\to \infty$
  \begin{equation} \left(\frac{X_0 -
u}{\sigma(u)}\CommaBin\frac{X_1}{b_1(X_0)}\CommaBin\cdots\CommaBin\frac{X_t}{b_t(X_0)}\right)~\Big|~\{X_0>u\}
\cind (E_0,\Res_1,\hdots,\Res_t)
    \label{eq:htc_scale}
  \end{equation} where
  \begin{enumerate}
  \item[$(i)$] $E_0\sim H_0$ and $(\Res_1,\Res_2\hdots,\Res_t) $ are
independent,
  \item[$(ii)$] $\Res_0=1$ a.s., $(\Res_1,\hdots,\Res_{k})\sim G$ and
    \begin{equation} \Res_{s} = \psi_s^b(\bm
\Res_{s-k\,:s-1})\,\varepsilon_s,\quad s=k,k+1,\hdots
      \label{eq:tail_chain_scale}
    \end{equation} for a sequence of i.i.d. random variables $\varepsilon_s\sim K$.
  \end{enumerate}
\end{theorem}
\begin{remark} Theorem \ref{thm:tailchain:nonneg} appears
  simply to be the $k$-order extension of Theorem 3.1 in
  \cite{kulisoul} but it differs as $H_0$ is in a broader class than
  the Pareto family considered by them. Here, an exponential tail for
  $H_0$ is also included.
\end{remark}  

\subsection{Back-and-forth hidden tail chains}

{ \color{black} In the discussion above, formally the entities we have
  referred to as tail chains and hidden tail chains are in fact
  forward tail and hidden tail chains
  \citep[\textit{cf.}][]{janssege14}.\ These describe the behaviour of
  the Markov chain only forward in time from a large observation.\
  There is also the parallel interest in a backward tail/hidden tail
  chain, to give how the chain evolves into an extreme event, and the
  joint behaviour of the two, known as back-and-forth tail processes.
}

{ \color{black} Here we focus on an extension of the back-and-forth
  tail chains developed by \cite{janssege14}.\ The backward hidden
  tail chain characteristics are similar in structure to the forward
  hidden tail chain properties identified in
  Sections~\ref{sec:A}--\ref{sec:scale_only}.\ To save repetition,
  here we outline the 
  back-and-forth hidden tail chains for the assumptions in
  Section~\ref{sec:A} only. For this purpose, it suffices to consider
  a straightforward extension of Assumption $\AB$ which allows us to
  characterize the backward behaviour of the chain from an extreme
  event by requiring functional normalization for the backward chain
  $X_{-s}\mid \bm X_{-s+1\,:\,-s+k}$, $s\in\NN$.\ Clearly, if the
  chain is time-reversible, then Assumption $\AB$ holds backwards with
  the same functional normalizations $a$ and $b$ and the same limit
  distribution $K$. In general, however, there is no mathematical
  connection between these forward and backward quantities and
  Assumption $\AB^{-}$ below considers this more general case.
  \begin{description}[wide=0\parindent]
  \item[Assumption $\AB^{-}$.] \emph{(behaviour of the backward state
      of the process)} \em Let
    $k\geq 1$, $a_0(x)=x$ and $b_0(x)=1$. If $k > 1$, suppose that for
    $\bm X_{\kkminone}$, the Assumption $\AA$ holds with norming
    functions $\bm a_{\kkminone}$ and $\bm b_{\kkminone}$.\ There
    exist,
  \begin{enumerate}[wide=0\parindent] 
  \item[$(i)$] for $s=1,2,\dots$, measurable functions
    $a_{-s}\,:\,\RR \to \RR$ and $b_{-s}\,:\,\RR \to \RR_+$, continuous
    update functions $\psi_{-s}^{a^-}\,:\,\RR^k \to \RR$,
    $\psi_{-s}^{b^-}\,:\,\RR^k \to \RR_+$ and measurable functions
    $\afunD^{-}\,:\,\RR^k \to \RR$, $\bfunD^{-}\,:\,\RR^k \to \RR_+$,
    such that, for
    all $\bm \res\in\RR^k$
    \begin{equation}
      \frac{a^{-}(\bm A_{-s}(v,\bm \res_v))-a_{-s}(v)}{b_{-s}(v)}
      \to \psi_{-s}^{a^-}( \bm \res) \quad \text{and} \quad
      \frac{b^{-}(\bm A_{-s}(v,\bm \res_v))}{b_{-s}(v)}\to \psi_{-s}^{b^-}(\bm \res), 
      \label{eq:psi_functions_back}
    \end{equation}
    whenever $\bm \res_v \to \bm \res$ as $v\to\infty$, where
    $\bm A_{-s}(v,\bm \res):= \bm a_{\skb}(v)+\bm
    b_{\skb} (v)\,\bm \res$;
  \item[$(ii)$] a non-degenerate distribution $K^{-}$
    supported on $\RR$, such that for all $\bm \res\in\RR^k$ and for
    any $f\in C_b(\RR)$
  \begin{equation*}
    \int_\RR f(x) \pi[\bm A_{-s}(v,\bm \res_v), \afunD^{-}(\bm A_{-s}(v,\bm \res_v)) +
    \bfunD^{-}(\bm A_{-s}(v,\bm \res_v))\,dx]\to\int_{\RR}f(x)K^{-}(dx),\qquad s=1,2,\dots,
    \label{eq:A2_back}
  \end{equation*}
  whenever $\bm \res_v \to \bm \res$ as $v\to \infty$.
\end{enumerate}
  \end{description}
  
  The back-and-forth hidden tail chain is presented in Theorem
  \ref{thm:tailchain_back}.\ For the sake of brevity, we do not
  include its proof as this is identical to the proof of
  Theorem~\ref{thm:tailchain}.
  \begin{theorem}
    \label{thm:tailchain_back}
    Let $\{X_t\,:\,t=0,1,\hdots\}$ be a homogeneous $k$th order Markov
    chain satisfying Assumptions $\AAA, \AA$, $\AB$ and $\AB^-$.\ Then
    as $v\to \infty$
    \begin{equation*}
      \left(\frac{\bm X_{-s\,:\,-1} -
          \bm a_{-s\,:\,-1}(X_0)}{\bm b_{-s\,:\,-1}(X_0)}\CommaBin\frac{X_0 - u}{\sigma(v)}\CommaBin\frac{\bm X_{1\,:\,t} -
          \bm a_{1\,:\,t}(X_0)}{\bm b_{1\,:\,t}(X_0)}\right)~\Big|~\{X_0>v\} \cind
      (\bm \Res_{-s\,:\,-1},E_0,\bm \Res_{1\,:\,t}),\qquad t \geq k,\quad s\in\NN,
    \end{equation*}
    where
    \begin{enumerate}
    \item[$(i)$] $E_0\sim H_0$ is independent of
      $(\bm \Res_{-s\,:\,-1}, \bm \Res_{1\,:\,t})$,
    \item[$(ii)$] $\Res_0=0$ a.s., $\bm Z_{1\,:\,k-1}\sim G$,
      \begin{equation*}
        \Res_{t} = \psi_t^a( \bm \Res_{t-k\,:\,t-1}) + \psi_t^b(
        \bm \Res_{t-k\,:t-1})\,\varepsilon_t,\qquad t=k,k+1,\hdots
      \end{equation*}
      and
      \begin{equation*}
        \Res_{s} = \psi_{-s}^{a^-}( \bm \Res_{-s+1\,:\,-s+k}) + \psi_{-s}^{b^-}(
        \bm \Res_{-s+1\,:\,-s+k})\,\varepsilon_{-s},\qquad s=1,2,\hdots
      \end{equation*}
      for independent sequences of i.i.d. random variables
      $\{\varepsilon_{-s}\}_{s=1}^{\infty}$ and
      $\{\varepsilon_t\}_{t=k}^{\infty}$, where
      $\varepsilon_{-s} \sim K^-$ and $\varepsilon_t\sim K$.
    \end{enumerate}
  \end{theorem}


  In general there is a relationship between the forward and backward
  hidden tail chains.\ When $k=1$ these are independent, but when
  $k>1$ and $t+s> k$, then $\Res_{t}$ is conditionally independent of
  $\Res_{-s}$ given $(\bm \Res_{-s+1\,:\,-1}, \bm\Res_{1\,:\,t-1})$.
  Hence, given any consecutive block of terms in the back-and forth
  hidden tail chain of size $k$, then the values before and after this
  block are independent.\ We remark that the precise dependence
  conditions between the forward and backward hidden tail chains have
  been given for the case where only $a_j(x)=x$ and $b_j(x)=1$ for all
  $j\neq 0$ by \cite{janssege14}.\ We focus on the forward hidden tail
  chain and do not address the inter-connections between the different
  $a_j(x)$ and $b_j(x)$ for positive and negative $j$.\
}

\section{Stochastic recurrence equations and dependence parameters}

\label{sec:SDE}
\subsection{Introduction}
The theory presented in Section~\ref{sec:th_res} comprises a
generalization of the theory presented in \cite{papaetal17} who dealt
only with first-order homogeneous Markov chains.\ Although working
with homogeneous chains embeds the theory in a rather broad setting,
it is impossible to explore the form of the results of
Theorems~\ref{thm:tailchain}--\ref{thm:tailchain_back}
unless further structure is imposed.\ In practice, standardization to
a common marginal form is typically performed on all marginal
distributions of the process via the probability integral transform.\
This approach is in the style of copula methods where the assumption
of identical margins is sufficient to identify the extremal dependence
structure of a random vector.\ Therefore, we will assume that all
one-dimensional marginal distributions of the Markov process are
standardized to unit-rate exponential random variables, as this is the
marginal choice with the clearest mathematical formulation.\ To this
end, we will assume that the Markov chain $\{X_t\}$ is stationary with
unit-exponential marginal distributions, that is,
$F_t(x)=\Pr(X_t \leq x) = (1-\exp(-x))_+$, for $t=0,1,\dots$, which
implies that the limit distribution $H_0$ in Assumption $\AAA$ is also
unit exponential.\ In Section~\ref{sec:SDE_AD} we consider the fully
asymptotic dependence case when $\chi_D>0$, so that $\chi_t > 0$ for
all $t\geq 1$.\ In Section~\ref{sec:SDE_AI} we consider the case where
the process is fully asymptotically independent, that is, $\chi_j=0$
for all $j \in D$, which also implies $\chi_t = 0$ for all $t\geq 1$.\
Intermediate cases are discussed in Example \ref{ex:alog} of Section
\ref{sec:examples}.

\cite{hefftawn04} found that for various copula models for a random
vector $\bm X_{0\,:\,d}$, $d\in \NN$, with exponentially tailed random
variables, the weak convergence of the conditional distribution of the
renormalized states
$\{\bm X_{1\,:\,d}-\bm a_{1\,:\,d}(X_0)\}/\bm b_{1\,:\,d}(X_0)$ given
$X_0>v$, to some distribution with non-degenerate margins, holds with
the normalization functions taking the simple form
$\bm a_{1\,:\,d}(v)=\bm \alpha_{1\,:\,d}\, v$ and
$\bm b_{1\,:\,d}(v) = v^{\beta}\bm 1_d$, where
$(\bm \alpha_{1\,:\,d},\beta) \in [0,1]^{d}\times[0,1)$.\ The
parameters $\alpha_t$ and $\beta$ have a simple interpretation and
control the strength of extremal association between variables $X_0$
and $X_t$, for $t=1,\dots,d$. Informally, in the presence of an
extreme event $X_0$ with $X_0 > v$ and $v$ sufficiently large, we may
then think of $X_t$ as $X_t = \alpha_t X_0 + X_0^\beta \,\Res_t$ where
$\Res_t$ arises from a non-degenerate distribution.\ Thus, $\alpha_t$
and $\beta$ are slope and scale parameters, respectively, with larger
values of $\alpha_t$ indicating stronger linear dependence between
$X_0$ and $X_t$ and with larger values of $\beta$ indicating a more
diffuse distribution for $X_t\mid X_0 > v$.\ In particular, when
$\alpha_t=1$ and $\beta=0$, then $X_0$ and $X_t$ are asymptotically
dependent so that $\chi_t >0$.\ When $\alpha_t\in [0,1)$ and
$\beta\in[0,1)$, then $X_0$ and $X_t$ are asymptotically independent
so that $\chi_t = 0$.\ An important special is when $X_0$ and $X_t$
are extremally independent so that no norming is needed, thus
$\alpha_t=0$ and $\beta=0$ as in Section \ref{sec:EI_chains}.\




\subsection{Fully asymptotically dependent Markov chains}
\label{sec:SDE_AD}
\begin{corollary}[\em Fully asymptotically dependent Markov chains]
  \label{cor:SDE_AD}
  \begin{description}[wide=1\parindent] Let $\{X_t\,:\,t=0,1,\hdots\}$
    be a $k$-th order stationary Markov chain with unit-exponential
    margins.\ Suppose that Assumption $\AA$ holds with $a_t(x)=x$ and
    $b_t(x)=1$, for $t=1,\dots,k-1$.\ Suppose further that $\AB$ holds
    with the function $a$ being non-zero and continuous, such that
    $\exp\{a( \log \bm x )\}$, $\bm x \in \RR^k_+$, is
    $1$--homogeneous, that is,
    \begin{equation} a(v\bm 1_k+ \bm y) - v = a(\bm y), \qquad
      \text{for all $v\in\RR$ and all $\bm y\in\RR^k$,}
      \label{eq:ams_condition}      
    \end{equation}
    and $a(\bm 0) \leq 0$.\ Then, the convergence \eqref{eq:htc} holds
    with $a_t(x)=x$ and $b_t(x)=1$ for $t\geq k$, and
    \begin{equation*} \Res_t = a(\bm \Res_{\tk}) + \varepsilon_t,
      \quad t=k,k+1,\dots,
    \end{equation*}
    for a sequence $\{\varepsilon_{t}\}_{t=k}^\infty$ of i.i.d. random
    variables with a distribution supported on $\RR$. Furthermore,
    $\mathbb{E}(Z_t) < 0$ for all $t\geq 1$.
  \end{description}
\end{corollary}
Although $a_t(x)=x$ and $b_t(x)=1$ for all $t\geq 1$, fully
asymptotically dependent Markov chains return to the body of the
distribution after witnessing an extreme event.\ This is due to the
negative drift of the tail chain, that is, $\mathbb{E}(\Res_t) < 0$
for all $t\geq 1$ which ensures that the Markov chain will return to
the body regardless of the behaviour of the norming functions.

\subsection{Fully asymptotically independent Markov chains}
\label{sec:SDE_AI}
\begin{corollary}[Fully asymptotically independent Markov chains with
  location and scale norming]
  \label{cor:SDE_AI}
  \begin{description}[wide=1\parindent] Let $\{X_t\,:\,t=0,1,\hdots\}$
    be a $k$-th order stationary Markov chain with unit exponential
    margins.\ Suppose Assumption $\AA$ holds with
    $a_t(v)=\alpha_t\,v$, $\alpha_t \in (0,1)$ and
    $b_t(v) = v^{\beta}$, $\beta \in [0,1)$, for $t=1,\dots,k-1$.\
    Suppose that Assumption $\AB$ holds with the function $a$ being
    $1$--homogeneous and the function $b$ being $\beta-$homogeneous
    when $\beta\in (0,1)$ and unity when $\beta = 0$.\
    Then, convergence \eqref{eq:htc}
    holds with $a_t(v)=\alpha_t \, v$ and $b_t(v)=v^{\beta}$, where
    \begin{IEEEeqnarray}{lCl} \alpha_t = a(\bm \alpha_{\tk}), \quad
      t=k,k+1,\dots,
      \label{eq:recurrence_slope}
    \end{IEEEeqnarray} and
    \begin{IEEEeqnarray}{rCl} \Res_t &=& \nabla a(\bm \alpha_{\tk})
      \cdot \bm \Res_{\tk} \,+\, b(\bm \alpha_{\tk})
      \,\varepsilon_t\qquad
      t=k,k+1,\ldots\label{eq:linear_tail_chain},
    \end{IEEEeqnarray} for a sequence $\{\varepsilon_t\}_{t=k}^\infty$
    of i.i.d. random variables from a non-degenerate distribution $K$
    on $\RR$.\ Additionally, if the function $a$ is twice continuously
    differentiable, $a(\bm 1_k)\neq 1$, and the spectral radius
    $\rho(\bm J_f(\bm 0))$ of the Jacobian matrix
    $\bm J_f\,:\,\RR^{k}\to \RR^{k\times k}$ of the map
    $\RR^k\ni\bm x_{1\,:\,k}\mapsto f(\bm x_{1\,:\,k})=(\bm
    x_{2\,:\,k},a(\bm x_{1\,:\,k}))$, is strictly less than unity,
    then
    \begin{equation}
      \alpha_t \to 0\quad\text{as $t\to\infty$}.
      \label{eq:alpha_convergence}
    \end{equation}
  \end{description}
\end{corollary}

{\color{black} Despite its relatively weak assumptions
  Corollary~\ref{cor:SDE_AI} provides considerable insight into the
  behaviour of the hidden tail chain.  It shows that the norming
  functions $a_t$, $t=k,k+1,\ldots,$ have a particularly neat
  structure, not least $a_t(X_0)=\alpha_t\,X_0$, where $\alpha_t$ is
  determined by the recurrence equation~\eqref{eq:recurrence_slope} of
  the $k$ previous values $\bm \alpha_{\tk}$ through the
  $1$--homogeneous function $a$.\ Also $\alpha_t\to 0$ as
  $t\to \infty$ leading eventually to no location norming in the
  limit, which is consistent with the independence case.  For a
  flexible parametric class of the function $a$, in
  Section~\ref{sec:autocovariance} we are able to explicitly solve the
  recurrence equation \eqref{eq:recurrence_slope} and find the form of
  a geometric decay to zero in $\alpha_t$ as $t$ increases.

  When $\beta=0$ the behaviour of the forward tail chain $t\to \infty$
  is almost entirely given by Corollary~\ref{cor:SDE_AI}. Here
  $X_t|\{X_0=v\} = \alpha_t \,v + Z_t +o_p(1)$ as
  $v\rightarrow \infty$.\ For fixed $t$, $X_t$ grows as
  $v\rightarrow \infty$.\ However, if we allow $t\rightarrow \infty$
  sufficiently fast with $v\rightarrow \infty$ the location tends to
  zero and $X_t$ converges to the process $\{Z_t\}$, which is a
  non-degenerate stationary autoregressive process. So, with such
  combined limiting operations, we have that $X_t\mid \{X_0=v\}$
  returns to the body of the distribution as $t\to \infty$, becoming
  independent of $X_0$. Here, if $\alpha_t\sim A^t$, as
  $t\rightarrow \infty$ where $A$ is a constant with $0<A<1$, i.e.,
  geometric decay, then we would need $t/\log(v)\rightarrow \infty$,
  as $v\rightarrow \infty$ for this result to hold.\ When $0<\beta<1$
  the limiting behaviour of the forward chain $X_t|\{X_0=v\}$ as
  $t\to \infty$ is only partially implied by
  Corollary~\ref{cor:SDE_AI}. This is because $Z_t\cinp 0$, since both
  the location and scale terms of $\epsilon_t$ in
  expression~\eqref{eq:linear_tail_chain} tend to zero, but its
  scaling $(v)^{\beta}$ tends to infinity.  Consequently, the limiting
  behaviour is determined by the relative speed of convergence of
  $\alpha_t\, v\rightarrow 0$ and $v^{\beta}\,Z_t\cinp 0$ if we link
  the growth rate of $t$ to that of $v$.}

In general we can view the recurrence relation in
expression~\eqref{eq:recurrence_slope} as the parallel of the
Yule--Walker equations and hence, term them the \textit{extremal
  Yule--Walker equations}.\ The Yule--Walker equations provide a
recurrence relation for the autocorrelation function in standard time
series that is used to determine the dependence properties of a linear
Markov process.\ For a $k$th order linear Markov process
$Y_t = \sum_{i=1}^k \phi_i\, Y_{t-k+i-1}+\eta_t$ with
$\{\eta_t\}_{t=-\infty}^\infty$ a sequence of zero mean, common finite
variance and uncorrelated random variables, where the set of
regression parameters $\phi_1,\ldots, \phi_k$ are real valued
constants such that the characteristic polynomial
$1-\phi_k\, z - \phi_{k-1}\, z^2- \cdots - \phi_1\, z^k\neq 0$ on
$\{z\in \mathbb{C}\,:\, \lvert z \lvert \leq 1\}$, the Yule--Walker
equations relate the autocorrelation function of the process
$\rho_t = \text{cor}(Y_{s-t}, Y_s)$ at lag $t$ with the regression
parameters $\phi_1,\ldots, \phi_k$ and the $k$ lagged autocorrelations
according to $\rho_t = \sum_{i=1}^k \phi_{i} \,\rho_{t-i}$,
$t \in \ZZ$. The sequence $\{\alpha_t\}$ has a similar structure for
extremes via recurrence~\eqref{eq:recurrence_slope}.

In Corollary \ref{cor:SDE_AI} we rule out the case $\beta < 0$,
considered by \cite{hefftawn04}, corresponding to the case where
location only normalization gives limits that are degenerate, with all
limiting mass at $\{0\}$. For simplicity, the theory developed in this
paper deals only with positive extremal association in Markov chains
and hence, in Corollary \ref{cor:SDE_AI}, the case $\alpha_t <0$
corresponding to the case where $X_0$ and $X_t$ exhibit negative
extremal association, is also ruled out.\ We note, however, that this
latter case can be easily accommodated by suitable transformations of
the marginal distributions, e.g., by standardizing margins to standard
Laplace distributions, see for example \cite{keefpaptawn13} and
\citet[Theorem 3,][]{papaetal17}.

\begin{corollary}[\em Fully asymptotically independent Markov chains
  with only scale norming]
  \label{cor:SDE_AI_scale}
  \begin{description}[wide=1\parindent] Let $\{X_t\,:\,t=0,1,\hdots\}$
    be a $k$-th order stationary Markov chain with unit-exponential
    margins.\ Suppose that Assumption $\BA$ holds with
    $b_t(x) = x^{\beta}$, $\beta \in (0,1)$, for $t=1,\dots,k-1$.\
    Suppose further that Assumption $\BB$ holds with the function $b$
    being continuous and $\beta-$homogeneous.\ Then convergence
    \eqref{eq:htc_scale} holds with $b_t(x)=x^{\beta_t}$,
    $\beta_t\in (0,1)$, $t\geq k$, where $\beta_t$ satisfies the
    recurrence relation
    $\log \beta_t = \log \beta +
    \log\left(\max_{i=1,\ldots,k}\beta_{t-i}\right)$.\ This gives the
    solution
    \begin{IEEEeqnarray*}{rCl} \log \beta_t &=& (\lfloor 1 +
      (t-1)/k\rfloor) \log \beta, \qquad t\geq k,
    \end{IEEEeqnarray*} where $\lfloor x \rfloor$ denotes the integer part
    of $x$.\ It follows that $\beta_t \to 0$ as $t\to \infty$. Also, for
    $t\geq k$ we have
    \[ \Res_{t} =
      \begin{cases} b(\Res_{t-k}, \bm 0_{k-1}) \,\varepsilon_{t} &
                                                                   \text{when $\text{mod}_k(t)=0$}\\ b(\Res_{t-k},\hdots, \Res_{t-1})
        \,\varepsilon_{t} & \text{when $\text{mod}_k(t) = 1$}\\
        b(\Res_{t-k},\hdots, \Res_{t-j}, \bm 0_{j-1}) \,\varepsilon_{t} &
                                                                          \text{when $\text{mod}_k(t) = j \in \{2,\ldots,k-1\}$},
      \end{cases}
    \] for a sequence $\{\varepsilon_t\}_{t=k}^\infty$ of
    i.i.d. random variables with distribution supported on $\RR_+$ and
    $\Res_0=1$ a.s..
  \end{description}
\end{corollary} As with Corollary \ref{cor:SDE_AI} we find a form of
geometric decay in the dependence parameters $\beta_t$ as $t$
increases, leading eventually to extremal independence ($\alpha_t=0$
and $\beta_t\to 0$) in the limit as $t\to \infty$, so that $X_t$
returns to the body of the distribution as it becomes independent of
$X_0$.\ However, in contrast to Corollary~\ref{cor:SDE_AI} where the
location parameter $\alpha_t$ changed with $t$, here it is the power
parameter $\beta_t$ of the scale function.\ In particular, $\beta_t$
decays geometrically to $0$ stepwise, with steps at every $k$ lags. At
time $t$ the resulting hidden tail chain depends only on the last $j$
values, with $j = \text{mod}_k(t)$.

\section{A class of recurrence relations for dependence parameters in
  asymptotically independent Markov chains with closed form solutions}
\label{sec:autocovariance}
The results of Section~\ref{sec:SDE} provide insight into the form of
the norming and updating functions of Theorems~\ref{thm:tailchain} and
\ref{thm:tailchain:nonneg}, not least for asymptotically independent
Markov chains where $(\alpha_t,\beta_t)\neq (1,0)$ for all $t > 0$.\ A
precise formulation of the location and scale parameters $\alpha_t$
and $\beta_t$ for $t\geq k$, however, depends on the form of
functionals $a(\cdot)$ and $b(\cdot)$ which is opaque even when these
are assumed to be homogeneous functionals.\ Motivated by examples
considered in Section~\ref{sec:examples}, here we give an explicit
characterization of the solution to the extremal Yule--Walker
equations~\eqref{eq:recurrence_slope} in Corollary~\ref{cor:SDE_AI}
for a parsimonious parametric subclass of $1$--homogeneous functionals
which imbeds many of the examples of Section
\ref{sec:examples_section}.

Consider the $1$--homogeneous function
$a_{\mathcal{M}}:\RR_+^k \to \RR_+$ defined by
\begin{equation} a_{\mathcal{M}}(\bm x) = c\,\left\{\gamma_1\,
    (\gamma_1 \, x_1)^{\delta} + \cdots + \gamma_k\, (\gamma_k
    \,x_k)^{\delta}\right\}^{1/\delta},  \quad\delta \in
  \overline{\RR}, \quad(\gamma_1,\ldots, \gamma_k) \in \SI_{k-1},
  \quad 0< c < \left(\sum_{i=1}^k
    \gamma_{i}^{1+\delta}\right)^{-1/\delta}.
  \label{eq:model}
\end{equation}
The functional $a_{\mathcal{M}}$ is continuous in
$\delta \in \overline{\RR}$ with
\begin{equation} \lim_{\delta \to 0} a_{\mathcal{M}}(\bm x ) = c_0\,
  x_1^{\gamma_1}\,x_2^{\gamma_2}\cdots \,x_k^{\gamma_k}, \qquad
  c_0=c\,\prod_{i=1}^n\gamma_i^{\gamma_i} < 1. 
  \label{eq:deltacont}
\end{equation} and $\lim_{\delta \to \pm \infty} a_{\mathcal{M}}( \bm x ) = a_{{\mathcal{M}},\pm
  \infty}\left(\gamma_1\,x_1, \gamma_2\,x_2,\hdots ,
  \gamma_k\,x_k\right) $ where
$a_{\mathcal{M}, \infty}(\bm x)= c_{\infty}\,\max_{i=1,\ldots,k} x_i$,
$a_{\mathcal{M},-\infty}(\bm x)=c_{-\infty}\min_{i=1}^k x_i$,
$c_{\infty}=c\,\max_i\gamma_i$, $c_{-\infty}=c\,\min_{i\,:\,\gamma_i > 0}\gamma_i$, where $0 < c_{-\infty}< 1$ and $0 < c_{-\infty}<1$.

\begin{proposition}
  \label{prop:closedform} Consider the function $a_{\mathcal{M}}$
  defined by expression~\eqref{eq:model}.\ Suppose that the $s\in\NN$
  distinct (possibly complex) roots of the characteristic polynomial
  \begin{equation*} x^k - c^\delta\, \gamma_k^{1+\delta}\,
    x^{k-1}-\cdots-c^\delta\,\gamma_1^{1+\delta} = 0
  \end{equation*} are $r_1,\hdots,r_s$ with multiplicities
  $m_1,\hdots,m_s$, $\sum_i m_i = k$.\ Then the solution of the
  recurrence relation~\eqref{eq:recurrence_slope} with $a(\bm x) =
  a_{\mathcal{M}}(\bm x)$ for all $\bm x\in\RR_+^k$, subject to the
  initial condition $(\alpha_1,\hdots,\alpha_{k-1}) \in (0,1)^{k-1}$, is
  \begin{equation} \alpha_t =\left(\sum_{i=1}^s ( C_{i0} + C_{i1}\,t +
      \cdots + C_{i,m_i-1} t^{m_i-1}) \,r_i^t \right)^{1/\delta}
    \qquad \text{for $t=k,k+1,\hdots$}
    \label{eq:soln_mult}
  \end{equation} where the constants $C_{i0},\hdots,C_{i,m_i-1}$,
  $i=1,\hdots,s$, are uniquely determined by the initial condition via
  the system of equations
  \[ \alpha_t = \left(\sum_{i=1}^s ( C_{i0} + C_{i1}\,t + \cdots +
      C_{i,m_i-1} t^{m_i-1}) \,r_i^t \right)^{1/\delta} \qquad \text{for
      $t=0,\hdots,k-1$},
  \] with $\alpha_0=1$.
\end{proposition}
From Corollary~\ref{cor:SDE_AI}, it follows that the sequence
$\{\alpha_t\}$ in Proposition \ref{prop:closedform} satisfies
$\alpha_t\to 0$ as $t\to \infty$.\ Let
$I_r=\{I\in\{1,\dots,s\}\,:\,|r_I|=\max_{i=1,\dots,s}|r_i|\}$. Under
the assumption that $|I_r|=1$, then we have that $\alpha_t$ in
expression \eqref{eq:soln_mult} satisfies
$\alpha_{t} \sim C_{I,m_I-1}\,t^{(m_I-1)/\delta} (r_I^{1/\delta})^t\to
0$, $I\in I_r$, $\delta \in \RR \sm \{0\}$, as $t\to
\infty$.\ 
\begin{remark} Although solution \eqref{eq:soln_mult} holds for any
  $\delta \in \overline{\RR}$, it is not evident what form the
  solution takes when $\delta=0$ or when $\delta=\pm\infty$.\ These
  cases are considered separately below.
  \begin{description}[wide=0\parindent]
    \item[Case $\delta \to 0$:] A logarithmic transformation in
limit~\eqref{eq:deltacont} results in the linear nonhomogeneous
recurrence relation
    \[ \log \alpha_t - \gamma_1 \log \alpha_{t-1} - \cdots - \gamma_k
\log \alpha_{t-k} = \log c-I(\bm \gamma)
    \] where $I(\bm\gamma)=-\sum_{i=1}^k\gamma_i\log \gamma_i$.\
Suppose that the $s\in\NN$ distinct (possibly complex) roots of the
characteristic polynomial
    \begin{equation} x^k - \gamma_k\, x^{k-1}-\cdots- \gamma_1 = 0,
    \end{equation} are $r_1,\hdots,r_s$ with multiplicities
$m_1,\hdots,m_s$, $\sum m_i=k$.\ Then the solution of recurrence
\eqref{eq:recurrence_slope} is
    \[ \alpha_t = \exp\left\{\sum_{i=1}^s \left( C_{i0} + C_{i1}\,t +
\cdots + C_{i,m_i-1} t^{m_i-1}\right) \,r_i^t + \frac{c- I(\bm
\gamma)}{\gamma_1+ 2\gamma_2 + \cdots + k\gamma_k} t\right\} \quad
\text{for $t=k,k+1,\hdots$}
    \] where the constants $C_{i0},\hdots,C_{i,m_i-1}$,
$i=1,\hdots,s$, are uniquely determined by the system of equations
    \[ \alpha_t = \exp\left\{\sum_{i=1}^s ( C_{i0} + C_{i1}\,t +
\cdots + C_{i,m_i-1} t^{m_i-1}) \,r_i^t+ \frac{c- I(\bm
\gamma)}{\gamma_1+ 2\gamma_2 + \cdots + k\gamma_k} t \right\} \quad
\text{for $t=0,\hdots,k-1$},
    \] with $\alpha_0=1$ and $\alpha_t \in (0,1)$ for $t=1,\ldots,
k-1$.
    
  \item[Case $\lvert\delta\rvert \to\pm \infty$:] Using forward
substitution, we have that for $\delta \to \infty$ the solution of
\eqref{eq:recurrence_slope} is
    \begin{equation} \alpha_{t}=c^t \,\max_{i=1,\ldots, k}(c_{t-i}
\,\alpha_{i-1}),\qquad c_{t-i}=\max \prod_{n=1}^k \gamma_{k+1-n}^{j_n}
\qquad t\geq k
      \label{eq:max_recursive}
    \end{equation} where the maximum for $c_{t-i}$ in
expression~\eqref{eq:max_recursive} is taken over $0\leq j_1\leq
\ldots \leq j_k\leq t-i$ such that $\sum_{m=1}^{t-i} m\,j_m = t-i$.\
The case $\delta \to -\infty$ is obtained by replacing the maximum
operator in expression \eqref{eq:max_recursive} by the minimum
operator.
  \end{description}
\end{remark} Although Corollary~\ref{cor:SDE_AI} assumes $a$ is
differentiable, it is useful to understand the behaviour of the
solution of the recurrence relation \eqref{eq:recurrence_slope}, for
$a=a_{\mathcal{M}}$, with $a_{\mathcal{M}}$ given in equation
\eqref{eq:model}, as $\lvert \delta \rvert \to \infty$.\ Recurrence
relationship \eqref{eq:max_recursive} shows that even when function
$a$ is non-differentiable, a simple recurrence is possible.\
\section{Results for kernels based on important copula classes}
\label{sec:examples_section}
\subsection{Strategy for finding norming functionals}
\label{sec:strategy} In Section~\ref{sec:examples}, we consider
examples of $k$th order Markov chains that are covered by the proposed
theory.\ In order to obtain the hidden tail chain for a given Markov
process one needs to derive appropriate norming functions $a_t, b_t$,
for time lags $t=1,\ldots, k-1$ and norming functionals $a$ and $b$ to
ensure that after an extreme event at time $t=0$, the joint
distribution of the renormalized $k$ states in the process, converges
weakly to a non-degenerate limit.\ The methods for finding $a_t, b_t$
for $t=1, \ldots ,k-1$ for all our examples is given by Theorem 1 in
\cite{hefftawn04}.\ To obtain the full sequence of norming functions
$a_t, b_t$, for $t=1,2,\ldots$, we need a new strategy that allows us
to also derive the functionals $a$ and $b$.\ We explain the strategy
here for the case $a\neq 0$ and note that the case $a=0$ and
$b \neq 1$ is handled in similar manner.

Assuming the conditional distribution of $X_k\mid \bm X_{0\,:\,k-1}$
admits a Lebesgue density almost everywhere, a similar argument as in
the proof of Theorem 1 in \cite{hefftawn04} guarantees that the
functionals $a$ and $b$ can be identified, up to type, by
\begin{equation} \lim_{\level\to \infty} \PR\left(X_k < a(\bm
    X_{\kk})~\Big |~ \bm X_{\kk}=\bm A_t(\level, \bm \res) \right) = p \in
  (0,1), \qquad \text{for all $\bm z\in\RR^k$},
  \label{eq:kernel_a}
\end{equation} and
\begin{dmath} b(\bm A_t(\level, \bm \res))=\frac{\PR\left(X_k > a(\bm
      X_\kk)~\big |~ \bm X_{\kk}=\bm A_t(\level,\bm
      \res)\right)}{\left[d\,\PR\left\{X_k \le y~\big|~ \bm
        X_{\kk}=\bm A_t(\level,\bm \res)\right\}/dy \right]|_{y=a(\bm
      A_t(\level, \bm \res))}},
  \label{eq:hazard}
\end{dmath} where
$\bm A_t(\level,\bm \res)=\bm a_{\tk}(\level)+\bm b_{\tk}(\level)
\,\bm \res$ with $\bm a_{\tk}$ and $\bm b_{\tk}$ as in Assumption
$\AB$.\ Expression \eqref{eq:hazard} can be cumbersome to use in
practice so we resort to asymptotic inversion in order to identify
$\bfunD$. In particular, to find a representative form for $\bfunD$ we
make an informed choice based on the leading order terms in an
asymptotic expansion of the conditional distribution in expression
\eqref{eq:kernel_a} to obtain
\begin{equation} \PR\left(X_k < a(\bm X_{\kk})+b(\bm
    X_{\kk})\,y~\Big |~ \bm X_{\kk}=\bm A_t(\level,\bm \res) \right) \wk
  K[-\infty,y], \qquad x\in\RR, \qquad \text{for all $\bm z\in\RR^k$}
  \label{eq:kernel}
\end{equation} where $K$ is a non-degenerate distribution on $\RR$.\
This strategy is illustrated step by step in examples 1--3 of
Section~\ref{sec:examples}, where the identification of leading order
terms is straightforward.
\subsection{Preliminaries}
\label{sec:prel}
To illustrate the results in Theorems~\ref{thm:tailchain}
and~\ref{thm:tailchain:nonneg}, we study the extremal behaviour of
$k$th order stationary Markov chains with unit exponential margins,
with transition probability kernels for the copula of $k+1$
consecutive values given in Section~\ref{sec:examples}. The examples
cover both extremal dependence types, not least they include two
subclasses of asymptotically dependent max-stable
distributions---namely those with logistic and H\"{u}sler--Reiss
dependence---and two classes of asymptotically independent
distributions---namely the Gaussian copula and the inverted max-stable
distribution with logistic dependence. We also consider an example of
a transition probability kernel for a second-order Markov chain using
max-stable distribution which exhibits a mixture of asymptotic
independence and asymptotic dependence over different lags, and does
not satisfy the assumptions of the theory developed in
Section~\ref{sec:th_res}.\ The theory which motivates these as copulae
does not matter here, we simply view them as a range of interesting
and well known copula families for which we study their extremes in a
Markov setting.\

Let $F$ denote the joint distribution function of a random vector
$\bm X=(X_0,\ldots,X_{k})$, assumed to be absolutely continuous
w.r.t. Lebesgue measure with unit exponential margins, that is,
$F_i(x)=F_E(x)=(1-\exp(-x))_+$, $i=0,\ldots,k$.\ Further, let
$[k]=\{0,1,\ldots,k\}$ and
$\mathscr{P}([k])=2^{[k]} \sm \{\emptyset\}$.\ The construction
of all Markov processes studied in this section, is summarised as
follows. Writing $C\,:\,[0,1]^{k+1}\to [0,1]$ for the copula of
$\bm X$, that is,
$C(\bm u)=F(F_E^\leftarrow(u_0),\ldots,F^\leftarrow_E(u_k))$, where
$\bm u=(u_0,\ldots,u_k) \in [0,1]^{k+1}$, we define the Markov kernel
$\pi_E\,:\, \borel(\RR^k)\to [0,1]$ of the stationary process by
\begin{IEEEeqnarray*}{rCl} \pi_E(\bm x_{\kk},
  x_k)&=&\left[\frac{\partial^{k}} {\partial u_{0} \cdots \partial
      u_{k-1}}\,C(\bm u_{\kk}, u_{k})\Big/ \frac{\partial^{k}}
    {\partial u_{0}\cdots \partial u_{k-1}}\,C(\bm u_{\kk},
    1)\right]\Bigg|_{\bm u_{\kplusonekplusone}=\bm
    v_{\kplusonekplusone}},
\end{IEEEeqnarray*} where
$\bm v_{\kplusonekplusone}=\{1-\exp(-\bm x_{\kplusonekplusone})\}_+$.\
Assuming the copula function satisfies appropriate conditions that
ensure stationarity \citep{joe15}, then the initial distribution
$F(\bm x_{\kk},\infty)$ is $k$ dimensional invariant distribution of a
Markov process with unit exponential margins and kernel $\pi_E$.\
    
In what follows, we set up the notation for the associatede transition
probability kernels that we study in Section~\ref{sec:examples} where
we derive norming functions and hidden tail chains and impose
conditions that ensure stationarity of the Markov chain in each
example.
\begin{description}[wide=0\parindent]      
\item[Gaussian copula:] Our first example concerns stationary Gaussian
  autoregressive processes with positive dependence transformed
  componentwise to have exponential marginal distributions.\ Let
  $\bm \Sigma \in \RR^{(k+1) \times (k+1)}$ be a $(k+1)$-dimensional
  Toeplitz correlation matrix, that is,
  $\bm \Sigma = (\rho_{\lvert i-j \rvert})_{1 \leq i,j \leq k+1}$ with
  $\rho_0=1$, $\rho_i > 0$, for $i=1,\ldots,k$, assumed to be positive
  definite.\ The distribution function of the standard
  $(k+1)$-dimensional Gaussian with mean $\bm 0_{k+1}$ and variance
  $\bm \Sigma$, in exponential margins, is
  \[ F(\bm x_{\kplusonekplusone})=
    \int_{-\bm\infty}^{\Phi^{\leftarrow}\{1-\exp(-\bm
      x_{\kplusonekplusone})\}} \frac{\lvert \bm Q
      \lvert^{1/2}}{(2\pi)^{k/2}}\, \exp\left(- \bm s^{\top}\,\bm Q\,
      \bm s/2\right)~\mathrm{d}\bm s \qquad \bm s=(s_0,\ldots,
    s_{k})^\top, \bm x_{\kplusonekplusone} \in \RR^{k+1},
  \] where $\Phi^{\leftarrow}\,:\,[0,1]\to \RR$ denotes the quantile
  function of the standard normal distribution function $\Phi(\cdot)$
  and $\bm Q=\bm \Sigma^{-1} = (q_{i-1,j-1})_{1\leq i,j \leq k+1}$ is
  a symmetric positive definite matrix.\ This joint distribution gives
  the transition probability kernel
  \begin{IEEEeqnarray*}{rCl} \pi_E(\bm x_{\kk}, x_k) &=&
    \pi_G\left\{\Phi^\leftarrow\{1-\exp(-\bm x_{\kk})\},
      \Phi^\leftarrow\{1-\exp(-x_{k})\}\right\}, (\bm x_{\kk}, x_k) \in
    \RR^{k} \times \RR
  \end{IEEEeqnarray*} where the kernel $\pi_G$ is the full conditional
  distribution function of the multivariate normal given by
  \[ \pi_G(\bm x_{\kk}, x_k) = \Phi\left[q_{kk}^{1/2}\left\{x_k-
        \sum_{t=0}^{k-1}\left(-\frac{q_{tk}}{q_{kk}}\right)\,x_t
      \right\}\right].
  \] The condition $\rho_i > 0$, for $i=1,\ldots, k$ appears
  restrictive but is made to simplify the presentation.\ If we worked
  with standard Laplace marginals, instead of exponential marginals,
  as say in \cite{keefpaptawn13}, the presentation would be equally
  simple for any values $\lvert \rho_i \rvert > 0$, $i=1,\ldots,k$, of
  the correlation matrix $\bm \Sigma$.
\item[Max-stable copula:] A class of transition probability kernels
  for asymptotically dependent Markov processes is obtained from the
  class of multivariate extreme value distributions \citep{resn87}. The
  $k+1$ dimensional distribution function of the multivariate extreme
  value distribution with exponential margins is given by
  \begin{equation} F(\bm x_{\kplusonekplusone}) = \exp(-V(\bm
    y_{\kplusonekplusone})), \quad \text{where
      $\bm y_{\kplusonekplusone} = T(\bm x_{\kplusonekplusone}):=
      -1/\log(1- \exp(-\bm x_{\kplusonekplusone}))$}, \quad \bm
    x_{\kplusonekplusone} \in \RR^{k+1}_+,
    \label{eq:MEV}
  \end{equation} with $V\,:\,\RR_+^{k+1}\to \RR_+$ a $-1$-homogeneous
  function, known as the exponent measure, given by
  \begin{equation} V(\bm y_{\kplusonekplusone})=\int_{\SI_{k}}
    \max_{i=0, \ldots, k}\left(\frac{\omega_i}{y_i}
    \right)~H(\mathrm{d}\bm \omega),
    \label{eq:exp_measure}
  \end{equation} where $H$ is a Radon measure on $\SI_k$
  that has total mass
  $k+1$ and satisfies the moment constraints $\int_{\SI_{k}} \omega_i
  H(\mathrm{d}\bm \omega) = 1$, for $i =
  0,\ldots,k$.\ Throughout this section, we assume that
  $V$ has continuous mixed partial derivatives of all orders which
  ensures that a density for
  $F$ exists \citep{coletawn91}.\ For any $J \subseteq
  [k]$, we write
  $V_J$ to denote the higher-order partial derivative
  $\partial^{\lvert J \lvert} V(\bm x_{\kplusonekplusone})/\prod_{j
    \in J}\partial x_j$ and
  $\Pi_{m}$ for the set of partitions of $[m]$, with
  $m=0,\ldots,k$.\ Furthermore, for a vector $\bm z = (\bm
  x_{0\,:\,m}, \bm x_{m+1\,:\,k})$, we write $V(\bm x_{0\,:\,m}, \bm
  x_{m+1\,:\,k})=V(\bm z)$.  For $m = 0,\ldots, k-1$ we define $V(\bm
  x_{0\,:\,m}, \,\infty \cdot \bm 1_{k-m}):=\lim_{\bm x_{m+1\,:k} \to
    \infty \cdot \bm 1_{k-m}}V(\bm x_{0\,:\,m}, \bm x_{m+1\,:\,k})$,
  and for $J\subseteq [m]$, we define
  $V_{J}(\bm x_{0\,:\,m}, \infty\cdot\bm 1_{k-m}):=\partial^{\lvert J
    \rvert}\,V(\bm x_{0\,:\,m}, \infty \cdot \bm 1_{k-m})/\prod_{j\in
    J}\partial x_j$.
  Stationarity is achieved by requiring that the
  distributions of $\{X_i\,:\,i\in A\}$ and $\{X_i\,:\,i\in B\}$ are
  identical for any set $B$ that is a translate of the set $A$, that is,
  when there exists a unique $\omega \in \ZZ$ such that
  $B = \{x + \omega\,:\, x\in A\}$.\ Hence, to ensure stationarity in
  the time series outlined in Examples \ref{ex:ims}--\ref{ex:alog}, we
  assume for the $(k+1)$-variate exponent measures associated to
  multivariate extreme value and inverted max-stable copula models
  that
  \begin{equation} \lim_{\bm x_{[k]\sm A} \to \infty} V(\bm
    x)\,\Big|_{\bm x_A=\bm y} = \lim_{\bm x_{[k]\sm B} \to
      \infty} V(\bm x)\,\Big|_{\bm x_B=\bm y}, \qquad \bm y\in
    \RR_+^{\lvert A \lvert}.
    \label{eq:statV}
  \end{equation} whenever $B$ is a translate set of $A$, with $A, B
  \subseteq [k]$.  The transition probability kernel induced by the
  multivariate extreme value copula, in exponential margins, is
  \begin{equation} \pi_E(\bm x_{\kk}, x_k) = \frac{\left[ \sum_{p \in
          \Pi_{k-1} }(-1)^{\lvert p \lvert} \prod_{J \in p} V_J(\bm
        y_{\kplusonekplusone})\right]} {\left[\sum_{p \in
          \Pi_{k-1}}(-1)^{\lvert p\lvert} \prod_{J \in p} V_J(\bm
        y_{\kk}, \infty)\right]} \exp\left\{V(\bm y_{\kk}, \infty) -
      V(\bm y_{\kplusonekplusone})\right\}
    \label{eq:mstranskernel}
  \end{equation} where $(\bm x_{\kk}, x_k) \in \RR^{k} \times \RR$ and
  with $\bm y_{\kplusonekplusone}$ as defined in
  expression~\eqref{eq:MEV}.
\item[Inverted max-stable:] Lastly, the final class of transition
  kernels is based on the class of inverted max-stable distributions
  \citep{ledtawn97, papatawn16}.\ The specification of this
  distribution is most elegantly expressed in terms of its
  $(k+1)$-dimensional survivor function. In exponential margins, this
  is expressed as
  \begin{equation} \overline{F}(\bm x_{\kplusonekplusone}) =
    \exp(-V(1/\bm x_{\kplusonekplusone})),
    \label{eq:inverted_ms}
  \end{equation} where $V$ denotes an exponent measure as defined by
  expression \eqref{eq:exp_measure}. To ensure stationarity, $V$ is
  assumed to satisfy conditions \eqref{eq:statV}.\ This distribution
  gives the transition probability kernel
  \begin{IEEEeqnarray}{rCl} \pi^{\text{inv}}(\bm x_{\kk}, x_k) &=& 1-
    \pi_E[-\log\{ 1- \exp(-\bm x_{\kk})\}, -\log\{ 1- \exp(-1/x_k)\}],
    \label{eq:ims_kernel}
  \end{IEEEeqnarray} where $(\bm x_{\kk}, x_k) \in \RR^{k} \times \RR$
  and $\pi_E$ as given by equation~\eqref{eq:mstranskernel}.
\end{description}

\subsection{Examples}
\label{sec:examples} We illustrate examples for a range of $k$th order
Markov processes and show how they all fit with the theory
developed. Central to all examples is the weak convergence of the
renormalized initial distribution and the renormalized transition
probability kernel.\ Under suitable regularity conditions, this
permits the complete characterisation of the hidden tail chain.\ Our
proofs for the weak convergence of each transition probability kernel
are presented in
Appendices~\ref{sec:kernel_MVN}--\ref{sec:HR_convergence} where we
implement step by step the strategy we outlined in
Section~\ref{sec:strategy}.\ The behaviour of each (hidden) tail chain
is illustrated in Figure \ref{fig:tail_chains} using simulation, for
specific examples of the classes of processes we cover.

Example 4 is an example of a Markov process not covered by the theory
developed thus far which exhibits irregular behaviour; this new
behaviour permits the possibility of sudden switches from extreme to
non-extreme states and vice versa. In this setting, a novel form of
normalization of the transition probability kernel is required which,
together with the associated hidden tail chain, carries information
about the mechanism that governs the sudden transitions. Although the
development of general theory for this type of process is beyond the
scope of this paper, for our example we derive the hidden tail chain
and illustrate its behaviour in Figure \ref{fig:alog}. We only mention
in passing that for this example, the strategy that is implemented is
similar to the strategy presented in Section \ref{sec:strategy}.
\begin{example}[\em Stationary Gaussian autoregressive
  process---positive dependence]
  \label{ex:gaussian} \normalfont \citet[Section~8.6]{hefftawn04}
  showed that Assumption $\AA$ holds with norming functions
  $a_i(v)=\rho_i^2 v$, $b_i(v)=v^{1/2}$, that is,
  $\alpha_i = \rho_i^2$ and $\beta_i=1/2$, for $i=1,\ldots,k-1$ and
  initial limiting distribution
  $G(\bm \res_\kkminone)=\Phi_k(\bm \res_\kkminone; \Sigma_0)$,
  $\bm \res_\kkminone \in \RR^{k-1}$, where $\Phi_k(\cdot;\bm \Sigma_0)$ denotes
  the cumulative distribution function of the $k$-dimensional
  multivariate normal distribution with mean vector zero and
  covariance matrix
  $\Sigma_0=(2 \rho_{i}\,\rho_{j}(\rho_{\lvert j-i\lvert}-
  \rho_i\,\rho_j))_{1\leq i,j\leq k-1}$.
  \noindent Appendix \ref{sec:kernel_MVN} shows Assumption $\AB$ holds
  with norming functionals
  \begin{IEEEeqnarray}{lll} a(\bm u)=\Big(\sum_{i=1}^{k} \phi_i \,
    u_i^{1/2}\Big)^2, &\qquad b(\bm u)=a(\bm u)^{1/2}, & \qquad \bm
    u=(u_1,\ldots,u_k)\in \RR_+^k,
    \label{eq:a_gaussian}
  \end{IEEEeqnarray} where $\phi_{i} = -q_{k-i,k}/q_{kk}$,
  $i=1,\ldots, k$ denote the first $k$ partial autocorrelation
  coefficients of the stationary Gaussian process (on Gaussian
  margins), and the transition probability kernel of the renormalized
  Markov chain converges weakly to the distribution
  \begin{equation} K(x)=\Phi\left\{(q_{kk}/2)^{1/2}\,x \right\},
    \qquad x\in \RR.
    \label{eq:limit_Gaussian_kernel}
  \end{equation} Corollary~\ref{cor:SDE_AI} asserts that a suitable
  location normalization after $t\geq k$ steps has $\alpha_t=a(\bm
  \alpha_{\tk}) = \rho_t^2$ and $\beta_t={1/2}$, with $\rho_t=
  \sum_{i=1}^k \phi_i \, \rho_{t-i}$, for $t\geq k$. This leads to the
  scaled autoregressive tail chain
  \begin{equation} \Res_t = \rho_t \, \sum_{i=1}^k
    \frac{\phi_i}{\rho_{t-i}} \, \Res_{t-i} + \rho_t \, \varepsilon_t
    \qquad t \geq k,
  \end{equation} and $\{\varepsilon_t\}_{t=k}^\infty$ is a sequence of
  i.i.d. random variables with distribution $K$ given by
  expression~\eqref{eq:limit_Gaussian_kernel}.

  The hidden tail chain is a non-stationary $k$th order autoregressive
  Gaussian process with zero mean and auto-covariance function
  $\text{cov}(\Res_{t-s}, \Res_t) = (2 \rho_{t-s}\,\rho_{t}(\rho_{s}-
  \rho_{t-s}\,\rho_{t}))$ when $t\neq s$.\ The variance of the process
  satisfies $\text{var}(\Res_t)=\mathcal{O}(\rho_t^2)$ as
  $t\to \infty$, hence showing that the process degenerates to 0 in
  the limit as $t\to \infty$.\ This long-term degenerative behaviour
  is shown for a special case of this hidden tail chain in panel $(a)$
  of Figure~\ref{fig:tail_chains}.
\end{example}

\begin{remark}  The location functional $a$ in
expression~\eqref{eq:a_gaussian} can be written in form
\eqref{eq:model} with $c = \left(\sum_{i=1}^k\phi_i^{2/3}\right)^3$
and $\gamma_i = \phi_i^{2/3}/\sum_{i=1}^k \phi_i^{2/3}$.
\end{remark}

\begin{example}[\em Inverted max-stable copula with logistic
  dependence]
  \label{ex:ims} \normalfont Consider a stationary $k$th order Markov
chain with a $(k+1)$-dimensional survivor function
\eqref{eq:inverted_ms} and exponent measure of logistic type given by
      \begin{IEEEeqnarray}{rCl} V(\bm y_{\kplusonekplusone})&=& \lVert
\bm y_{\kplusonekplusone}^{-1/\alpha}\rVert^\alpha \qquad \bm
y_{\kplusonekplusone} \in\RR_+^d, \label{eq:logistic}.
\end{IEEEeqnarray} where $\alpha \in (0,1)$.\
\citet[Section~8.5]{hefftawn04} showed that Assumption $\BA$ holds
with $b_i(v)=v^{1-\alpha}$, that is, $\beta_i=1-\alpha$ for
$i=1\ldots, k-1$, and limiting initial distribution
$G(\bm \res)= \prod_{i=1}^{k-1}\{1-\exp(-\alpha \res_i^{1/\alpha})\}$,
$\bm \res \in (0,\infty)^{k-1}$.\ Appendix~\ref{sec:kernel_inv_log}
shows that Assumption $\BB$ holds with normalizing functionals
  \begin{IEEEeqnarray}{ll} a(\bm u)= 0,\qquad b(\bm u)=\lVert \bm
u^{1/\alpha} \rVert^{\alpha\,(1-\alpha)},& \qquad \bm
u=(u_1,\ldots,u_k)\in\RR_+^k,
    \label{eq:ab_invlog}
  \end{IEEEeqnarray} and the transition probability kernel of the
  renormalized Markov chain converges weakly to the distribution
  \[
    K(x) = 1-\exp(-\alpha \,x^{1/\alpha}), \qquad x \in (0,\infty),
  \]
  as $u \to \infty$.\ Corollary~\ref{cor:SDE_AI_scale} asserts that a
  suitable normalization after $t\geq k$ steps is $a_t(v)=0$,
  $\log b_t(v)=\left({(1-\alpha)^{1 + \lfloor (t-1)/k\rfloor}}\right)
  \, \log v$, which leads to the scaled random walk hidden tail chain
  \[
    \Res_{t} = \begin{cases}
                 \lVert (\Res_{t-k}, \bm 0_{k-1})^{1/\alpha}\rVert^{\alpha (1-\alpha)} \,\varepsilon_{t} & \text{when $\text{mod}_k(t)=0$}\\
                 \lVert\bm \Res_{t-k\,:\,t-1}^{1/\alpha}\rVert^{\alpha (1-\alpha)} \,\varepsilon_{t} & \text{when $\text{mod}_k(t)=1$}\\
                 \lVert(\bm \Res_{t-k\,:\,t-j}^{1/\alpha}, \bm 0_{j-1})\rVert^{\alpha
                 (1-\alpha)} \,\varepsilon_{t} & \text{when
                                                 $\text{mod}_k(t) = j\in\{2,\ldots,k-1\}$},
               \end{cases}
             \]    
             where $\{\varepsilon_t\}_{t=k}^\infty$ is a sequence of
             i.i.d. random variables with distribution $K$.

    This hidden tail chain is a non-stationary process, specifically,
    after a logarithmic transfomation, it is a non-stationary
    non-linear $k$th order autoregressive process.\ The first element
    of the process is $\Res_0=1$ a.s., the next $k-1$ elements of the
    process $\bm \Res_{\kkminone}$ are i.i.d. positive random variables
    with distribution function $K(x)$, $x > 0$.\ Subsequent elements
    $\Res_t$, for $t\geq k$, have distributions that vary both in mean
    and variance.\ We see that for $b$ given by
    expression~\eqref{eq:ab_invlog} and any $\bm x_{1:k} \in \RR^k_+$,
    then
    $b(\bm x_{1:k}) > b( \bm x_{2:k}, 0) > \cdots > b(x_k, \bm
    0_{k-1})$.  This leads to oscillating behaviour which is shown for
    a special case of this hidden tail chain in panel $(b)$ of
    Figure~\ref{fig:tail_chains}. In particular, the mean and variance
    of the hidden tail chain both can be seen to decrease in a segment
    of $k$ consecutive time points $(s, s+1, \ldots, s+k-1)$ for any
    $s > k$ such that $\text{mod}_k(s) = 1$.
  \end{example}
    
  \begin{example}[\em Multivariate extreme value copula--all mass on
    interior of simplex]
    \label{ex:mev}
    \normalfont \citet[Section~8.4]{hefftawn04} showed that if the
    spectral measure $H$ in expression \eqref{eq:exp_measure} places
    no mass on the boundary of $\SI_k$, then Assumption $\AA$ holds
    for distribution \eqref{eq:MEV} with norming functions $a_i(v)=v$,
    $b_i(v)=1$, for $i=1,\ldots,k-1$ and limiting distribution
    \begin{equation}
      G(\bm \res_{\kkminone}) = -V_{0} [\exp\{(0,\bm \res_{\kkminone})\}, \infty],
      \qquad \bm \res_{\kkminone}\in \RR^{k-1}.
      \label{eq:initial_ms}
    \end{equation} 
    Appendix~\ref{sec:proof_ms} shows that for any functional $a$
    satisfying condition~\eqref{eq:ams_condition}, a slighthly weaker
    form of Assumption $\AB$ holds (remembering here that
    $b\equiv 1$), in the sense that the limit distribution $K$, which
    is given here by
    \begin{IEEEeqnarray}{rCl}
      K(x; \bm \res_{\kkminone})&=& \ddfrac{V_{\kk} [\exp(\bm
        \res_{\kk}), \exp(a(\bm \res_{\kk}) + x) ]}{V_{\kk}
        [\exp(\bm \res_{\kk}),\infty]},\qquad x\in\RR,
      \label{eq:ms_convergence_example}
    \end{IEEEeqnarray}
    depends on $\bm \res_{\kkminone}\in\RR^{k-1}$.\ Without additional
    assumptions, it is impossible to elicit additional information
    about $a$ or the distribution $K$ in expression
    \eqref{eq:ms_convergence_example}.\ However, because both
    $V_{\kk}(\cdot,\infty)$ and $V_{\kk}(\cdot)$ are
    $-(k+1)$--homogeneous functions, the map
    \begin{equation}
      \RR^{k+1}_+\ni\bm y_{\kplusonekplusone}\mapsto
      V_{\kk}(\bm y_{\kk}, y_k)/V_{\kk}(\bm
      y_\kk,\infty)\in\RR_+
      \label{eq:map_0_hom}
    \end{equation}
    is $0$--homogeneous and this latter property restricts the
    possible forms the transition kernel
    $V_{\kk}\,:\,\RR_+^{k}\times \borel(\RR_+)\to \RR_+$ can
    take.\ One such simple form which is seen to hold for a wide
    variety of parametric models for the exponent measure is given by
    Property $K_1$ below.
    \begin{description}[wide=0\parindent] \em
    \item[Property $K_1$.] There exists a continuous function
      $a_P\,:\,\RR_+^k\to \RR_+$ which is $1$--homogeneous, and a
      non-degenerate distribution function $K_P$ on $\RR_+$, such
      that,
      \begin{itemize}
      \item[$(i)$] $K_P^\leftarrow(p^\star)=1$ for some
        $p^\star\in(0,1)$, where
        $K_P^\leftarrow(p)=\inf\left\{x\in\RR \,:\, K_P(x)>
          p\right\}$;
      \item[$(ii)$]
        $V_{\kk}(\bm y_{0\,:\,k-1}, y_k) = V_{\kk}(\bm y_{0\,:\,k-1} ,
        \infty) \, K_P\{y_k/a_P(\bm y_{0\,:\,k-1})\}$, for all
        $\bm y_{\kplusonekplusone}\in\RR_+^{k}$.
      \end{itemize}
    \end{description}
    Under Property $K_1$, some additional information about the
    location functional $a$ and the limit distribution $K$ can be
    given. Proposition \ref{prop:max_stable} gives a simple method for
    below.
    \begin{proposition}
      \label{prop:max_stable}
      Suppose that for a max-stable distribution with exponent measure
      $V$, Property $K_1$ holds.\ Let
      $a(\bm x_{\kk})=\log [a_P\{\exp (\bm x_{\kk})\}]$,
      $\bm x_{\kk}\in\RR^k$ and assume there exists a right-inverse
      $\RR\times\RR^k_+\ni (q,\bm y_{\kk})\mapsto
      V_{\kk}^\leftarrow(q;\bm y_{\kk})\in\RR_+$ such that
      $V_{\kk}\{\bm y_\kk, V_{\kk}^{\leftarrow}(q;\bm y_\kk)\} = q$
      for all $q$ and $\bm y_\kk$ in the domain of
      $V_{\kk}^{\leftarrow}$.\ Then 
      \begin{description}[wide=0\parindent]
      \item[$(i)$] $a$ satisfies property~\eqref{eq:ams_condition} and
        for all $\bm x_\kk\in\RR^k$,
        \begin{equation}
          a(\bm x_\kk)= \log V_{\kk}^{\leftarrow}\{p^\star \,
          V_{\kk}(e^{\bm x_\kk},\infty);e^{\bm x_\kk}\}.
          \label{eq:a_max_stable}
        \end{equation}
      \item[$(ii)$] Assumption $\AB$ holds with normalizing
        functionals $a$ and $b\equiv 1$, and $K(x) = K_P (e^x)$,
        $x\in\RR$.
      \item[$(iii)$] For all $x\in\RR$,
        \[
          K(x) = \frac{V_{\kk} \{\exp(\bm \res_{\kk}^\star),
            \exp(x)\}} {V_{\kk}\{\exp(\bm \res_{\kk}^\star),
            \infty\}},
        \]
        where $\bm \res_\kk^\star$ satisfies $a_P(\bm z_\kk^\star)=1$.
      \end{description}
    \end{proposition}
    
    In what follows we treat two important special cases for the
    distribution \eqref{eq:MEV}. 
    These special cases are the multivariate
    extreme value distribution with logistic dependence
    \citep{beiretal04} and H\"{u}sler--Reiss dependence
    \citep{husedavi13}, which we cover below.
    \begin{description}[wide=0\parindent]
    \item[Logistic dependence:] The exponent measure of the
      $(k+1)$-dimensional max-stable distribution with logistic
      dependence is given in expression \eqref{eq:logistic} where
      $\alpha\in(0,1)$ controls the strength of dependence, with
      stronger dependence as $\alpha$ decreases. The case $\alpha=1$
      is excluded as that corresponds to independence.\ The initial
      limiting distribution \eqref{eq:initial_ms} is
      $G(\bm \res_{\kkminone})= \left\{1 + \lVert \exp\left(-\bm
          \res_{\kkminone}/\alpha\right) \rVert\right\}^{\alpha - 1}$,
      $\bm \res_{\kkminone}\in \RR^{k-1}$.\
      Appendix~\ref{sec:logistic_convergence} shows that Assumption
      $\AB$ holds with normalizing functionals
      \begin{IEEEeqnarray*}{lll}
        a(\bm u)=-\alpha\,\log\left( \lVert\exp\left(-\bm
            u/\alpha\right)\rVert\right), &\qquad b(\bm u)=1,
        & \qquad \bm u=(u_1,\ldots,u_k)\in\RR_+^k,
      \end{IEEEeqnarray*}
      and the transition probability kernel of the renormalized Markov chain
      converges weakly to the distribution
      \[
        K(x) = \{1 + \exp(-x/\alpha)\}^{\alpha-k}\qquad x \in \RR.
      \]
      Corollary~\ref{cor:SDE_AD} asserts that a suitable normalization
      after $t\geq k$ steps is $a_t(v)=v$, $b_t(v)=1$, which leads to
      the tail chain
      \begin{equation}
        \Res_t = -\alpha \, \log \, \lVert \exp(- \bm
        \Res_{t-k\,:\,t-1}/\alpha) \rVert  + \varepsilon_t, \qquad t=k,k+1,\dots,
      \end{equation}
      where $\{\varepsilon_t\}_{t=k}^\infty$ is a sequence of
      i.i.d. random variables with distribution $K$. Note that the
      tail chain can also be expressed as
      \[
      \Res_t = \Res_{t-k}-\alpha \, \log \, \lVert \exp\{- (0, \Res_{t-k}\bm 1_{k-1} - \bm
        \Res_{t-k+1\,:\,t-1})/\alpha\} \rVert  + \varepsilon_t,\qquad t=k,k+1,\dots.
      \]
      Here the hidden tail chain is identical to the tail chain as
      $a_t(x)=x$ and $b_t(x)=1$ for all $t=1,2,\dots$. 
      When $k=1$, the tail chain can be seen to reduce to the random
      walk results of \cite{Smith92} and \cite{perf94}, but when
      $k>1$, the tail chain behaves like a random walk with an
      additional factor which depends in a non-linear way on the
      ``profile''
      $ \Res_{t-k}\bm 1_{k-1}- \bm \Res_{t-k+1\,:\,t-1}$, of the
      $k-1$ previous values.\ 
    \item[H\"{u}sler--Reiss dependence:] The exponent measure of the
      $(k+1)$-dimensional max-stable distribution with
      H\"{u}sler--Reiss dependence is
      \begin{IEEEeqnarray*}{rCl}
        V(\bm y_{\kplusonekplusone})&=& \sum_{i=0}^k
        \frac{1}{y_i}\,\Phi_{k}\left[\left\{\log(y_j/y_i) + \sigma^2 -
            \sigma_{ij}\right\}_{j\neq i};\,\bm \Sigma^{(i)}\right],
        \qquad \bm y_{\kplusonekplusone}\in\RR_+^{k+1}
    \end{IEEEeqnarray*}
    where $\Phi_{k}(\,\cdot\, ;\,\bm \Sigma^{(i)})$ denotes the
    multivariate normal distribution function with mean zero and
    covariance matrix
    $\bm \Sigma^{(i)} = \bm T_i\, \bm \Sigma \, \bm T_i^\top$, for
    $\bm \Sigma = (\sigma_{ij})_{i,j=0}^k$ a positive definite
    Toeplitz covariance matrix with common diagonal elements
    $\sigma_{ii}=\sigma^2$. Here $\bm T_i$ is a $k\times (k+1)$ matrix
    with the $(i+1)$th column having $-1$ for each entry and the other
    columns being the $k$ standard orthonormal basis vectors of
    $\RR^k$, that is,
    \[
      \bm T_i = \left(\begin{matrix}
          1&0& \cdots & 0 & -1 & 0 & \cdots & 0\\
          0&1& \cdots & 0 & -1 & 0 & \cdots & 0  \\
          &\dots& \cdots & \cdots & \cdots & \cdots & \cdots &\\
          0 & 0 & \cdots & 0 & -1 & 0  & \cdots &1
        \end{matrix}\right),\qquad i=0,\ldots,k.
    \]
    The matrix $\bm \Sigma$ controls the strength of dependence, with
    larger values for $\sigma_{ij}$ indicating stronger dependence
    between the associated elements of the random vector. The initial
    limiting distribution \eqref{eq:initial_ms} is
    $G(\bm x)=\Phi_{k-1}[\bm x-\{-\text{diag}(\bm
    \Sigma^{(0)})/2\};\bm \Sigma^{(0)}]$, $\bm x \in \RR^{k-1}$.\
    Appendix~\ref{sec:HR_convergence} shows Assumption $\AB$ holds with
    normalizing functionals
      \[
        a(\bm u) = -\tau \bm K_{01}^\top\bm C \bm K_{10} \cdot \bm u,
        \qquad b(\bm u)=1, \qquad \bm u = (u_1,\ldots, u_k)\in\RR_+^k
      \]
      where the quantities $\tau, \bm q, \bm C, \bm K_{10}$ and
      $\bm K_{01}$ are defined in Appendix~\ref{sec:HR_convergence}.\
      The transition probability kernel of the renormalized Markov chain converges
      weakly to the distribution
      \[
        K(x) = \Phi\{(x/\tau)+  (\bm K_{01}^\top \bm \Sigma^{-1}\bm
        1_{k+1}^\top)/(\bm 1_{k+1}^\top \bm q)\} \qquad x \in \RR.
      \]
      Corollary~\ref{cor:SDE_AD} asserts that a suitable normalization
      after $t\geq k$ steps is $a_t(v)=v$, $b_t(v)=1$, which leads to
      the hidden tail chain (identical to the tail chain) of
      \begin{IEEEeqnarray*}{rCl}
        \Res_t &=& -\tau \bm K_{01}^\top\bm C \bm K_{10} \cdot \bm
        \Res_{t-k\,:\,t-1} + \varepsilon_t,\qquad t=k,k+1,\dots,
      \end{IEEEeqnarray*}
      where $\{\varepsilon_t\}_{t=k}^\infty$ is a sequence of
      i.i.d. random variables with distribution $K$.\ Note that the
      tail chain can also be expressed as
      \[
        \Res_t = \Res_{t-k}+ \tau \bm K_{01}^\top\bm C \bm K_{10} \cdot (0,
        \Res_{t-k}\bm 1_{k-1} - \bm \Res_{t-k+1\,:\,t-1}) + \varepsilon_t,
        \qquad t=k,k+1,\dots,
      \]
       which shows, similarly with the
      logistic copula, that the tail chain behaves like a random walk
      with an additional factor which depends linearly on the
      ``profile''
      $\Res_{t-k} \bm 1_{k-1}-\bm \Res_{t-k+1\,:\,t-1}$, of the
      $k-1$ previous values, so differs from the previous example in
      this respect.\
    \end{description}

    Panels $(c)$ and $(d)$ in Figure~\ref{fig:tail_chains} show an
    almost linear behaviour for two special cases of the tail chains
    presented. Although the copulae used to derive both tail chains
    have the same extremal coefficient \citep[see][]{schltawn03},
    ensuring that the core level of extremal dependence is common in
    both, the decay rate of the two processes is markedly
    different. This shows that the type of drift function and
    distribution for the innovation term $\varepsilon_t$ might impact
    upon the characteristics of transitioning from an extreme state to
    the main body of the process.
  
    \begin{figure}[htpb!]
      \centering
      \includegraphics[scale=1]{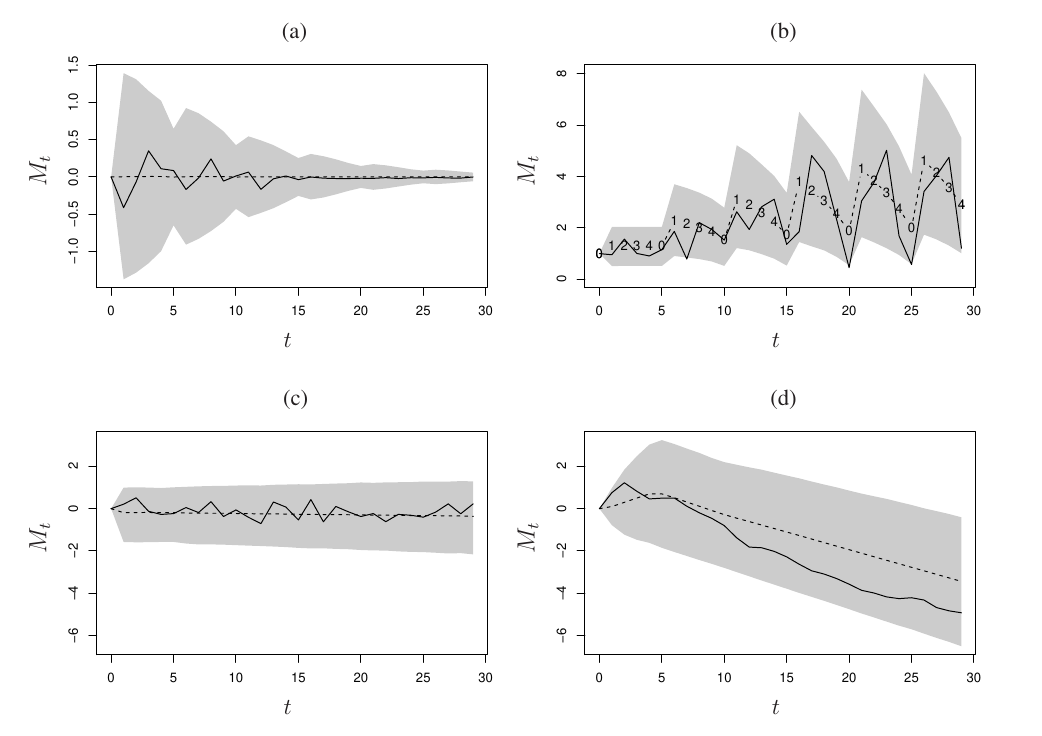}
      \caption{Properties for each hidden
        tail chain for Examples 1--3. Presented for each chain are:
        pointwise 2.5\% and 97\% quantiles of the sampling
        distribution (shaded region), mean of the sampling
        distribution (dashed line) and one realization from the
        (hidden) tail chain (solid line). The copula of
        $\bm X_{\kplusonekplusone}$ used to derive the (hidden) tail
        chain comes from: (a) standard multivariate Gaussian copula
        with Toeplitz positive-definite covariance matrix $\Sigma$
        generated by the vector $(1, 0.70, 0.57, 0.47, 0.39, 0.33)$,
        (b): inverted logistic with
        $\alpha=\log(\bm 1_{k+1}^\top\,\Sigma^{-1}\,\bm 1_{k+1})/\log
        k=0.27$, (c) logistic copula with $\alpha=0.32$. The value of
        the function $\text{mod}_k(t)$ is also highlighted on the mean
        function of the time series with numbers ranging from 0 to 4
        for all $t$. (d): H\"{u}sler--Reiss copula with Toeplitz
        positive-definite covariance matrix generated by the vector
        $(1, 0.9, 0.7, 0.5, 0.3, 0.1)$. The parameters for all copulas
        are chosen such that the coefficient of residual tail
        dependence $\eta$ \citep{ledtawn97} and the extremal
        coefficient $\theta$ \citep{beiretal04} are equal for the
        copulas in panels (a) and (b), and (c) and (d), respectively.}
      \label{fig:tail_chains}  
    \end{figure}
  \end{example}

  \begin{example}[Multivariate extreme value copula
     with asymmetric logistic structure
    \label{ex:alog}
    \citep{tawn90}]\normalfont This is a second-order Markov process
    for which Assumptions $\AA$ and $\AB$ fail to hold and it has more
    complicated structure than we have covered so far where weak
    convergence on $\RR^k$ was studied.\ In this example, the weak
    convergences in Assumptions $\AA$ and $\AB$ no longer hold on
    $\RR^{k-1}$ and $\RR$ (cf. Remark \ref{rem:dist_support}), but on
    $\overline{\RR}^{k-1}$ and $\overline{\RR}$, respectively.\ The
    example is a special case of a stationary Markov chain with
    transition probability kernel \eqref{eq:mstranskernel} and
    exponent measure given by
    \begin{IEEEeqnarray}{rCl}
      V(x_0,x_1,x_2)&=& \theta_0\, x_0^{-1} + \theta_1\,
      x_1^{-1}+ \theta_{2}\, x_2^{-1} +\nonumber\\
      && + \theta_{01}\left\{ \left(x_0^{-1/\nu_{01}}+
          x_1^{-1/\nu_{01}}\right)^{\nu_{01}} +
        \left(x_1^{-1/\nu_{01}}+
          x_2^{-1/\nu_{01}}\right)^{\nu_{01}}\right\}  +      \label{eq:alog} \\
      &&+\, \theta_{01} \left(x_0^{-1/\nu_{02}}+x_2^{-1/\nu_{02}}
      \right)^{\nu_{02}} + \theta_{012} \left(x_0^{-1/\nu_{012}}+
        x_1^{-1/\nu_{012}}+x_2^{-1/\nu_{012}}\right)^{1/\nu_{012} },\nonumber
    \end{IEEEeqnarray}
     where $\nu_A \in (0,1)$ for any
    $A\in 2^{\{0,1,2\}}\sm \emptyset$, and
    \begin{IEEEeqnarray*}{l}
      \theta_0 + \theta_{01} + \theta_{02} + \theta_{012}=1, \quad
      \theta_1 + 2\theta_{01} + \theta_{012}=1, \quad \theta_2 +
      \theta_{01} + \theta_{02} + \theta_{012}=1
    \end{IEEEeqnarray*}
    with
    $\theta_0, \theta_1, \theta_2, \theta_{01}, \theta_{02},
    \theta_{012} >0.$

    The initial distribution of the Markov process is
    $F_{01}(x_0,
    x_1)=F_{012}(x_0,x_1,\infty)=\exp\{-V(y_0,y_1,\infty)\}$, with
    $(y_0, y_1)$ defined in expression \eqref{eq:MEV}.\ It can be seen
    that the transition probability kernel
    $\pi(x_0, x_1)= -y_0^2 \,V_0(y_0,y_1,\infty)
    \exp\left(y_0^{-1}-V(y_0,y_1,\infty)\right)$ associated with the
    conditional distribution of $X_1 \mid X_0$,
    converges with two distinct normalizations, that is,
    $\pi(v, dx)\wk K_0(dx)$ and $\pi(v, v+dx)\wk K_1(dx)$
    as $v\to\infty$, to the distributions
    $K_0=(\theta_0+\theta_{02}) \, F_E +
    (\theta_{01}+\theta_{012})\,\delta_{+\infty},$ and
    $K_1 = (\theta_0+\theta_{02})\, \delta_{-\infty} + \theta_{01}
    \,G_{01} + \theta_{012} \, G_{012}$, respectively, where
    $F_E(x)=(1-\exp(-x))_+$,
    $G_{A}(x)=\left(1+\exp(-x/\nu_A)\right)^{\nu_A-1}$ and $\delta_x$
    is a point mass at $x \in [-\infty, \infty]$ \citep[cf. Example 5
    in][]{papaetal17}. Distributions $K_0$ and $K_1$ have entire mass
    on $(0,\infty]$ and $[-\infty,\infty)$ respectively.\ In the first
    and second normalizations, mass of size
    $(1-\theta_{01}-\theta_{012})$ escapes to $+\infty$ and mass of
    size $(\theta_0+\theta_{02})$ escapes to $-\infty$,
    respectively. As explained by \cite{papaetal17}, the reason for
    this behaviour is that the separate normalizations are related to
    two different modes of the conditional distribution of
    $X_{1}\mid X_0$.\ This phenomenon also manifests in the
    conditional distribution of $X_2\mid \{X_0, X_1\}$, which is given
    by
  \begin{IEEEeqnarray*}{rCl}
    \pi(\bm x_{0:1},\, x_2) &=& \frac{(V_0 V_1 -
      V_{01})(\bm y_{0:2})}{(V_0 V_1-V_{01}
      )(\bm y_{0:1},\infty)}
    \,\exp(V(\bm y_{0:1},\infty)-V(\bm y_{0:2})),
  \end{IEEEeqnarray*}
  where $g(f_1,f_2,f_3)(x) := g(f_1(x),f_2(x), f_3(x))$ for maps $g$
  and $f_i$, $i=1,2,3$.\ Here the problem is more complex, with this
  transition probability kernel converging with $2\,(2^k-1)=6$
  distinct normalizations.\ Letting
  \begin{IEEEeqnarray}{rCl}
    a_{11,1}(v_1,v_2) &=&
    -\nu_{012}\log\{\exp(-v_1/\nu_{012})+\exp(-v_2/\nu_{012})\}\nonumber\\
    a_{10,1}(v_1,v_2) &=&v_1, a_{01,1}(v_1,v_2) =v_2 \label{eq:functionals}\\
    a_{11,0}(v_1,v_2) &=&a_{01,0}(v_1,v_2) = a_{10,0}(v_1,v_2) = 0\nonumber,
  \end{IEEEeqnarray}
  it
  can be shown that for $(x_0,x_1)\in \RR^2$ and as $v\to \infty$,
  \begin{IEEEeqnarray*}{rl}
    \pi((v+x_0, v+x_1), a_{11,1}(v+x_0, v+x_1) + dy)\wk
    K_{\{1,1\},\{1\}}(dy;x_0,x_1)& \quad \text{on $[-\infty,\infty)$}\\
    \pi((v+x_0, v+x_1),  a_{11,0}(v+x_0, v+x_1)+dy)\wk K_{\{1,1\},\{0\}}(dy;x_0,x_1)& \quad \text{on $(0,\infty]$}\\
    \pi((v+x_0, x_1), a_{10,1}(v+x_0, x_1)+dy)\wk K_{\{1,0\},\{1\}}(dy;x_0,x_1)& \quad \text{on $[-\infty,\infty)$},\\
    \pi((v+x_0, x_1),  a_{10,0}(v+x_0, x_1) + dy)\wk K_{\{1,0\},\{0\}}(dy;x_0,x_1)& \quad \text{on $(0,\infty]$}\\
    \pi((x_0, v+x_1), a_{01,1}(x_0,v+x_1)+dy)\wk K_{\{0,0\},\{1\}}(dy;x_0,x_1)& \quad \text{on $[-\infty,\infty)$}\\
    \pi((x_0, v+x_1), a_{01,0}(x_0,v+x_1)+dy)\wk K_{\{0,1\},\{0\}}(dy;x_0,x_1) & \quad \text{on $(0,\infty]$}
  \end{IEEEeqnarray*}  
  \sloppy where the limiting measures are given by
  \begin{IEEEeqnarray*}{rCl}
    &&K_{A, \{1\}}=m_{A} \,\delta_{-\infty}+(1-m_{A})\,G_{A,\{1\}}
    \text{ and } K_{A, \{0\}}=m_{A}
    \,G_{A,\{0\}}+(1-m_{A})\,\delta_{\infty}, \quad A\in
    \{0,1\}^2\sm \{0,0\}
\end{IEEEeqnarray*}
with
\begin{equation*}
  m_{A}(x_0,x_1) = \begin{cases}\Big\{1 + (\kappa_{012}/\kappa_{01})\,
    e^{\lambda(x_0+x_1)}\,{\displaystyle
      W_{\nu_{01}-1}(e^{x_0},e^{x_1}\,;\,\nu_{01})
      /
      W_{\nu_{012}-1}(e^{x_0},e^{x_1}\,;\,\nu_{012})}\Big\}^{-1},& \text{$A=\{1,1\}$}\\
  \theta_{0}/(\theta_0+\theta_{02}), & \text{$A=\{1,0\}$}
  \\
  \theta_{1}/(\theta_1+\theta_{01}), & \text{$A=\{0,1\}$},
  \end{cases}
\end{equation*}
and
\begin{align*}
  G_{\{1,1\},\{1\}}(y\,;\,x_0,x_1) &=W_2\{1, \exp(y)\,;\,\nu_{012}\}\\
  G_{\{1,1\},\{0\}}(y\,;\,x_0,x_1) &=F_E(y)\\
  G_{\{1,0\},\{1\}}(y\,;\,x_0,x_1) &= W_1\{1,\exp(y)\,;\,\nu_{02}\}\\
  G_{\{1,0\},\{0\}}(y\,;\,x_0,x_1) &= \left[\theta_1 +
                           \theta_{01} \left\{1 + W_1(1, T(y)/T(x_1);\nu_{01})\right\} +
                           \theta_{012} W_1(1, T(y)/T(x_1)\,;\,\nu_{012})\right] g_{10}(x, y)\\
  G_{\{0,1\},\{1\}}(y\,;\,x_0,x_1) &=W_1\{1, \exp(y)\,;\,\nu_{01}\}\\
  G_{\{0,1\},\{0\}}(y\,;\,x_0,x_1) &= \left[\theta_0 +
                           \theta_{01} +\theta_{02} W_1(1, T(y)/T(x_0)\,;\,\nu_{02}) +\theta_{012}
                           W_1(1, T(y)/T(x_0)\,;\,\nu_{012})\right] g_{01}(x, y),
\end{align*}
and the function $T$ is defined in expression~\eqref{eq:MEV},
$\kappa_A = \theta_{A}\,(\nu_{A}-1)/\nu_{A}$,
$W_p(x, y\,;\,\nu)=(x^{-1/\nu} + y^{-1/\nu})^{\nu-p}$, with $x,y>0$,
$p \in \RR$, $\nu \in(0,1)$, and
\begin{IEEEeqnarray*}{l}
  \log g_{10}(x,y)=V(\infty, T(x), \infty) - V(\infty, T(x),
  T(y)),\quad \text{and} \quad
  \log g_{01}(x,y)=V(T(x), \infty, \infty) - V(T(x), \infty,
  T(y)).
\end{IEEEeqnarray*}
To help explain the necessity for requiring the normalizing
functionals \eqref{eq:functionals} to describe the evolution of an
extreme episode after witnessing an extreme event in this 2nd order
Markov process, it is useful to consider the behaviour of the spectral
measure $H$, defined in equation \eqref{eq:exp_measure}, for the
initial distribution $F_{012}$ of this process. Here, the spectral
measure $H$ places mass of size $\lvert A\lvert\,\theta_A$ on each
subface $A\in\mathscr{P}([2])$ of $\SI_2$ \citep{coletawn91} which
implies that different subsets of the variables
$(X_{t-2},X_{t-1},X_t)$ can take their largest values simultaneously,
see for example \cite{simpetal20}. Hence, if the Markov process is in
an extreme episode at time $t-1$, $t\geq 3$, then it follows that
there are four possibilities for the states $(X_{t-2}, X_{t-1})$,
that is, either the variables $X_{t-2}$ and $ X_{t-1}$ are simultaneously
extreme or just one of them is. Consequently, there are two
possibilities for the state of the process at time $t$, that is, the
variable $X_t$ can be either extreme or not, and this is demonstrated
by bimodality in the transition probability kernel under all four
distinct possibilities for the states $X_{t-2}$ and $X_{t-1}$. In
total, this gives rise to 6 distinct possibilities which necessitate
an ``event specific'' normalizing functional to guarantee the weak
convergence of the transition probability kernel. This justifies the
labelling of the functionals in expression~\eqref{eq:functionals}
where the label $(A,b)$ appearing in the subscript, with
$A\in \{0,1\}^2$ and $b\in \{0,1\}$, indicates transitioning from one
of four possible configurations $(A)$ at times $t-2$ and $t-1$ into
two possible configurations $(b)$ at time $t$---with 1 indicating that
the state is extreme and 0 otherwise. The case where the Markov
process is in an extreme episode at time $t-1$, for $t=2$, is handled
similarly noting now that $X_0$ is, by virtue of the conditioning,
already extreme and hence there are two possibilities for $X_0$ and
$X_1$, that is, either $X_1$ is extreme or not.
\begin{figure}[htpb!]
  \centering
  \includegraphics[scale=1]{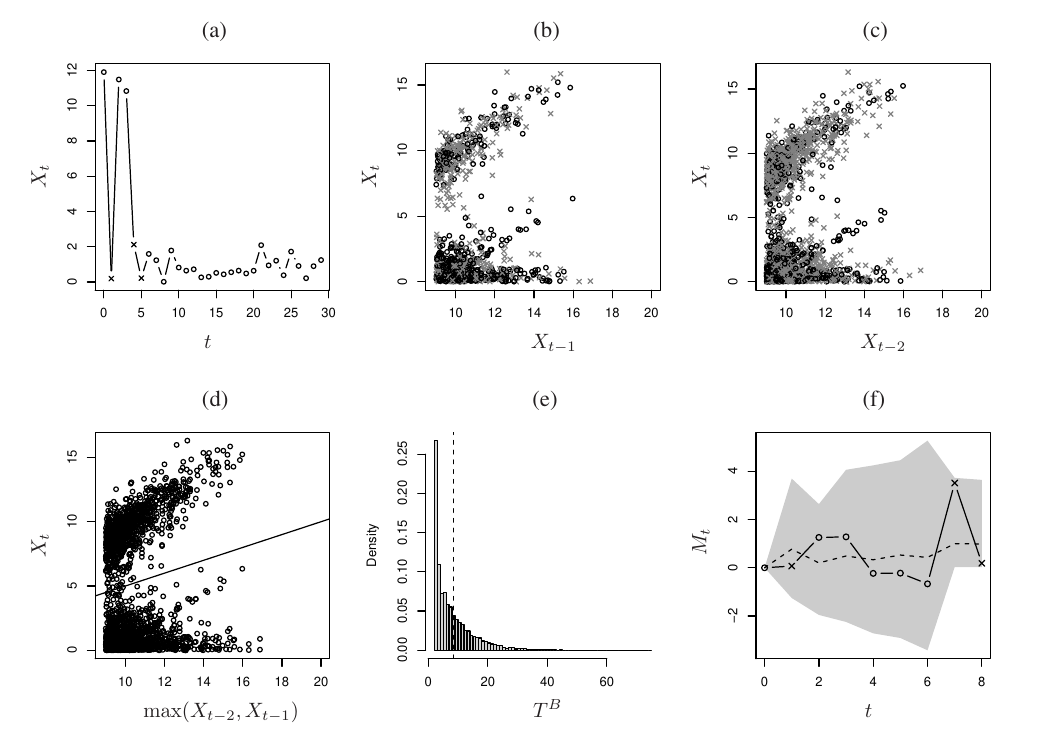}
  \caption{(a): time series plot showing a single realization from the
    2nd order Markov chain with asymmetric logistic dependence
    \eqref{eq:alog} initialized from the distribution of
    $X_0\mid X_0 > 9$.\ For this realization, there are three
    change-points $T_1^X, T_2^X$ and $T_3^X$ and are highlighted with
    a cross.\ (b): Scatterplot of states
    $\{(X_{t-1},X_{t})\,:\, X_{t-1} > 9\}$ drawn from $10^3$
    realisations of the Markov chain initialized from the distribution
    of $X_0\mid X_0>9$.\ Points for which $X_{t-2}<9$ and
    $X_{t-2}\geq 9$ are highlighed with grey crosses and black
    circles, respectively.\ (c): Scatterplot of consecutive states
    $(X_{t-2},X_{t})$.\ Points for which $X_{t-1}<9$ and
    $X_{t-1}\geq 9$ are highlighed with grey crosses and black
    circles, respectively.\ (d): Scatterplot of states
    $\{(\max(X_{t-2},X_{t-1}), X_t)\,:\, \max(X_{t-2}, X_{t-1}) > 9\}$
    and line $X_t = c\,\max(X_{t-2},X_{t-1})$ with $c=\frac{1}{2}$
    superposed.\ (e): Histogram of termination time $T^B$ obtained
    from $10^4$ realizations from the hidden tail chain.\ The Monte
    Carlo estimate of the mean of the distribution is $8.42$ and shown
    with a dashed vertical line.\ (f): pointwise 2.5\% and 97\%
    quantiles of the sampling distribution (shaded region), mean of
    the sampling distribution (dashed line) and one realization from
    the hidden tail chain (solid line), conditioned on $T^B=8$.\ The
    value of the latent Bernoulli process $B_t$ is highlighted with a
    cross when $B_t=0$ and with a circle when $B_t=1$. For all plots
    presented,
    $\theta_{0}=\theta_{1}=\theta_{2}=\theta_{01}=\theta_{02} = 0.3$,
    $\theta_{012}=0.1$, and $\nu_{01}=\nu_{02}=\nu_{012}=0.5$.\ }
\label{fig:alog}  
\end{figure}
Although complex, these modes can be identified by any line determined
by the loci of points
$(\max(x_{t-2},x_{t-1}), \zeta \, \max(x_{t-2}, x_{t-1}))$, where
$x_{t-2},x_{t-1} \in \RR$, for some $\zeta\in(0,1)$, in the distribution
of $X_{t}\mid \max\{X_{t-2},X_{t-1}\}>v$, see panels $(b)$, $(c)$ and
$(d)$ Figure~\ref{fig:alog}, where $v$ is taken equal to 9.\ This
facilitates accounting for the identification of the normalizing
functionals by introducing the stopping times, $T_0^X = 0$ a.s., and
\begin{IEEEeqnarray*}{rCl}
  T^X_j&=&\inf \{t \in (T_{j-1}^X , T^X]\,:\, X_t \leq \zeta \,
  \max(X_{t-2}, X_{t-1})\}, \qquad j\geq 1,
\end{IEEEeqnarray*}
where
\begin{IEEEeqnarray*}{rCl}
  T^X&=&\inf\{t\geq 2\,:\,X_{t-1}\leq  
  \zeta\,\max (X_{t-3},X_{t-2}), X_t\leq 
  \zeta\,\max (X_{t-2},X_{t-1})\},
\end{IEEEeqnarray*}
subject to the convention $X_{-s}=0$ for $s\in \NN\sm\{0\}$,
that is $T_j^X$, with $j\geq 1$, is the $j$th time that $\zeta$ multiplied by
the maximum of the previous two states is not exceeded after time 0,
and the termination time $T^X$ is the first time after time $0$ where
two consecutive states did not exceed $\zeta$ times the maximum of their
respective two previous states.\ Define
\[
  a_1(v) =
  \begin{cases}
    v & T^X_1 > 1 \\
    0 & T^X_1 =1
  \end{cases}
  \qquad \text{and} \qquad b_t(v) = 1 \qquad \text{for all $t\geq 1$}.
\]
Then for $t \in (T_{j-1}^X , T_j^X ]$, letting 
\[
  a_t(v_1,v_2) = \begin{cases}
    a_{11,1}(v_1,v_2)& \text{if $t\neq T_j^X$, $t-1\neq T_{j-1}^X$, $t-2\neq T_{j-1}^X$} \\
    a_{11,0}(v_1,v_2) & \text{if $t=T_j^X$, $t-1\neq T_{j-1}^X$, $t-2\neq T_{j-1}^X$}\\
    a_{10,1}(v_1,v_2)& \text{if $t\neq T_j^X$, $t-1= T_{j-1}^X$, $t-2\neq T_{j-1}^X$}\\
    a_{10,0}(v_1,v_2)& \text{if $t= T_j^X$, $t-1= T_{j-1}^X$, $t-2\neq T_{j-1}^X$}\\
    a_{01,1}(v_1,v_2)& \text{if $t\neq T_j^X$, $t-1\neq T_{j-1}^X$, $t-2= T_{j-1}^X$}\\
    a_{01,0}(v_1,v_2)& \text{if $t= T_j^X$, $t-1\neq T_{j-1}^X$, $t-2= T_{j-1}^X$},
  \end{cases}
\]
yields the hidden tail chain of this process. 
Specifically, let $\{B_t\,:\,t=0,1,\ldots\}$ be a sequence of latent
Bernoulli random variables.\ Define the hitting times
$T_j^B=\inf\{T_{j-1}^B < t \leq T^B\,:\,B_t=0\}$ with $T_0^B=0$
a.s. and $T^B=\inf\{t \geq 2 \,:\,B_{t-1}=0,B_t=0\}$. Then the hidden
tail chain process $\{\Res_t\}$ together with the latent Bernoulli
process $\{B_t\}$ form a second-order Markov process with initial 
distribution $(B_0, \Res_0)=(1,0)$ a.s.,
$B_1 \sim \text{Bern}(\theta_{01}+\theta_{02})$, and
\[
  \PR(\Res_1 \leq y\mid \Res_0, \bm B_{0:1}) =
  \begin{cases}
    \theta_{01} \,G_{01}(y) + \theta_{012} \,G_{012}(y)& B_1 = 1\\
   F_E(y) & B_1=0.
  \end{cases}
\]
The transition mechanism is given by
\[
  B_t\mid \bm B_{t-2\,:\,t-1}, \bm \Res_{t-2\,:\,t-1}, \{t \leq T^B \} \sim
  \text{Bern}(m_{\{\bm B_{t-2\,:\,t-1}\}}(\bm \Res_{t-2\,:\,t-1})),
\]
and
\begin{IEEEeqnarray*}{rCl}
  \PR(\Res_t \leq z \mid \bm B_{t-2\,:\,t}, \bm \Res_{t-2\,:\,t-1}, \{ t
  \leq T^B \})&=& G_{\{\bm B_{t-2\,:\,t-1}\},\{B_t\}}(z - a_{\bm
    B_{t-2\,:\,t-1},B_t}(\bm \Res_{t-2\,:\,t-1})).
\end{IEEEeqnarray*}
Panel $(a)$ in Figure~\ref{fig:alog} illustrates a realization from a
special case of this 2nd order Markov process. This realized path
shows that after witnessing an extreme event at time $t=0$ the process
transitions to the body of the process at time $t=1$ and then, has two
extreme states at $t=2$ and $3$ and two non-extreme states at $t=4$
and $5$. After two non-extreme values the process has permanently
transitioned to its equilibrium, that is, for $t=6,\ldots$ in this
realization. The sampling distribution of the average termination time
$T_B$ of the hidden tail chain is presented in panel $(e)$ whereas the
behaviour of hidden tail chain conditioned on it terminating after 8
steps, that is, $T_B=8$, is shown in panel $(f)$. This shows that
whilst at an extreme state, the average value of $\Res_t$ is stable
through time. 
\end{example}
\appendix
\renewcommand{\theequation}{A.\arabic{equation}}
\setcounter{equation}{0}
\section{Proofs}
\label{sec:proofs}
\subsection{Preparatory results for Theorems 1 and 2}
The proofs of Theorems~\ref{thm:tailchain}
and~\ref{thm:tailchain:nonneg} are based on Lemmas~\ref{lemma:1} and
\ref{lemma:2} below whose proofs are similar to Lemmas 4 and 5 in
\cite{papaetal17} and are omitted for brevity.
\begin{lemma}\normalfont
  \label{lemma:1}
  Let $\{X_t\,:\,t=0,1,\hdots\}$ be a homogeneous $k $-th order Markov
  chain satisfying Assumption $\AB$.\ Then, for any
  $g \in C_b(\RR)$ and for each time step $t=k,k+1,\hdots,$ as
  $v\to \infty$
  \begin{dmath*}
    \int_{\RR} g(y) \pi\left( \bm A_{t}(v,\bm \res), a_{t}(v)+b_{t}(v)
      \,dx \right) \to \int_{\RR} g(\psi_{t}^a(\bm \res)+
    \psi_{t}^b(\bm \res)\,x)\,
    K(dx),
  \end{dmath*}
  and the convergence holds uniformly on compact sets in the variable
  $\bm \res \in \RR^k$.
\end{lemma}

\begin{lemma}\normalfont
  \label{lemma:2} Let $\{X_t\,:\,t=0,1,\hdots\}$ be a homogeneous
  $k $-th order Markov chain satisfying Assumption $\BB$.\ Then, for
  any $g \in C_b([0,\infty))$ and for each time step $t=k,k+1,\hdots$,
  as $v\to \infty$
\[
  \int_{[0,\infty)} g(y) \pi\left(\bm B_{t}(v, \bm x), b_{t+1}(v) \,dy
  \right) \to \int_{[0,\infty)} g(\psi_t^b(\bm x)\,y)\, K(dy),
\]
and the convergence holds uniformly on compact sets in the variable
$\bm x \in [\delta_1,\infty) \times \cdots \times [\delta_k, \infty)$
for any $(\delta_1,\ldots, \delta_{k}) \in (0,\infty)^k$.
\end{lemma}

\begin{lemma}\label{lemma:unif_conv}\normalfont
  [Slight variant of \cite{kulisoul}] Let $(E,d)$ be a
  complete locally compact separable metric space and $\mu_n$ be a
  sequence of probability measures which converges weakly to a
  probability measure $\mu$ on $E$ as
  $n\to \infty$.
  \begin{itemize}
    \item[$(i)$] Let $\varphi_n$ be a uniformly bounded
sequence of measurable functions which converges uniformly on compact
sets of $E$ to a continuous function $\varphi$. Then $\varphi$ is
bounded on $E$ and $\lim_{n\to \infty}
\mu_n(\varphi_n)\to \mu(\varphi)$.
    \item[$(ii)$] Let $F$ be a topological space. If $\varphi\in
C_b(F\times E)$, then the sequence of functions $F\ni x \mapsto \int_E
\varphi(x,y)\mu_n(dy)\in\RR$ converges uniformly on compact sets of
$F$ to the (necessarily continuous) function $F\ni x\mapsto\int_E
\varphi(x,y)\mu(dy)\in \RR$. 
  \end{itemize}
\end{lemma}

\subsection{Proofs of Theorems 1 and 2}
\label{sec:proof1}
\noindent \textit{Preliminaries}. Let $a_0(v)\equiv v$ and $b_0(v)\equiv 1$ and define
\begin{IEEEeqnarray}{rCl} && v_u(\res_0) = u + \sigma(u)\res_0, A_t(v,x) =
  a_t(v)+b_t(v)x, \text{ and }\nonumber \\
  && \bm A_{\tk}(v,\bm x_{\tk}) = (A_{t-k}(v,x_{t-k})
  ,\hdots, A_{t-1}(v,x_{t-1})).
  \label{eq:bmA}
\end{IEEEeqnarray}
We note that in our notation, when $k=1$, the initial distribution of
the rescaled conditioned Markov chain is
\begin{equation} \frac{F_0( v_u(d\res_0))}{\overline{F}_0(u)} =
\PR\left(\frac{X_0-u}{\sigma(u)} \in d\res_0 \mid X_0 > u\right).
\label{eq:init}
\end{equation} whereas when $k>1$, it equals to the product of the right hand
side of equation \eqref{eq:init} with 
\[ \pi_0( v_u(\res_0),
\bm A_{\kkminone}(v_u(\res_0),d\bm \res_{\kkminone})) :=
\PR\left(\bigcap_{j=1}^{k-1}\left\{\frac{X_j-a_j(X_0)}{b_j(X_0)} \in d
\res_j\right\}~\Bigg | ~\frac{X_0 - u}{\sigma(u)} = \res_0\right).
\] For $j \in \{k,\dots,t\}$ with $t\geq k \geq 1$ the transition
kernels of the rescaled Markov chain can be written as
\[
  \pi(\bm A_{\jk}(v_u(\res_0),\bm \res_{j-1,k}),A_j(v_u(\res_0),d\res_j))=\PR\left(\frac{X_j
      - a_j(X_0)}{b_j(X_0)} \in d \res_j~\Bigg | ~
    \left\{\frac{X_{j-i}-a_{j-i}(X_0)}{b_{j-i}(X_0)} =
      \res_{j-i}\right\}_{i=1,\dots, k}\right).
\]
\begin{proof}[Proof of Theorem~\ref{thm:tailchain}] Consider, for
$t\geq k\geq 1$, the measures
\begin{IEEEeqnarray*}{rCl} \mu^{(u)}_t(d\res_0,\dots,d\res_t) &=&
  \prod_{j=k}^t
  \pi(\bm A_{\jk}(v_u(\res_0),\bm \res_{\jk}),A_j(v_u(\res_0),d\res_j))\\
  &&\times \left[\pi_0( v_u(\res_0), \bm
    A_{\kkminone}(v_u(\res_0),d\bm \res_{\kkminone}))\right]^{1(k>1)}
  \frac{F_0( v_u(d\res_0))}{\overline{F}_0(u)}
  \end{IEEEeqnarray*} and
  \begin{IEEEeqnarray*}{rCl} \mu_t(d\res_0,\dots,d\res_t) &=& \prod_{j=k}^t
K\left(\frac{d\res_j -
\psi_{j}^a(\bm \res_{\jk})}{\psi_j^b(\bm \res_{\jk})}\right)
[G(d\res_1\times\cdots\times
d\res_{k-1})]^{1(k>1)}\,H_0(d\res_0),
\end{IEEEeqnarray*} on $[0,\infty)\times \RR^t$, where $1(k>1)$ denotes
the indicator function of $\{k>1\}$.\ For $f\in C_b([0,\infty)\times \RR^t)$,
we may write
  \begin{equation*}
\mathbb{E}\left[f\left(\frac{X_0-u}{\sigma(u)},\frac{X_1-a_1(X_0)}{b_1(X_0)},\hdots,\frac{X_t-a_t(X_0)}{b_t(X_0)}\right)~\Bigg|~
X_0>u\right]= \int_{[0,\infty)\times \RR^t} f(\bm \res_{\tplusonetplusone})
\mu_t^{(u)} (d\res_0,\hdots,d\res_t)
  \end{equation*}  
and
\[ \mathbb{E}\left[f\left(E_0,\Res_1,\cdots,\Res_t\right)\right] =
\int_{[0,\infty)\times \RR^t} f( \bm \res_{\tplusonetplusone}) \mu_t
(d\res_0,\hdots,d\res_t).
\]
We need to show that $\mu_t^{(u)}$ converges weakly to $\mu_t$.\ Let
$g_0\in C_b([0,\infty))$ and $g\in C_b(\RR^{k})$.\ The proof is by
induction on $t$. For $t = k$ it suffices to show that
\begin{IEEEeqnarray}{rCl} && \int_{[0,\infty)\times \RR^{k}} g_0(\res_0)
g(\res_1,\hdots,\res_{k})
\mu_k^{(u)}(d\res_0,d\res_1,\hdots,d\res_{k})\nonumber\\\nonumber\\ &&=
\int_{[0,\infty) } g_0(\res_0) \left[\int_{\RR^{k}} g(\bm \res_{\kplusonek})
\pi((v_u(\res_0),\bm A_{\kkminone}(v_u(\res_0),\bm \res_{\kkminone})),A_{k}(v_u(\res_0),d\res_k))
\right.\nonumber\\\nonumber\\ && \qquad \qquad
\qquad\qquad\qquad\qquad \pi_0( v_u(\res_0),
\bm A_{\kkminone}(v_u(\res_0),d\bm \res_{\kkminone}))\bigg]
\frac{F_0( v_u(d\res_0))}{\overline{F}_0(u)},
\label{eq:induction_first}
\end{IEEEeqnarray} converges to
$\mathbb{E}(g_0(E_0))\,\mathbb{E}(g(\Res_1,\hdots,\Res_k))$.

By Assumptions $\AA$ and $\AB$, the integrand in the term in square
brackets in \eqref{eq:induction_first} converges pointwise to a limit
and is dominated by
$\sup\{g(\bm \res)\,:\, \bm \res\in \RR^k\}\times \pi((v_u(\res_0),\bm
A_{\kkminone}(v_u(\res_0),\bm
\res_{\kkminone})),A_{k}(v_u(\res_0),d\res_k))$.\ Lebesgue's dominated
convergence theorem yields that the term in square brackets
of~\eqref{eq:induction_first} is bounded and converges to
$\mathbb{E}\left[g(\bm \Res_{\kplusonek})\right]$ for $u\to\infty$ since
$v_u(\res_0)\to \infty$ as $u\to \infty$.\ The convergence holds
uniformly in the variable $\res_0\in[0,\infty)$ since $\sigma(u)>0$.\
Therefore Lemma~\ref{lemma:unif_conv} applies, which guarantees
convergence of the entire term \eqref{eq:induction_first} to
$\mathbb{E}(E_0)\,\mathbb{E}(g(\bm \Res_{\kplusonek}))$ due to Assumption
$\AAA$.

Next, assume that the statement is true for some $t > k$. It suffices
to show that for any $g_0\in C_b([0,\infty)\times \RR^t$,
$g\in C_b(\RR)$,
\begin{IEEEeqnarray}{rCl} && \int_{[0,\infty)\times \RR^{t+1}} g_0(\bm
\res_{\tplusonetplusone}) g(\res_{t+1})
\mu_{t+1}^{(u)}(d\res_0,d\res_1,\hdots,d\res_{t+1})\nonumber\\\nonumber\\ &&=
\int_{[0,\infty)\times\RR^t } g_0(\bm \res_{\tplusonetplusone}) \left[\int_{\RR}
g(\res_{t+1})\pi(\bm A_{\tplusonek}(v_u(\res_0),\bm \res_{\tplusonek}),A_t(v_u(\res_0),d\res_{t+1}))
\right]\nonumber\\\nonumber\\
&&\qquad\qquad\qquad\qquad\qquad\qquad\qquad\qquad\qquad
\mu_{t}^{(u)}(d\res_0,d\res_1,\hdots,d\res_{t})
\label{eq:induction_step}
\end{IEEEeqnarray}
converges to
\begin{IEEEeqnarray}{rCl} && \int_{[0,\infty)\times \RR^{t+1}} g_0(\bm
\res_{\tplusonetplusone}) g(\res_{t+1})
\mu_{t+1}(d\res_0,d\res_1,\hdots,d\res_{t+1})\nonumber\\\nonumber\\
&&=\int_{[0,\infty)\times \RR^{t}} g_0(\bm \res_{\tplusonetplusone})
\left[\int_{\RR}g(\res_{t+1})K\left(\frac{d\res_{t+1} -
\psi^a_{t}(\bm \res_{\tplusonek})}
{\psi^b_{t}(\bm \res_{\tplusonek})}\right)\right]
\mu_{t}(d\res_0,d\res_1,\hdots,d\res_{t}).\nonumber\\\nonumber\\
\end{IEEEeqnarray} The term in square brackets
of~\eqref{eq:induction_step} is bounded, and by Lemma~\ref{lemma:1}
and Assumptions $\AA$ and $\AB$, it converges uniformly on compact
sets in both variables
$(\res_0,\bm \res_{\tplusonek})\in[0,\infty)\times \RR^k$ jointly, since
$\sigma(u)>0$.\ Hence the induction hypothesis and
Lemma~\ref{lemma:unif_conv} imply the desired result.
\end{proof}

\begin{proof}[Proof of Theorem~\ref{thm:tailchain:nonneg}] 
  Define
  \begin{equation*}
\bm b_{\tk}(v,\bm x_{\tk}) = (b_{t-k}(v)\,x_{t-k}
,\hdots, b_{t-1}(v)\,x_{t-1}).
  \end{equation*} Consider the measures 
  \begin{IEEEeqnarray}{rCl} \mu^{(u)}_t(d\res_0,\dots,d\res_t) &=&
\prod_{j=k}^t
\pi(\bm b_{\jk}(v_u(\res_0),\bm \res_{\jk}),b_j(v_u(\res_0)) \,d\res_j)\nonumber
\\ \nonumber\\
&& \left[\pi_0( v_u(\res_0),
\bm b_{\kkminone}(v_u(\res_0),d\bm \res_{\kkminone}))\right]^{1(k>1)}
\frac{F_0( v_u(d\res_0))}{\overline{F}_0(u)}
    \label{eq:firstmeasure:nonneg}
  \end{IEEEeqnarray} and
  \begin{IEEEeqnarray}{rCl} \mu_t(d\res_0,\dots,d\res_t)&=& \prod_{j=k}^t
    K\left(\frac{d\res_j}{\psi^\scale_{j}(\bm \res_{\jk})}\right)
    \left[G(d\res_1, \ldots, d\res_{k-1})\right]^{1(k>1)}
    H_0(d\res_0),
    \label{eq:secondmeasure:nonneg}
  \end{IEEEeqnarray} 
  on $[0,\infty) \times [0,\infty)^t$. We may write
  \begin{equation*} \EE \left[ f\left(\frac{X_0 -
u}{\sigma(u)},\frac{X_1}{\scale_1(X_0)},\dots,\frac{X_t}{\scale_t(X_0)}\right)
\,\bigg\vert\, X_0 > u\right] = \int_{[0,\infty) \times [0,\infty)^t}
f\left(\bm \res_{\tplusonetplusone}\right) \mu^{(u)}_t(d\res_0,\dots,d\res_t)
  \end{equation*} and
  \begin{align*} \EE \left[ f\left(E_0,\Res_1,\dots,\Res_t\right) \right] =
\int_{[0,\infty) \times [0,\infty)^t} f\left(\bm \res_{\tplusonetplusone}\right)
\mu_t(d\res_0,\dots,d\res_t)
  \end{align*} for $f \in C_b([0,\infty)\times [0,\infty)^t)$.\ Note
  that $b_j(0)$, $j=1,\dots,t$ need not be defined in
  (\ref{eq:firstmeasure:nonneg}), since $v_u(\res_0)\geq u>0$ for $\res_0\geq
  0$ and sufficiently large $u$, whereas (\ref{eq:secondmeasure:nonneg})
  is well-defined, since the measures $G$ and $K$ put no mass at any
  half-plane $C_j=\{(\bm \res_{\kkminone})\in [0,\infty)^{k-1}: \text{$\res_j =
    0$}\}\in [0,\infty)^{k-1}$ and at $0\in [0,\infty)$ respectively.\
  Formally, we may set $\psi^\scale_{j}(\bm 0)=1$,
  $j=1,\dots,t$ in order to emphasize that we consider measures on
  $[0,\infty)^{t+1}$, instead of $[0,\infty)\times (0,\infty)^t$.\ To
  prove the theorem, we need to show that $\mu^{(u)}_t(d\res_0,\dots,d\res_t)$
  converges weakly to $\mu_t(d\res_0,\dots,d\res_t)$.\ The proof is by
  induction on $t$.\ We show two statements by induction on $t$:
  \begin{description}[wide=0\parindent]
  \item[(I)] $\mu^{(u)}_t(d\res_0,\dots,d\res_t)$ converges weakly to
$\mu_t(d\res_0,\dots,d\res_t)$ as $u \uparrow \infty$.
  \item[(II)] For all $\varepsilon>0$ there exists $\delta_t>0$ such
that $\mu_t([0,\infty)\times [0,\infty)^{t-1} \times
[0,\delta_t])<\eps$.
  \end{description} We start proving the case $t=k$.  \\
  \noindent {(I) for $t=k$:} It suffices to show that for any
  $g_0 \in C_b([0,\infty))$ and $g \in
  C_b([0,\infty)^{k-1})$ 
  \begin{align}
\notag&\int_{[0,\infty)\times [0,\infty)^{k-1} \times [0,\infty)}
g_0(\res_0) g(\bm \res_{\kplusonek}) \mu^{(u)}_1(d\res_0,\ldots, d\res_k)\\
=&\int_{[0,\infty)} g_0(\res_0) \left[\int_{[0,\infty)^{k-1}}
\int_{[0,\infty)} g(\res_1,\ldots,\res_{k})
\pi((v_u(\res_0),\bm b_{\kk}(v_u(\res_0),\bm \res_{\kk})),b_{k}(v_u(\res_0))d\res_k)
\right]\nonumber \\ & \qquad \qquad \qquad \pi_0( v_u(\res_0),
\bm b_{\kkminone}(v_u(\res_0),d\bm \res_{\kkminone}))
                      \frac{F_0(v_u(d\res_0))}{\overline{F}_0(u)}
                      \label{eq:induction:start:lhs:nonneg}
  \end{align} converges to
  \[ \int_{[0,\infty)\times [0,\infty)^{k}} g_0(\res_0) g(\bm \res_{\kplusonek})
\mu_1(d\res_0,\ldots, d\res_k)=\EE(g_0(E_0))
\EE\left(g(\Res_1,\ldots,\Res_k)\right).
\] By Assumptions $\AAA$ and $\BB$, the integrand in the term in
square brackets converges pointwise to a limit and is dominated by
  \[ \sup\{g(\bm \res)\,:\, \bm \res\in \RR_+^k\}\times
\pi((v_u(\res_0),\bm b_{\kkminone}(v_u(\res_0),\bm \res_{\kkminone})),b_{k}(v_u(\res_0))\,d\res_k).
\] Lebesgue's dominated convergence theorem yields that the term in
square brackets of~\eqref{eq:induction:start:lhs:nonneg} is bounded
and converges to $\EE(g(\bm \Res_{\kplusonek}))$ for $u\uparrow \infty$,
since $v_u(\res_0) \to \infty$ for $u \uparrow \infty$.\ 
The convergence is uniform in the variable $\res_0$, since
$\sigma(u)>0$.\ Therefore, Lemma~\ref{lemma:unif_conv} (i) applies,
which guarantees convergence of the entire term
(\ref{eq:induction:start:lhs:nonneg}) to
$\EE(g_0(E_0)) \EE\left[g(\bm \Res_{\kplusonek})\right]$ due
to Assumption $\AAA$.  \\
  \noindent {(II) for $t=k$:} Since $K(\{0\})=0$, there exists
$\delta>0$ such that $K([0,\delta])<\eps$, which immediately entails
$\mu_k([0,\infty)^{k}\times[0,\delta])=H_0([0,\infty))\,
\smash{\left[G([0,\infty)^{k-1})\right]^{ 1(k>1)
}}\,K([0,\delta])<\epsilon$.\\

  \noindent Now, let us assume that both statements ((I) and (II)) are
proved for some $t \in \NN$.  \\
  
\noindent {(I) for $t+1$:} It suffices to show that for any
$g_0 \in C_b([0,\infty)\times [0,\infty)^t)$, $g\in C_b([0,\infty))$
  \begin{align}\label{eq:induction:step:nonneg} \notag
    &\int_{[0,\infty)\times [0,\infty)^{t+1}} g_0(\bm \res_{\tplusonetplusone})
      g(\res_{t+1}) \mu^{(u)}_{t+1}(d\res_0,d\res_1,\dots,d\res_t,d\res_{t+1})\\ \notag
    &=\int_{[0,\infty)\times [0,\infty)^{t}} g_0(\bm \res_{\tplusonetplusone})
      \left[\int_{[0,\infty)} g(\res_{t+1})
      \pi(\bm b_{\tplusonek}(v_u(\res_0), \bm \res_{\tplusonek}),
      b_{t+1}(v_u(\res_0))\,d\res_{t+1}) \right]\\
    &\hspace{9.5cm}\mu^{(u)}_{t}(d\res_0,d\res_1,\dots,d\res_t)
  \end{align} converges to
  \begin{IEEEeqnarray}{rCl}
    \notag&\int_{[0,\infty)\times [0,\infty)^{t+1}} g_0(\bm \res_{\tplusonetplusone})
            g(\res_{t+1}) \mu_{t+1}(d\res_0,d\res_1,\dots,d\res_t,d\res_{t+1})\\ 
          &=\int_{[0,\infty)\times [0,\infty)^{t}} g_0(\bm \res_{\tplusonetplusone})
            \left[\int_{[0,\infty)}g(\res_{t+1})
            K\left(d\res_{t+1}/\psi_\tpo^\scale(\bm \res_{\tplusonek})\right)
            \right] \mu_{t}(d\res_0,d\res_1,\dots,d\res_t).
            \label{eq:induction:step:limit:nonneg}
  \end{IEEEeqnarray}
  From Lemma~\ref{lemma:2} and Assumptions $\BA$ and $\BB$ we know
  that, for any $\delta>0$, the (bounded) term in the square brackets
  of (\ref{eq:induction:step:nonneg}) converges uniformly on compact
  sets in the variable
  $\bm \res_{\tplusonek} \in \prod_{i=1}^k[\delta_i,\infty)$ to the
  continuous function
  \[\int_{[0,\infty)} g(\psi_\tpo^\scale(\bm \res_{\tplusonek})
    \res_{t+1}) K(d\res_{t+1})\] (the term in the square brackets of
  (\ref{eq:induction:step:limit:nonneg})). This convergence holds even
  uniformly on compact sets in both variables
  $(\res_0,\bm \res_{\tplusonek}) \in [0,\infty)\times
  \prod_{i=1}^k[\delta_i,\infty)$ jointly, since $\sigma(u)>0$.
  Hence, the induction hypothesis (I) and Lemma~\ref{lemma:unif_conv}
  (i) imply that for any $\delta>0$ the integral in
  (\ref{eq:induction:step:nonneg}) converges to the integral in
  (\ref{eq:induction:step:limit:nonneg}) if the integrals with respect
  to $\mu_t$ and $\mu_t^{(u)}$ were restricted to
  $A_\delta:=[0,\infty)\times [0,\infty)^{t-1} \times [\delta,\infty)$
  (instead of integration over
  $[0,\infty)\times [0,\infty)^{t-1} \times [0,\infty)$).
                
  Since $g_0$ and $g$ are bounded, it suffices to control the mass of
$\mu_t$ and $\mu_t^{(u)}$ on the complement
$A_\delta^c=[0,\infty)\times [0,\infty)^{t-1} \times [0,\delta)$.  For
some prescribed $\eps>0$ it is possible to find some sufficiently
small $\delta>0$ and sufficiently large $u$, such that
$\mu_t(A_\delta^c)<\eps$ and $\mu^{(u)}_t(A^c_\delta)<2\eps$.  Because
of the induction hypothesis (II), we have indeed
$\mu_t(A_{\delta_t})<\eps$ for some $\delta_t>0$.  Choose
$\delta=\delta_t/2$ and note that the sets of the form $A_\delta$ are
nested. Let $C_\delta$ be a continuity set of $\mu_t$ with $A^c_\delta
\subset C_\delta \subset A^c_{2\delta}$.  Then the value of $\mu_t$ on
all three sets $A^c_\delta,C_\delta,A^c_{2\delta}$ is smaller than
$\eps$ and because of the induction hypothesis (I), the value
$\mu^{(u)}_t(C_\delta)$ converges to $\mu_t(C_\delta)<\eps$. Hence,
for sufficiently large $u$, we also have
$\mu^{(u)}_t(A^c_\delta)<\mu^{(u)}_{t}(C_\delta)<\mu_t(C_\delta)+\eps<2\eps$,
as desired.\\

\noindent \underline{(II) for $t+1$:} We have for any $\delta>0$ and
any $c>0$
\begin{align*}
  &\mu_{t+1}([0,\infty)\times[0,\infty)^t \times [0,\delta]) =
    \int_{[0,\infty)\times[0,\infty)^t}
    K\left(\left[0,\delta/\psi_\tpo^\scale(\bm \res_{\tplusonek})\right]\right)
    \mu_t(d\res_0,\dots,d\res_t).
\end{align*}
Splitting the integral according to
$\{\psi_\tpo^\scale(\bm \res_{\tplusonek})>c\}$ or
$\{\psi_\tpo^\scale(\bm \res_{\tplusonek})\leq c\}$ yields
\begin{align*}
  &\mu_{t+1}([0,\infty)\times[0,\infty)^t \times [0,\delta]) \leq
    K\left(\left[0,\delta/c\right]\right) + \mu_t( [0,\infty)\times
    [0,\infty)^{t-1}\times (\psi_\tpo^\scale)^{-1}([0,c])\}).
\end{align*}
By Assumption $\BB$ (i) and the induction hypothesis (II) we may
choose $c>0$ sufficiently small, such that the second summand
$\mu_t([0,\infty)\times [0,\infty)^{t-1}\times
(\psi_\tpo^\scale)^{-1}([0,c])\})$ is smaller than
$\eps/2$. Secondly, since $K(\{0\})=0$, it is possible to choose
$\delta_{t+1}=\delta>0$ accordingly small, such that the first
summand $K\left(\left[0,\frac{\delta}{c}\right]\right)$ is smaller
than $\eps/2$, which shows (II) for $t+1$.
\end{proof}
\subsection{Proof of Proposition \ref{lemma:doa_equivalence}}
\label{sec:proof_doa_equivalence}
\begin{proof}
  We start by proving that $(i)$ implies $(ii)$.\ Suppose there exist
  $a_t$, $b_t$ $\psi_t^a$ and $\psi_t^b$ such that $(i)$ holds.\ Then,
  for $t=k,k+1\dots$,
\begin{IEEEeqnarray*}{rCl}
  \frac{a(\bm A_t(\level, \bm \res_\level))-a(\bm A_t(\level, \bm
    0))}{b(\bm A_t(\level, \bm 0))} &=& \frac{a(\bm A_t(\level, \bm
    \res_\level))-a_t(\level)}{b_t(\level)}\frac{b_t(\level)}{b(\bm
    A_t(\level,\bm 0))} - \frac{a(\bm A_t(\level, \bm
    0))-a_t(\level)}{b_t(\level)}\frac{b_t(\level)}{b(\bm
    A_t(\level,\bm 0))}\\
  &\to& \frac{\psi_t^a(\bm \res)}{\psi_t^b(\bm 0)} -
  \frac{\psi_t^a(\bm 0)}{\psi_t^b(\bm 0)} = \frac{\psi_t^a(\bm \res) -
    \psi_t^a(\bm 0)}{\psi_t^b(\bm 0)}, \quad \text{whenever
    $\bm z_\level \to \bm z$ as $v\to\infty$},
\end{IEEEeqnarray*}
and
\begin{IEEEeqnarray*}{rCl}
  \frac{b(\bm A_t(\level, \bm \res_\level))}{b(\bm A_t(\level, \bm
    0))} &=& \frac{b(\bm A_t(\level, \bm \res_\level))}{b_t(\level)}
  \, \frac{b_t(\level)}{b(\bm A_t(\level, \bm 0))} \to
  \frac{\psi_t^b(\bm \res)}{\psi_t^b(\bm 0)}, \quad \text{whenever
    $\bm z_\level \to \bm z$ as $v\to\infty$}.
\end{IEEEeqnarray*}
Next we prove $(ii)$ implies $(i)$.\ Let
$a_t(\level) = a(\bm A_t(\level, 0)) - c\,b(\bm A_t(\level, \bm 0))$
and $b_t(\level)=d\, b(\bm A_t(\level, \bm 0))$ for arbitrary
constants $c\in \RR$, $d\in \RR_+$.\ Then, for $t=k,k+1\dots$,
\[
  \frac{a(\bm A_t(\level, \bm \res_\level))-a_t(\level)}{b_t(\level)} = \frac{b(\bm A_t(\level,\bm
    0))}{b_t(\level)} \left[\frac{a(\bm A_t(\level, \bm \res_\level))-a(\bm A_t(\level, \bm
      0))}{b(\bm A_t(\level,\bm 0))} + c\right]\to
  \frac{\lambda_t^b(\bm \res)}{d}\left[\lambda_t^a(\bm \res)+c\right],
\]
whenever $\bm \res_\level\to\bm \res\in\RR^k$ as $\level\to\infty$.\
Define $\psi_t^a(\bm \res) = (\lambda_t^a(\bm \res) + c)/d$ and
$\psi_t^b(\bm \res)=\lambda_t^b(\bm \res)/d$. By assumption,
$\lambda_t^a(\bm 0)=0$ and $\lambda_t^b(\bm 0) = 1$. Hence,
$\lambda_t^a(\bm \res) = [\psi_t^a(\bm \res) - \psi_t^a(\bm
0)]/\psi_t^b(\bm 0)$ and
$\lambda_t^b(\bm \res)=\psi_t^b(\bm \res)/\psi_t^b(\bm 0)$, which
completes the proof.\
\end{proof}

\subsection{Proof of Corollaries}
\begin{proof}[Proof of Corollary~\ref{cor:SDE_AD}]
  Since $a$ is continuous, we have that
  $a(\level\bm 1 + \bm \res_\level) - \level = a(\bm \res_\level)\to
  a(\bm \res)$, whenever $\bm \res_\level\to\bm \res\in\RR^k$.\ Hence,
  convergence \eqref{eq:psi_functions} holds true with
  $\psi_t^a(\bm \res)=a(\bm \res)$ and $\psi_t^b(\bm \res)=1$.

  For any $s\in \NN$, \cite{asensege2022} show that under the
  assumptions of Corollary~\ref{cor:SDE_AI}, the random vector
  $\bm X_{P}= \exp(\bm X_{0\,:\, s})$ is multivariate regularly
  varying, that is, for any $A\subset \SI_{s}$, 
  \[
    \PP\left(\frac{\bm X_{P}}{\lVert \bm
        X_{P}\rVert}\in A, \lVert\bm X_{P}\rVert >
      \level\, r ~\Big|~ \lVert\bm X_{P}\rVert > \level\right)
    \rightarrow H(A)\, r^{-1}, \quad r\geq 1,
  \]
  where $H$ is a Radon measure on $\SI_{s}$ satisfying $H(\SI_s)=s+1$
  and $\int_{\SI_s} w_j dH(\bm w_{0\,:\,s}) = 1$ for $j=0,\dots,s$.\
  Theorem \ref{thm:tailchain} and Proposition 4 of \cite{heffres07}
  imply that $\Res_t\sim G_t$ where
  \[
    G_t(z)=\int_{0}^{q(z)}(1-w) H_t(dw), \quad \text{with }
    q(z)=e^z/(1+e^z), \quad z\in\RR,
  \]
  where $H_t$ denotes the lag $t$ bivariate spectral measure
  associated with $H$, that is, for every $t \geq 1$, $H_t$ is a Radon
  measure on $\SI_1$ that satisfies $H_t(\SI_1)=2$ and
  $\int_{\SI_1} w dH_t(w)=1$.\ Thus, we have that the expected value
  of $\Res_t$ satisfies
  \begin{IEEEeqnarray*}{rCl}
    \mathbb{E}(\Res_t)&=&\int_\RR z \,\mathrm{d}G_t(z)= \int_\RR \int_{0}^z \,\mathrm{d}u\,
    \,\mathrm{d}G_t(z) =-\int_{-\infty}^0 G_t(z)\,\mathrm{d}{z} +
    \int_{0}^{\infty} [1 - G_t(z)]\,\mathrm{d}{z}=\\
    &=&-\int_{0}^{1/2} \int_0^u (1-w)\,\mathrm{d}
    H_t(w)\,\mathrm{d}\log\{u/(1-u)\} + \int_{1/2}^{1}
    \int_u^1 (1-w)\,\mathrm{d} H_t(w)\,\mathrm{d}\log\{u/(1-u)\} =\\
    &=&-\int_{0}^{1/2} \int_w^{1/2} [\mathrm{d}\log\{u/(1-u)\}]\,
    (1-w)\,\mathrm{d} H_t(w) + \int_{1/2}^{1}
    \int_{1/2}^w [\mathrm{d}\log\{u/(1-u)\}]\,(1-w)\,\mathrm{d} H_t(w) =\\
    &=& \int_{\SI_1} \log\{w/(1-w)\}\, (1-w)\,\mathrm{d} H_t(w) < \log
    \left[\int_{\SI_1} w\,\mathrm{d} H_t(w)\right]=0.
  \end{IEEEeqnarray*}
  The strict inequality follows from Jensen's inequality, the strict
  concavity of the $\log$ function, and due to
  Corollary~\ref{cor:SDE_AI} which requires $G_t$ to be a
  non-degenerate distribution.\ The latter ensures
  $H_t\neq 2\,\delta_{1/2}$, where $\delta_{x}$ denotes the Dirac
  measure at $\{x\}$.

\end{proof}

\begin{proof}[Proof of Corollary~\ref{cor:SDE_AI}]
  It suffices to show that for $a_t(x)=\alpha_t \,x$ and
  $b_t(x)=x^\beta$, with $\alpha_t$ given by expression
  \eqref{eq:recurrence_slope}, then convergence
  \eqref{eq:psi_functions} holds true with
  $\psi_t^a(\bm \res)=\nabla a(\bm \alpha_{\tk})\cdot \bm \res$ and
  $\psi_t^b(\bm \res)=b(\bm \alpha_{\tk})$, and that $\alpha_t\to 0$
  as $t\to\infty$.\ Since $a$ is twice continuously differentiable, we
  have
  \begin{IEEEeqnarray*}{rCl}
    & & v^{-\beta}[a(\bm \alpha_{\tk}\,v + v^\beta\,\bm \res_v)- a(\bm
    \alpha_{\tk}\,v)] =  \nabla a(\bm
    \alpha_{\tk}\,v)\cdot \bm \res_v + v^{2\beta-1} \bm \res_v^\top
    \nabla
    \nabla^\top a(\bm \alpha_{\tk}v) \bm \res_v + O(v^{3(\beta-1)})\\
    && = \nabla a(\bm \alpha_{\tk})\cdot \bm \res_v + v^{\beta-1} \bm
    \res_v^\top \nabla \nabla^\top a(\bm \alpha_{\tk}) \bm \res_v +
    o(v^{\beta-1}), \quad \text{as $v\to\infty$},
  \end{IEEEeqnarray*}
  where the last equality follows because $a$ is $1$--homogeneous,
  which gives that
  $\nabla a(\bm \alpha_{\tk} v)=\nabla a(\bm \alpha_{\tk})$ and
  $\nabla\nabla^\top a(\bm \alpha_{\tk} v)=v^{-1}\nabla\nabla^\top
  a(\bm \alpha_{\tk})$.\ Similarly, because $b$ is continuous and
  $\beta$--homogeneous with $\beta\in[0,1)$, this gives
  $v^{-\beta}b(\bm \alpha_{\tk}\,v + v^\beta\,\bm \res_v) = b(\bm
  \alpha_{\tk} + v^{\beta-1}\,\bm \res_v) \to b(\bm \alpha_{\tk})$, as
  $v\to\infty$.\ Hence, convergence \eqref{eq:psi_functions} holds
  true with
  $\psi_t^a(\bm \res)=\nabla a(\bm \alpha_{\tk})\cdot \bm \res$ and
  $\psi_t^b(\bm \res)=b(\bm \alpha_{\tk})$.


  Lastly, we show that $\alpha_t\to 0$ as $t\to \infty$.\ Suppose
  there exists a fixed point $\bm x_{1\,:\,k}^\star$ of $f$, that is,
  $\bm x_{1\,:\,k}^\star$ satisfies
  $f(\bm x_{1\,:\,k}^\star)=(\bm x_{2\,:k}^\star, a(\bm
  x_{1\,:\,k}^\star))=\bm x_{1\,:\,k}^\star$.\ Equating terms
  element-wise gives that $\bm x_{1\,:\,k}^\star = x\bm 1$ and
  $a(x \bm 1)=x$ for some $x\in\RR$. By assumption, $a$ is continuous
  and 1--homogeneous.\ This gives that $a(\bm 0)=0$ and hence, $\bm 0$
  is a fixed point, that is $f(\bm 0)=\bm 0$.\ By assumption, we also
  have $a(\bm 1)\neq 1$ and therefore, $\bm 0$ is the unique fixed
  point of $f$.\ Thus, $\bm x_{1\,:\,k}^\star$ exists, is unique, and
  satisfies $\bm x_{1\,:\,k}^\star=\bm 0$.\ The condition
  $\rho(\bm J_f(\bm 0)) < 1$ ensures $\bm x_{1\,:\,k}^\star$ is a
  stable and attracting fixed point in the Lyapunov sense.\ Because
  $f$ is a homogeneous map, a simple adaptation of Theorem 17.1 of
  \cite{hahn67} to the case of autonomous discrete dynamical systems
  gives that $\bm x_{1\,:\,k}^\star$ is globally asymptotically
  stable.\ The latter implies that $f^t(\bm x) \to \bm x^\star$ as
  $t\to\infty$ for all $\bm x\in\RR^k$, where $f^t$, $t\geq 2$,
  denotes the $t$th functional power of $f$ defined by
  $f^t = f \circ f^{t-1}$.\ Therefore
  $f^t(1,\alpha_1,\dots,\alpha_{k-1})=(\alpha_{t},\dots,\alpha_{t+k-1})\to
  \bm 0$ as $t\to \infty$, which gives $\alpha_t \to 0$ as
  $t\to \infty$.
\end{proof}

\begin{proof}[Proof of Corollary~\ref{cor:SDE_AI_scale}]
  Let
  $\log \beta_t = \log \beta +
  \log\left(\max_{i=1,\ldots,k}\beta_{t-i}\right)$ and consider
  convergence \eqref{eq:psi_functions_scale}.\ For $t=k$, we have
  \[
    \level^{-\beta}b(\level \res_0, \level^\beta \res_1, \dots, \level^{\beta} \res_{k-1} ) =
    b(\res_0, \level^{\beta-1} \res_1, \dots, \level^{\beta-1} \res_{k-1} ) \to
    b(\res_0, \bm 0_{k-1}), \quad \text{as $\level\to\infty$}.
  \]
  For $t=k+1$ we have
  \[
    \level^{-\beta^2}b(\level^\beta \res_0, \level^\beta \res_1, \dots, \level^{\beta} \res_{k-1} )
    = b(\res_0, \res_1, \dots, \res_{k-1} ).
  \]
  and for $t=k + j$ with $j\in\{2,\dots,k-1\}$, we have
  \begin{IEEEeqnarray*}{rCl}
    \level^{-\beta^2}b(\level^\beta \res_0, \dots \level^{\beta}
    \res_{k-j}, \level^{\beta^2} \res_{k-j+1}\dots, \level^{\beta^2}
    \res_{k-1} ) &=& b(\res_0, \dots \res_{k-j},
    \level^{\beta(\beta-1)} \res_{k-j+1},\dots, \level^{\beta(\beta-1)} \res_{k-1} )\\
    &\to& b(\res_0,\hdots, \res_{k - j}, \bm 0_{j-1}), \quad \text{as
      $\level\to\infty$}.
  \end{IEEEeqnarray*}
  Iterating forwards for $t=2k,2k+1,\dots$, we see that convergence
  \eqref{eq:psi_functions_scale} holds with
  \begin{IEEEeqnarray*}{rCl}
    \psi_t^b(\bm \res)&=&
    \begin{cases}
      b(\res_{0}, \bm 0_{k-1})
      & \text{when $\text{mod}_k(t)=0$}\\
      b(\res_{0},\hdots, \res_{k-1})
      & \text{when $\text{mod}_k(t) = 1$}\\
      b(\res_{0},\hdots, \res_{k - j}, \bm 0_{j-1})
      & \text{when
        $\text{mod}_k(t) = j \in
        \{2,\ldots,k-1\}$},
    \end{cases}
  \end{IEEEeqnarray*}
  which completes the proof.
\end{proof}

\subsection{Proof of Proposition~\ref{prop:closedform} }
\begin{proof}[Proof of Proposition~\ref{prop:closedform}]
  The recurrence relation
  \smash{$\alpha_t = c \, [\sum_{i=1}^k \gamma_i (\gamma_i\,
  \alpha_{t-k+i-1})^\delta ]^{1/\delta}$}
can be converted to the order-$k$ homogeneous linear recurrence
relation $y_t = \sum_{i=1}^k c_i \, y_{t-k+i-1}$, where
$\{y_t\} = \{\alpha_{t}^{\delta}\}$ and
$c_i=c^\delta \gamma_i^{1+\delta}$.\ Solving the linear recurrence
relation and transforming the solution to the original sequence
$\{\alpha_t\}$ leads to the claim.
\end{proof}

\subsection{Proof of Proposition~\ref{prop:max_stable}}
\begin{proof}[Proof of Proposition~\ref{prop:max_stable}]
  \begin{description}[wide=0\parindent]
  \item[$(i)$] 
    Because $a_P$ is $1$--homogeneous, $a$ satisfies
    property~\eqref{eq:ams_condition}.\ By definition of the
    right-inverse $V_{\kk}^{\leftarrow}$ and due to Property $K_1$, we
    have that, for all $(\bm y_\kk,q)$ in the domain of
    $V_\kk^\leftarrow$,
    \[
      V_{\kk}\{\bm y_\kk, V_{\kk}^{\leftarrow}(q;\bm y_\kk)\}
      =V_\kk(\bm y_\kk,\infty) \, K_P\{V_\kk^{\leftarrow}(q;\bm
      y_{\kk})/a_P(\bm y_\kk)\} =q.
    \] Hence,
    $V_\kk^{\leftarrow}(q;\bm y_{\kk}) = a_P(\bm y_\kk)\,
    K_P^{\leftarrow}\{q/V_\kk(\bm y_\kk,\infty)\}$.\ Taking
    logarithms, setting $\bm y_\kk=e^{\bm x_\kk}$, and letting
    $q=p^\star\,V_\kk(\bm y_\kk,\infty)$ where $p^\star$ satisfies
    $K_P^\leftarrow(p^\star)=1$ gives
    \begin{IEEEeqnarray*}{rCl} a(\bm x_\kk) &=& \log
      V_\kk^\leftarrow(p^\star\,V_\kk(e^{\bm x_{\kk}},\infty);e^{\bm
        x_{\kk}}) - \log K_P^{\leftarrow}(p^\star).
    \end{IEEEeqnarray*}
  \item[$(ii)$] Since $a$ satisfies property~\eqref{eq:ams_condition},
    then Appendix~\ref{sec:proof_ms} shows that Assumption $\AB$ holds
    with limit distribution $K(x;\bm \res_{\kk})$ given by expression
    \eqref{eq:ms_convergence_example}.\ Using Property $K_1$, we
    further have that under the normalizing functionals $a$ and $b$,
    expression \eqref{eq:ms_convergence_example} simplifies to
    $K(x;\bm \res_{\kk})=K_P(e^x)$ for all $\bm \res_\kk\in\RR^k$.
  \item[$(iii)$] The limit distribution $K(x;\bm \res_{\kk})$ in
    expression \eqref{eq:ms_convergence_example} does not depend on
    $\bm \res_{\kk}$.\ Because $a_P$ is positive, $1$--homogeneous, and
    continuous, $1$ is in the image of $a_P$ and thus, $0$ is in the
    image of $a$.\ Hence, there exists $\bm z_\kk^\star\in\RR^k$ such
    that $a(\bm z_\kk^\star)=0$ and $K_P(e^x)=K(x;\bm z_\kk^\star)$ which
    proves the claim.
  \end{description}
\end{proof} \color{black}

\subsection{Convergence of multivariate normal full conditional
distribution}
\label{sec:kernel_MVN} Let $\bm X_N = (X_{N,0}, \dots, X_{N,k}) \sim
\mathcal{N}(\bm 0_{k+1}, \bm \Sigma)$ where $\bm
\Sigma\in\RR^{(k+1)\times (k+1)}$ is a positive definite correlation
matrix with $(i+1,j+1)$ element $\rho_{ij}$, $i,j=0,\ldots,k$.\ Let
$\bm Q= \bm \Sigma^{-1}$ and write $q_{ij}$ for its $(i+1,j+1)$
element, $i,j=0,\ldots,k$.\ For $k\geq 1$ and $\bm \res_{\kk} \in
\RR^{k}$, the conditional distribution of $X_{N,k}$ given $\bm
X_{N,\kk} = \bm \res_\kk$ is normal with mean $-q_{kk}^{-1}
\sum_{i=0}^{k-1} q_{ik} \,\res_i$ and variance $q_{kk}^{-1}$.\ Let
$\bm X=(X_0,\dots, X_k)$ with $X_i=-\log\{1-\Phi(X_{N,i})\}$, so that
$X_i\sim \text{Exp}(1)$ for $i=0,\dots,k$.\ Following the strategy
outlined in Section~\ref{sec:strategy}, we have that for any $t\geq k
\geq 1$,
\begin{equation} \Pr(X_k < a(\bm X_{\kk}) \mid \bm X_{\kk}= \bm
A_t(\level,\bm \res_{\kk})) =
\Phi\left[q_{kk}^{1/2}\left\{\Phi_{t,\level}^{-1} -
\sum_{i=0}^{k-1}\left(-\frac{q_{ik}}{q_{kk}}\right)\,\Phi_{t-k+i,\level}^{-1}\right\}\right],
  \label{eq:full_conditional_Gaussian}
\end{equation} where, for $i=t-1,\dots, t-k$, $\Phi_{i,\level}^{-1} :=
\Phi^{\leftarrow} \left[1-\exp\left\{-A_{i}(\level,
\res_i)\right\}\right]$ and $\Phi_{t,\level}^{-1} := \Phi^{\leftarrow}
\left\{1-\exp\left(-\left[a\{\bm A_t(\level,\bm \res_{\kk})\}
\right]\right)\right\}$.\

Now, let $t=k$.\ Firstly, we seek to find a function $a$ such that the
conditional probability in equation
\eqref{eq:full_conditional_Gaussian} converges to a number
$p\in(0,1)$.\ Suppose that this function satisfies $a(\bm a_\tk(v))\to
\infty$ as $v\to \infty$.\ Using standard asymptotic series for the
cumulative distribution function of the standard normal distribution,
we have
\begin{IEEEeqnarray*}{rCl} \Phi_{i,\level}^{-1} &=& \{2\,A_{i}(\level,
\res_i)\}^{1/2} - \frac{\log A_{i}(\level, \res_i) + \log
4\pi}{2\{2\,A_{i}(\level,\res_i)\}^{1/2}}+o(A_{i}(\level,\res_i)^{-1/2}),\qquad
i=t-k,\dots, t-1,\\ \Phi_{t,\level}^{-1} &=& \{2\,a(\bm A_{t}(\level,
\bm \res_{\kk}))\}^{1/2} - \frac{\log a(\bm A_{t}(\level, \bm
\res_{\kk})) + \log 4\pi}{2\{2\,a(\bm A_{t}(\bm \res_{\kk}))\}^{1/2}}
+o(a(\bm A_{t}(\level, \bm \res_{\kk}))^{-1/2}),
\end{IEEEeqnarray*} as $v\to \infty$.\ Therefore,
\begin{equation} \Phi_{t,\level}^{-1} -
\sum_{i=0}^{k-1}\left(-\frac{q_{ik}}{q_{kk}}\right)\,\Phi_{t-k+i,\level}^{-1}=
\left(2 \left[a\{\bm A_{t}(\level, \bm
\res_{\kk})\}\right]\right)^{1/2} -
\sum_{i=0}^{k-1}\left(-\frac{q_{ik}}{q_{kk}}\right)
\{2\,A_{t-k+i,\level}( \res_i)\}^{1/2} + o(1),
  \label{eq:Mills_difference}
\end{equation} as $\level\to\infty$.\ Substituting in
\eqref{eq:full_conditional_Gaussian}, we observe that for the choice
of $a$ being $a(\bm y_{\kk}) =
\big\{\sum_{i=0}^{k-1}(-q_{ik}/q_{kk})\,|y_i|^{1/2} \big\}^2$, for
$\bm y_{\kk} \in \RR^k$, the conditions set out in Section
\ref{sec:strategy} are met.\ In particular, due to expression
\eqref{eq:Mills_difference} converging to zero and since $\Phi$ is
continuous, we have the conditional probability
\eqref{eq:full_conditional_Gaussian} converging to $p=1/2$, that is
\[ \lim_{\level\to
\infty}\Phi\left[q_{kk}^{1/2}\left\{\Phi_{t,\level}^{-1} -
\sum_{i=0}^{k-1}\left(-\frac{q_{ik}}{q_{kk}}\right)\,\Phi_{t-k+i,\level}^{-1}\right\}\right]
= 1/2.
\] Using similar asymptotic series, we have that for $b(\bm y_{\kk}) =
a(\bm y_{\kk})^{1/2}$ and any $x_k \in \RR$,
\begin{equation} \lim_{\level \to \infty}\PR\left(X_k < a(\bm
X_{\kk})+b(\bm X_{\kk})\,x_{k}~\Big |~ \bm X_{\kk}= \bm A_t(\level,
\bm \res_\kk)\right)=\Phi\left\{(q_{kk}/2)^{1/2}\,x_k \right\}.
 \label{eq:mvn_conv_proof}
\end{equation} The convergence in \eqref{eq:mvn_conv_proof} holds
uniformly on compact sets in the variable $\bm \res_{\kk}$ by
continuous convergence \citep[see Section 0.1 in][]{resn87}. That is,
expression \eqref{eq:mvn_conv_proof} holds true after replacing $\bm
\res_{\kk}$ by $\bm \res_{\kk}(\level)$ satisfying $\bm
\res_{\kk}(\level) \to \bm \res_{\kk}$ as $\level\to \infty$ and since
the limit function is continuous in $\bm \res_{\kk}$ (constant
function), the argument follows.\ Additionally, we have that for any
$\bm \alpha_{\tk}\in(0,1]^k$,
\begin{IEEEeqnarray*}{rCl} &&\lim_{\level\to\infty}
\level^{-1/2}\,[a(\bm \alpha_{\tk} \, \level+\level^{1/2}\bm
\res_{\kk})-a(\bm \alpha_{\tk} \, \level)] = \nabla a(\bm
\alpha_{\tk}) \cdot \bm \res_{\kk}\\ &&\lim_{\level\to\infty}
\level^{-1/2}\,b(\bm \alpha_{\tk} \, \level+\level^{1/2}\bm
\res_{\kk})=b(\bm \alpha_{\tk}),
\end{IEEEeqnarray*} where both convergences hold uniformly on compact
sets in the variable $\bm \res_{\kk}$ since monotone increasing
functions (in every argument) are converging point-wise to a
continuous limit.\ Thus, Assumption $\AB$ holds true for the special
case $t=k$ with
\[
a_t(\level)=\left\{\sum_{i=0}^{k-1}\left(-\frac{q_{ik}}{q_{kk}}\right)\,\rho_{t-k+i}^{1/2}\right\}^2\,\level
\quad \text{and}\quad b_t(\level)=\level^{1/2}.
\] Finally, observe that the entire argument after expression
\eqref{eq:full_conditional_Gaussian} remains unchanged upon changing
$t=k$ to $t+1 = k+1$.\ The claim is proved through iteration.  

\subsection{Convergence of multivariate inverted logistic full
  conditional distribution}
\label{sec:kernel_inv_log}
The transition probability kernel of this process is given by
expression~\eqref{eq:ims_kernel} with
$V(\bm x)=\lVert \bm x^{-1/\alpha}\rVert^{\alpha}$,
$\bm x \in \RR_+^{k+1}$. For $t\geq k\geq 1$ and
$\bm \res_{\kk} \in \RR_+^{k}$, we have that %
$\Pr(X_k/b(\bm X_{\kk})<1\mid \bm X_\kk=\bm B_{t}(\level, \bm
\res_{\kk}) )$ is equal to
\begin{equation}
  \mathscr{L}\left[v, \bm \res_{\kk}\right]
  \exp\left\{\big\lVert \bm B_{t}(\level,\bm
    x_{\kk})^{1/\alpha} \big\lVert^{\alpha} - \big\lVert (\bm
    B_{t}(\level,\bm \res_{\kk}), b\{\bm B_{t}(\level,\bm \res_{\kk})\})^{1/\alpha}
    \big\lVert^{\alpha}\right\},
  \label{eq:ims_conditional}
\end{equation}
where
$\bm B_{t}(\level,\bm \res_{\kk}) = (B_{t-k}(\level, \res_{0}),\dots,
B_{t-1}(\level,\res_{k-1})) $ and $\mathscr{L}(v, \bm \res_{\kk})=1+o(1)$
for all $\bm \res_{\kk}\in \RR_+^k$ as $v\to\infty$.

Now let $t=k$ and set $z_0 = 1$.\ Firstly, we seek to find a function
$b$ such that the conditional probability in equation
\eqref{eq:ims_conditional} converges to a number $p\in(0,1)$.\ Suppose
that this function satisfies $b(\bm b_{\kk}(\level))\to \infty$ as
$\level\to \infty$ with $b(\bm b_{\kk}(\level)) = o(\level)$.\ Under
this assumption, we have that as $\level \to \infty$,
\begin{IEEEeqnarray*}{rCl}
  \lefteqn{\log \overline{\pi}^{\text{inv}} [\bm B_{t}(\level,\bm
    \res_{\kk}), b\{\bm B(\level,\bm \res_{\kk})\}]=} \\\\ &=& \lVert
  \{\bm B_{t} (\level,\bm \res_{\kk})\}^{1/\alpha} \lVert^\alpha - \lVert
  [\{\bm B_{t}(\level,\bm \res_{\kk})\}^{1/\alpha}, \{b(\bm
  B_{t}(\level,\bm
  \res_{\kk}))\}^{1/\alpha} ]\lVert^\alpha + o(1)\\\\
  &=& \Big(\sum_{i=t-k}^{t-1} B_{i}(\level,\res_{i})^{1/\alpha} + [b\{\bm
  B_{t}(\level,\bm \res_{\kk})\}]^{1/\alpha}\Big)^\alpha - \Big(
  \sum_{i=t-k}^{t-1} B_{i}(\level,\res_{i})^{1/\alpha}\Big)^\alpha + o(1)\\
  &=& \Big(\sum_{i=t-k}^{t-1}
  B_{i}(\level,\res_{i})^{1/\alpha}\Big)^{\alpha}\Big( 1+ \frac{[b\{\bm
    B_{t}(\level,\bm \res_{\kk})\}]^{1/\alpha}} { \sum_{i=t-k}^{t-1}
    B_{i}(\level,\res_{i})^{1/\alpha}}\Big)^\alpha - \Big(
  \sum_{i=t-k}^{t-1} B_{i}(\level,\res_{i})^{1/\alpha}\Big)^\alpha + o(1)\\
  & = & \alpha\, \Big(\sum_{i=t-k}^{t-1}
  B_{i}(\level,\res_{i})^{1/\alpha}\Big)^{-(1-\alpha)}\, [b\{\bm
  B_{t}(\level, \bm \res_{\kk})\}]^{1/\alpha} + o(1).
\end{IEEEeqnarray*}
This expression converges to a positive constant provided
$b(\bm B(\level,\bm \res_{\kk})) = \mathcal{O}\big[\big\{
\sum_{i=0}^{k-1}
B_{t}(\level,\res_{i})^{1/\alpha}\big\}^{\alpha(1-\alpha)}\big]$ as
$\level \to \infty$.\ Hence, choosing $b$ equal to
$b(\bm y) = \lVert \bm y^{1/\alpha}\lVert^{\alpha (1-\alpha)}$,
$\bm y \in \RR_+^{k}$, gives the conditional probability
\eqref{eq:ims_conditional} converging to $p=1-\exp(-\alpha)$, that is
\[
  \pi^{\text{inv}} [\bm B_{t}(\level,\bm \res_{\kk}), b\{\bm
  B(\level,\bm \res_{\kk})\}]\to 1-\exp(-\alpha), \qquad
  \alpha\in(0,1),
\]
and generally, we also have that for any $x_k\in \RR_+$,
\begin{equation}
  \lim_{\level\to \infty}\pi^{\text{inv}}(\bm B(\level,\bm \res_{\kk}),
  b(\bm B(\level,\bm \res_{\kk}))\,\res_k) = 1- \exp(-\alpha x_k^{1/\alpha}).
  \label{eq:inv_log_conv}
\end{equation}

Lastly, we note that the convergence in \eqref{eq:inv_log_conv} holds
uniformly on compact sets in the variable
$\bm x_{\kk}\in [\delta_1,\infty)\times\cdots\times [\delta_k,
\infty)$ by continuous convergence \citep[see Section 0.1
in][]{resn87}. That is, expression \eqref{eq:inv_log_conv} holds true
after replacing $\bm \res_{\kk}$ by $\bm \res_{\kk}(\level)$
satisfying
$\bm \res_{\kk}(\level) \to \bm \res_{\kk}\in [\delta_1,\infty)\times
\cdots\times [\delta_k, \infty)$ as $\level\to \infty$ and since the
limit function is continuous in $\bm \res_{\kk}$ (constant function),
the argument follows.\

Let $\beta_t$ satisfy the recurrence relation
$\log \beta_t = \log (1-\alpha) +
\log\left(\max_{i=1,\ldots,k}\beta_{t-i}\right)$ subject to
$\beta_{i} = 1-\alpha$ for $i=1,\dots,k-1$.\ For all
$\delta_1,\ldots,\delta_k > 0$ and
$\bm \res_{\kk} \in [\delta_1,\infty) \times\ldots \times [\delta_k,
\infty)$,
\[
  v^{-\beta_t}\, b( \bm B_{t}(v,\bm \res_v))\to \psi_{t}^b(\bm \res),
  \quad \text{whenever $\bm \res_v \to \bm \res$ as $v\to \infty$},
\]
where $\psi_t^\scale>0$ is continuous and has the same form as in
Corollary \ref{cor:SDE_AI_scale}.
Thus, Assumption $\BB$ holds for the special case $t=k$ with
$b_t(\level)=\level^{\beta_t}$.

Finally, observe that the entire argument after expression
\eqref{eq:ims_conditional} remains unchanged upon changing $t=k$ to
$t+1 = k+1$.\ The claim is proved through iteration.

\subsection{Convergence of max-stable full conditional distribution -
  no mass on boundary}
\label{sec:proof_ms}
Suppose that $a(\level\bm 1_k)\to \infty$ as $\level\to \infty$.\ Let
$\Pi_{k-1}$ denote the set of partitions of $[k-1]=\{0,\dots,k-1\}$.\
Then, for $\bm \res_{\kk} \in \RR^{k}$ and with some rearrangement,
$\Pr(X_k < a(\bm X_{\kk}) \mid \bm X_{\kk} = \bm A_t(\level,\bm
\res_{\kk}) )$ is equal to
\begin{IEEEeqnarray}{rCl}
  && \frac{1+ \sum_{p \in \Pi_{k-1}\sm [k-1] }(-1)^{\lvert p
      \lvert} \{\prod_{J \in p} V_J(\bm y_{\kplusonekplusone})\}/
    V_{\kk}(\bm y_{\kplusonekplusone})} {1+\sum_{p \in \Pi_{k-1}
      \sm [k-1]}(-1)^{\lvert p\lvert} \{\prod_{J \in p} V_J(\bm
    y_{\kk}, \infty)\}/V_{\kk}(\bm y_{\kk},
    \infty)}\times\frac{V_{\kk}(\bm y_{\kplusonekplusone})}
  {V_{\kk}(\bm y_{\kk}, \infty)}\nonumber\\\nonumber\\
  && \times \exp\left\{V(\bm y_{\kk}, \infty) - V(\bm
    y_{\kplusonekplusone})\right\}
  \label{eq:ms_kernel_ra}
\end{IEEEeqnarray}
where
$\bm y_{\kk} = -1/\log[1-\exp\{- \bm A_t(\level,\bm \res_\kk)\}]$ and
$y_k=-1/\log[1-\exp\{- a(\bm A_t(\level,\bm \res_{\kk}))\}]$.\ Since
$V_J$ is a $-(\vert J \vert+1)$-homogeneous function
\citep{coletawn91}, it follows that
\begin{IEEEeqnarray}{l}
  \frac{\prod_{J \in p}V_J(\bm y_{\kplusonekplusone})}{V_{[k-1]}(\bm
    y_{\kplusonekplusone})} = \mathcal{O}\left(\exp\{(1-\lvert p
    \rvert)\,\level \}\right),\qquad \frac{\prod_{J\in p} V_J(\bm
    y_{\kk}, \infty)}{V_{[k-1]}(\bm y_{\kk}, \infty)} =
  \mathcal{O}\left(\exp\{(1-\lvert p \rvert)\,\level
    \}\right),\label{eq:VderivO1}
\end{IEEEeqnarray}
as $\level\to \infty$.\ Because $\lvert p \rvert \geq 2$ for any
$p\in \Pi_{k-1} \sm [k-1]$, it follows that the first fraction
in expression~\eqref{eq:ms_kernel_ra} converges to unity as
$\level\to \infty$ whereas the homogeneity property of the exponent measure
$V$ also guarantees that the last term in
expression~\eqref{eq:ms_kernel_ra} converges to unity since
$V(\bm y_{\kk}, \infty) = \mathcal{O}\{\exp(-\level)\}$ and
$V(\bm y_{\kplusonekplusone}) = \mathcal{O}\{\exp(-\level)\}$.
This leads to
\begin{equation}
  \Pr(X_k < a(\bm X_{\kk})\mid \bm X_{\kk} =  
  \bm A_t(\level,\bm \res_{\kk})) = \frac{V_{[0\,:\,k-1]}(\bm y_{\kk}, y_k)} {V_{[0\,:\,k-1]}(\bm y_{\kk},
    \infty)}(1+o(1)),
  \label{eq:ms_asequiv}
\end{equation}
as $\level \to \infty$.\ Therefore, for any functional $a$ satisfying
property~\eqref{eq:ams_condition}, we have
\begin{IEEEeqnarray}{rCl}
  \lim_{\level\to \infty}\frac{V_{\kk}(\bm y_{\kk}, y_k)}
  {V_{\kk}(\bm y_{\kk}, \infty)}&=& \ddfrac{V_{\kk} [
    \exp(\bm \res_{\kk}), \exp\{a(\bm \res_{\kk})\} ]}{V_{\kk}
    [\exp\{\bm \res_{\kk},\infty\}]}.
  \label{eq:ms_convergence}
\end{IEEEeqnarray}
Similarly, we see that for $x_k \in \RR$, then
$\Pr(X_k < a(\bm X_{\kk}) + x_k \mid \bm X_{\kk} = \bm A_t(\level,\bm
\res_{\kk}))$ converges to $K$ given by expression
\eqref{eq:ms_convergence_example}.\ The convergence in
\eqref{eq:ms_convergence} holds uniformly on compact sets in the
variable $\bm \res_{\kk}$ by continuous convergence \citep[see Section
0.1 in][]{resn87}. That is, expression \eqref{eq:ms_convergence} holds
true after replacing $\bm \res_{\kk}$ by $\bm \res_{\kk}(\level)$
satisfying $\bm \res_{\kk}(\level) \to \bm \res_{\kk}$ as
$\level\to \infty$ and since the limit function is continuous in
$\bm \res_{\kk}$, the argument follows.\

\subsection{Convergence of logistic full conditional distribution}
\label{sec:logistic_convergence}
Under Assumption $\AB$, expression \eqref{eq:ms_asequiv} implies that
for all $\bm \res_{\kk} \in \RR^{k}$,
\[
  \Pr\left(X_k < a(\bm X_{\kk})\mid \bm X_{\kk}=\bm A_t(\level, \bm
    \res_{\kk})\right)= \left[1 + \frac{\exp\left\{-a(\bm
        y_{\kk})/\alpha\right\}}{\big\lVert \exp(-\bm y_{\kk}/\alpha)
      \big\lVert} \right]^{ \alpha-k} (1+o(1))
\]
as $\level\to\infty$, where
$\bm y_{\kk} = -1/\log[1-\exp\{- \bm A_t(\level, \bm \res_{\kk})\}]$.\
Making the choice of $a$ to be
$a(\bm \res_\kk) = -\alpha \log\,\left\{ \lVert \exp(-\bm
  \res_{\kk}/\alpha)\rVert\right\}$ we see that
\[
  \lim_{\level\to \infty} \Pr(X_k< a(\bm X_{\kk}) \mid \bm X_{\kk} = \bm
  A_\level(\bm \res_{\kk})) = 2^{\alpha - k} \in(0,1),
\]
and more generally, for any $x_k \in \RR$,
\[
  \lim_{\level\to \infty} \Pr(X_k< a(\bm X_{\kk}) + x_k \mid \bm X_{\kk} =
  \bm A_t(\level, \bm \res_{\kk})) = \{1+\exp(-x_k/\alpha)\}^{\alpha-k}.
\]
The limit distribution does not depend on $\bm \res_{\kk}$ since $a$
satisfies property~\eqref{eq:ams_condition}.
\subsection{Convergence of H\"{u}sler--Reiss full conditional
  distribution}
\label{sec:HR_convergence}
\citet[equation~(15)]{wadtawn14} and expression \eqref{eq:ms_asequiv}
imply that for all $\bm \res_{\kk} \in \RR^{k}$,
\begin{IEEEeqnarray*}{rCl}
  \Pr(X_k < a(\bm X_{\kk}) \mid \bm X_{\kk}=\bm A_\level(\bm \res_{\kk}))&=&
  \Phi\left[\tau^{-1}\{a(\bm A_\level(\bm \res_{\kk})) - \mu(\bm A_\level(\bm \res_{\kk}))\}\right] (1+o(1))
\end{IEEEeqnarray*}
as $\level\to\infty$, where $\Phi$ denotes the cumulative distribution
function of the standard normal distribution and
\smash{$\mu(\bm y_{\kk}) = - \tau (\bm K_{01}^\top \bm C \bm K_{10}
  \cdot \bm y_{\kk} + \bm K_{01}^\top \bm \Sigma^{-1}\bm
  1_{k+1}^\top/\bm 1_{k+1}^\top \bm q)$} where
$\tau^{-1}=\bm K_{01}^\top \bm C \bm K_{01}$,
$\bm C=(\bm \Sigma^{-1} - \bm q \bm q^\top/\bm 1_{k+1}^\top \bm q)$
is a $(k+1)\times (k+1)$ matrix of rank $k$,
$\bm q=\bm\Sigma^{-1}\,\bm 1_{k+1}$, and
\[
  \bm{K}_{10} = \left(
    \begin{matrix}
      \bm I_{k}\\
      \bm 0_{1,k}
    \end{matrix}
  \right),\quad
  \bm K_{01} = \left(
    \begin{matrix}
      \bm 0_{k,1}\\
      1
    \end{matrix}
  \right).
\]
Making the choice of $a$ to be
$a(\bm \res_{\kk}) = -\tau \bm K_{01}^\top\bm C \bm K_{10} \cdot \bm
\res_{\kk}$ we see that for any $x_k\in\RR$,
\[
  \lim_{\level\to \infty}\Pr(X_k < a(\bm X_{\kk})+x_k \mid \bm
  X_{\kk}=\bm A_t(\level, \bm \res_{\kk}))=\Phi[\{x_k+\tau (\bm
  K_{01}^\top \bm \Sigma^{-1}\bm 1_{k+1}^\top/\bm 1_{k+1}^\top \bm
  q)\}/\tau].
\]
The limit distribution does not depend on $\bm \res_{\kk}$ since $a$
satisfies property~\eqref{eq:ams_condition}.\ The latter follows from
the properties $\bm K_{10}\cdot \bm 1_{k} = \bm 1_{k+1} - \bm K_{01}$
and $\bm C \cdot \bm 1_{k+1} = \bm 0_{k+1, 1}$ which give
$-\tau \, \bm K_{01}^\top \bm C \bm K_{10} \cdot \bm 1_{k} = 1$.
\bibliographystyle{agsm}

\end{document}